\documentclass[reqno,12pt]{preprint}
\usepackage{geometry}
\usepackage{hyperref}

\usepackage[full]{textcomp}
\usepackage[osf]{newtxtext}

\usepackage[cal=boondoxo]{mathalfa}
\usepackage[dvipsnames]{xcolor}
\usepackage{forest}
\usepackage{amsmath, stmaryrd, amsfonts, amssymb, bbm, mathrsfs,amsbsy}
\usepackage{mathrsfs} 
\usepackage[all]{xy}
\usepackage[utf8]{inputenc} 
\usepackage{comment}
\usepackage{hyperref}
\usepackage{breakurl}
\usepackage{enumitem}
\usepackage{booktabs}
\usepackage{tikz}
\usetikzlibrary{cd}
\usepackage{dsfont}
\usepackage{mhequ} 
\usepackage{mhsymb} 
\usepackage{mhenvs} 
\usepackage{array}
\usepackage{wasysym}
\usepackage{centernot}
\usepackage{mathtools}

\setlist{noitemsep,nolistsep,leftmargin=1.7em}
\usetikzlibrary{snakes}
\usetikzlibrary{decorations}
\usetikzlibrary{decorations.markings}
\usetikzlibrary{positioning}
\usetikzlibrary{shapes}
\usetikzlibrary{external}

\providecommand{\figures}{false}
{ \ifthenelse{\equal{\figures}{false}} {#1}{\[ {\rm Figure \ missing !} \]} }{}

\newcommand{\Up}{\Upsilon}
\newcommand{\bt}{\boldsymbol{\tau}}
\newcommand{\bj}{\boldsymbol{j}}
\renewcommand{\N}{\mathbb{N}}
\renewcommand{\k}{\mathbf{k}}

\newcommand{\cA}{{\mathcal A}}
\newcommand{\cP}{{\mathcal P}}
\newcommand{\cF}{{\mathcal F}}
\newcommand{\cT}{{\mathcal T}}
\newcommand{\cR}{{\mathcal R}}

\def\cH{\mathcal{H}}
\def\X{\mathbb X}

\def\CC{\mathcal{C}}

\def\CT{\mathcal{T}}

\renewcommand{\R}{\mathbb{R}}

\def\d{\, {\mathrm d}}

\def\graf{\curvearrowright}

\newtheorem{assumption}[lemma]{Assumption}

\theorembodyfont{\rmfamily}
\newtheorem{example}[lemma]{Example}

\newfont{\indic}{bbmss12}
\def\un#1{\hbox{{\indic 1}$_{#1}$}}

\colorlet{symbols}{blue!90!black}
\colorlet{testcolor}{green!60!black}
\colorlet{connection}{red!30!black}

\def\drawx{\draw[-,solid] (-3pt,-3pt) -- (3pt,3pt);\draw[-,solid] (-3pt,3pt) -- (3pt,-3pt);}
\tikzset{
	root/.style={circle,fill=testcolor,inner sep=0pt, minimum size=2mm},
	dot/.style={circle,fill=black,draw=black, solid,inner sep=0pt,minimum size=1.2mm},
	square/.style={rectangle,fill=black,draw=black, solid,inner sep=0pt,minimum size=1mm},
	empty/.style={circle,fill=white,draw=white, solid,inner sep=0pt,minimum size=0.5mm},
	var/.style={circle,fill=black!10,draw=black,inner sep=0pt, minimum size=
		2mm},
	symb/.style={circle,fill=symbols,draw=symbols, solid,inner sep=0pt,minimum size=0.5mm},
	yy/.style={circle,fill=magenta!20,draw=black,inner sep=0pt,minimum size=0.8mm},
	>=stealth,
	dotred/.style={circle,fill=black!50,inner sep=0pt, minimum size=2mm},
	generic/.style={semithick,shorten >=1pt,shorten <=1pt},
	dist/.style={ultra thick,draw=testcolor,shorten >=1pt,shorten <=1pt},
	testfcn/.style={ultra thick,testcolor,shorten >=1pt,shorten <=1pt,<-},
	testfcnx/.style={ultra thick,testcolor,shorten >=1pt,shorten <=1pt,<-,
		postaction={decorate,decoration={markings,mark=at position 0.6 with {\drawx}}}},
	kprime/.style={semithick,shorten >=1pt,shorten <=1pt,densely dashed,->},
	kprimex/.style={semithick,shorten >=1pt,shorten <=1pt,densely dashed,->,
		postaction={decorate,decoration={markings,mark=at position 0.4 with {\drawx}}}},
	kernel/.style={semithick,shorten >=1pt,shorten <=1pt,->},
	multx/.style={shorten >=1pt,shorten <=1pt,
		postaction={decorate,decoration={markings,mark=at position 0.5 with {\drawx}}}},
	kernelx/.style={semithick,shorten >=1pt,shorten <=1pt,->,
		postaction={decorate,decoration={markings,mark=at position 0.4 with {\drawx}}}},
	kernel1/.style={->,semithick,shorten >=1pt,shorten <=1pt,postaction={decorate,decoration={markings,mark=at position 0.45 with {\draw[-] (0,-0.1) -- (0,0.1);}}}},
	kernel2/.style={->,semithick,shorten >=1pt,shorten <=1pt,postaction={decorate,decoration={markings,mark=at position 0.45 with {\draw[-] (0.05,-0.1) -- (0.05,0.1);\draw[-] (-0.05,-0.1) -- (-0.05,0.1);}}}},
	kernelBig/.style={semithick,shorten >=1pt,shorten <=1pt,decorate, decoration={zigzag,amplitude=1.5pt,segment length = 3pt,pre length=2pt,post length=2pt}},
	rho/.style={dotted,semithick,shorten >=1pt,shorten <=1pt},
	renorm/.style={shape=circle,fill=white,inner sep=1pt},
	labl/.style={shape=rectangle,fill=white,inner sep=1pt},
	xi/.style={circle,fill=symbols!10,draw=symbols,inner sep=0pt,minimum size=1.2mm},
	xix/.style={crosscircle,fill=symbols!10,draw=symbols,inner sep=0pt,minimum size=1.2mm},
	xib/.style={circle,fill=symbols!10,draw=symbols,inner sep=0pt,minimum size=1.6mm},
	xibx/.style={crosscircle,fill=symbols!10,draw=symbols,inner sep=0pt,minimum size=1.6mm},
	not/.style={circle,fill=symbols,draw=symbols,inner sep=0pt,minimum size=0.5mm},
	>=stealth,
}

\tikzset{ individus/.style={scale=0.40,draw,circle,thick,fill=black!10},
	individu/.style={scale=0.40,draw,circle,thick,fill=black!50},       } 
\makeatletter
\def\DeclareSymbol#1#2#3{\expandafter\gdef\csname MH@symb@#1\endcsname{\tikz[baseline=#2,scale=0.15,draw=symbols]{#3}}\expandafter\gdef\csname MH@symb@#1s\endcsname{\scalebox{0.7}{\tikz[baseline=#2,scale=0.15,draw=symbols]{#3}}}}
\def\<#1>{\csname MH@symb@#1\endcsname}
\makeatother

\def\un#1{\hbox{{\indic 1}$_{#1}$}}

\DeclareSymbol{0}{0}{\draw (-1,0.5) node[dot] {};}
\DeclareSymbol{1}{0}{\draw[white] (-.4,0) -- (.4,0); \draw (0,-0.3) node[dot] {} -- (0,1.7) node[dot] {};}

\DeclareSymbol{3}{0}{\draw (0,0) -- (0,2.2) node[dot] {}; \draw (-1.2,1.8) node[dot] {} -- (0,0) -- (1.2,1.8) node[dot] {};}

\DeclareSymbol{31}{-3}{\draw (0,0) -- (0,-1) -- (1,0) node[dot] {}; \draw (0,0) -- (0,1.2) node[dot] {}; \draw (-.7,1) node[dot] {} -- (0,0) -- (.7,1) node[dot] {};}

\DeclareSymbol{32}{-3}{\draw (0,0) -- (0,-1) -- (1,0) node[dot] {}; \draw (0,0) -- (0,1.2) node[dot] {}; \draw (-.7,1) node[dot] {} -- (0,0) -- (.7,1) node[dot] {};
	\draw (0,-1) -- (-1,0) node[dot] {};}

\DeclareSymbol{30}{-3}{\draw (0,0) -- (0,-1); \draw (0,0) -- (0,1.2) node[dot] {}; \draw (-.7,1) node[dot] {} -- (0,0) -- (.7,1) node[dot] {};}

\DeclareSymbol{30}{-3}{\draw (0,0.5) -- (0,-1.4); \draw (0,0.5) -- (0,2.2) node[dot] {}; \draw (-1.2,1.8) node[dot] {} -- (0,0) -- (1.2,1.8) node[dot] {};}

\DeclareSymbol{2}{0}{\draw (-0.9,1.8) node[dot] {} -- (0,0) node[dot] {}-- (0.9,1.8) node[dot] {};}
\DeclareSymbol{201}{0}{\draw (-0.9,1.8) node[dot] {} -- (0,0) node[dot] {}-- (0.9,1.8) node[dot] {}; \draw (-0.9,1.8) node[dot] {} -- (-0.9,3.4) node[dot] {};}
\DeclareSymbol{102}{0}{\draw (-0.9,1.8) node[dot] {} -- (0,0) node[dot] {}-- (0.9,1.8) node[dot] {}; \draw (0.9,1.8) node[dot] {} -- (0.9,3.4) node[dot] {};}

\DeclareSymbol{301}{0}{\draw (-0.9,1.8) node[dot] {} -- (0,0) node[dot] {}-- (0.9,1.8) node[dot] {}; \draw (-0.9,1.8) node[dot] {} -- (0,3.4) node[dot] {}; \draw (-0.9,1.8) node[dot] {} -- (-1.8,3.4) node[dot] {};}
\DeclareSymbol{103}{0}{\draw (-0.9,1.8) node[dot] {} -- (0,0) node[dot] {}-- (0.9,1.8) node[dot] {}; \draw (0.9,1.8) node[dot] {} -- (0,3.4) node[dot] {};\draw (0.9,1.8) node[dot] {} -- (1.8,3.4) node[dot] {};}

\DeclareSymbol{20}{0}{\draw (-0.9,1.8) node[dot] {} -- (0,0) node[dot] {} -- (0.9,1.8) node[dot] {};}

\DeclareSymbol{200}{0}{\draw (-0.9,1.8) node[dot] {} -- (0,0) node[dot] {} -- (0.9,1.8) node[dot] {}; \draw (0,0) -- (0,-1.5);}

\DeclareSymbol{11}{-3}{\draw (0,-1.4) node[dot] {} -- (0,0.2)  node[dot] {}; \draw (0,0.2) -- (0,2.1) node[dot] {};}

\DeclareSymbol{111}{-2}{\draw (0,0.5) node[dot] {} -- (0,2.5)  node[dot] {}; \draw (0,2.5) -- (0,4.5) node[dot] {}; \draw (0,4.5) -- (0,6) node[dot] {};}

\DeclareSymbol{110}{-3}{\draw[white] (-.4,0) -- (.4,0); \draw (0,-0.5)  node[dot] {}; \draw (0,-0.5) -- (0,1.5) node[dot] {};}

\DeclareSymbol{4}{-3}{\draw (0,-1.4) -- (0,0.2)  node[dot] {}; \draw (0,0.2) -- (0,2.1) node[dot] {};\draw (0,2.1) -- (0,4) node[dot] {};}

\DeclareSymbol{12}{0}{\draw (0,0)  node[dot] {} -- (0,1.7) node[dot] {}; \draw (-1.2,3) node[dot] {} -- (0,1.7) -- (1.2,3) node[dot] {};}

\DeclareSymbol{10}{-3}{\draw (0,0) -- (0,-1) {};  \draw (-.7,1) node[dot] {} -- (0,0);}

\DeclareMathAlphabet{\mathpzc}{OT1}{pzc}{m}{it}

\newcommand{\subsectionnotoc}[1]{%
	\refstepcounter{subsection}%
	\subsection*{\thesubsection\quad #1}%
}

\definecolor{AColor}{rgb}{0.0, 0.42, 0.24}
\definecolor{LAColor}{rgb}{0.64, 0.0, 0.0}

\begin{document}
	
	\title{Remainders of generalised Taylor expansions and a priori bounds for rough differential equations}
	\author{L. Agabiti$^a$, A. Bonicelli$^b$, L. Zambotti$^c$}
	
	\institute{ Laboratoire de Probabilit\'es Statistique et Mod\'elisation,\\ Sorbonne Université, Paris \\
		\email{$^a$agabiti@lpsm.paris,\, $^b$bonicelli@lpsm.paris,\, $^c$zambotti@lpsm.paris}
	}
	
	\maketitle
	
	\begin{abstract}
		In this article we establish global a priori estimates on the solution of a generic rough differential equation driven by an $\alpha$-H\"older path covering the full range of regularity $\alpha\in(0,1]$, under the hypothesis of Lipschitz continuity (but neither boundedness nor coercivity) of specific combinations of the coefficients and their derivatives, called elementary differentials. Our main tool is a closed formula for the remainder of a generalised Taylor expansion in terms of products of gradients of the elementary differentials, which allows to take advantage of cancellations in
an optimal way. 
	\end{abstract}


	
	\setcounter{tocdepth}{3}
	\tableofcontents

\section{Introduction}

In this article we study rough differential equations driven by branched rough paths, which naturally arise as extensions of controlled differential equations when the driving path is too irregular for classical theories to apply. Consider the simplest example of a controlled ODE
\begin{equation}\label{Eq: controlled ODE}
	\dot{Z}_t=\sum_{i=1}^d\sigma_i(Z_t)\, \dot{X_t^i}\,, \qquad  Z_0=z_0\in\R^n\,,\, t\in[0,T]\,,
\end{equation}
where $Z:[0,T]\to\R^n$, $X^i:[0,T]\to\R$ are differentiable paths for $i=1,\ldots,d$, and $\sigma_i:\R^n\to\R^n$ are sufficiently smooth non-linear coefficients. In this case \eqref{Eq: controlled ODE} admits the integral form
		\begin{equation}\label{integral form} 
		Z_t = z_0 + \int_0^t \sigma (Z_s) \, \dot{X_s} \d s\,, \qquad t \in [0, T]\,.
		\end{equation}

Lyons' theory of rough paths \cite{lyons1998differential} aims at finding an extension of the solution map $C^1\ni X\to Z$ to Hölder-continuous $X\in C^\alpha$,  $\alpha\in(0,1]$, allowing to cover the case of Itô stochastic differential equations
\begin{equation*}
	\d {Z}_t=\sum_{i=1}^d\sigma_i(Z_t) \d{B_t^i}\,, \qquad  Z_0=z_0\in\R^n\,,\qquad t\in[0,T]\,,
\end{equation*}
where $B^i$ are Brownian motions.

We focus our attention on Davie's notion of solutions to \eqref{Eq: controlled ODE}, first introduced in \cite{davie2008differential},
 in the context of branched $\alpha$-rough paths \cite{gubinelli2010ramification}
for any $\alpha\in(0,1]$. 
This introduction is intended to provide an intuitive overview of the problems and ideas that motivate this manuscript, while deferring precise definitions to later sections.

First, a branched rough path is a set of 
real-valued functions $(\X_{st}(\tau))_{0\le s\le t\le T}$, labelled by a combinatorial index $\tau$ which runs in a finite set of \emph{decorated trees}. These functions satisfy suitable algebraic and analytical constraints and are related to the path $X=(X^1,
\ldots,X^d)$ by identifying $\X_{st}(\bullet_i)=X^i_t-X^i_s$, while the exponent $\alpha$ refers to the Hölder continuity of $X$, see Definition \ref{Def: branched rough paths} below.

A Davie solution to \eqref{Eq: controlled ODE} is a path $Z:[0,T]\to\R^n$ that satisfies $Z_0=z_0$ and
\begin{equation}\label{eq:davie}
	Z_{t}-Z_s-\sum_{\tau\in\cT^{\le N}}\Up_\tau(Z_s) \, \X_{st}(\tau)=o(t-s)
\end{equation}
uniformly for $0\le s\le t\le T$, 
where $N:=\lfloor\frac1\alpha\rfloor$ and $\cT^{\le N}$ is the set of decorated rooted trees with at most $N$ nodes,  see Definition \ref{def:aladavie} below. Trading terminology from numerical analysis, the coefficients $\Up_\tau:\R^n\to\R^n$
are called \emph{elementary differentials} and are defined recursively in terms of $\sigma=(\sigma_i)_{i=1,\ldots,d}$, see Definition \ref{Def: elementary differential trees} below. Examples of such elementary differentials are
\[
\Upsilon_\bullet = \sigma, \qquad \Upsilon_{\<1>}= \nabla\sigma \cdot \sigma, \qquad 
\Upsilon_{\<2>}= \nabla^2\sigma \cdot \sigma \cdot \sigma.
\]
The family $(\X_{st}(\tau))_{\tau\in\cT^{\leq N}}$ plays the role of a set of \emph{generalised monomials} in the local expansion of the solution,
while $(\Upsilon_\tau(Z_s))_{\tau\in\cT^{\leq N}}$ should be interpreted as \emph{generalised derivatives} of the solution with respect to $(\X_{st}(\tau))_{\tau\in\cT^{\leq N}}$.

The theory of rough paths gives mainly local in time results, and global well-posedness results are often derived under the restrictive assumption 
of boundedness of the coefficients of the generalised Taylor expansion of the solution,
while the case of linear coefficients is treated separately, see e.g. \cite{lyons1998differential}, \cite[Thm: 8.4]{friz2020course} and \cite[Def. 3.7]{bailleul2015flows}. 

On the other hand, in the case of a rough equation driven by a 
geometric rough path over a path of $p$-bounded variation, $p\in[2,3)$, global existence under more general assumptions had already been obtained 
long before, see e.g. \cite{lejay2009rough,davie2008differential}. Other recent papers \cite{bonnefoi2022priori,chevyrev2025large} rely on coercivity 
assumptions, namely on the existence of a strong drift towards the origin of $\R^n$, that has the effect of confining the solution in a compact set. 

Since the formulation \eqref{eq:davie} relies only on the set of coefficients $(\Up_\tau)_{\tau\in\cT^{\le N}}$ for the definition of a solution, it would be natural
to expect that a solution theory could be given by assuming suitable analytical properties of such coefficients. By analogy with ordinary differential equations, Lipschitz continuity appears to be the natural assumption, while boundedness should not be necessary. Indeed, as observed by Lejay in \cite{lejay2012global}, for rough equations driven by the geometric rough path lift of a path with bounded $p$-variation, $p\in[2,3)$, a natural hypothesis on the coefficient $\sigma$ is $\nabla\sigma$ bounded and $\nabla\sigma\cdot\sigma$ H\"older continuous with exponent $\gamma\in(p-2,1]$. In particular, no H\"older continuity of $\nabla\sigma$ is required, in contrast with \cite{lejay2009rough, davie2008differential, friz2008euler}. 

Building on this key observation, the primary objective of this article is to establish strong a priori estimates for Davie solutions over the full range of Hölder regularity $\alpha \in (0,1)$ of the driving path $X$ under minimal assumptions on the elementary differentials, with the ultimate goal of developing a global solution theory. 

The aim of this article is to prove that, just assuming Lipschitz continuity (and neither coercivity nor boundedness) for all the coefficients that appear in \eqref{eq:davie}, one can establish a priori estimates on the solution and, ultimately, a complete solution theory. To achieve our goal, we introduce so-called solution remainders
\begin{equation}\label{Eq: Z remainders0}
	Z_{st}^{[m+1]}:=Z_{t}-Z_s-\sum_{\tau\in\cT^{\le m}}\Up_\tau(Z_s) \, \X_{st}(\tau)\,, \qquad m=0,\ldots,N\,,
\end{equation}
and find suitable analytic bounds on these quantities. 

In the following, for $\zeta>0$ we denote by $\|\cdot\|_\zeta$ certain H\"older-like norms on continuous functions defined on simplices, whose precise definition is postponed to Section \ref{Sec: notation}. For $\mu>0$, we will be also working with suitable weighted versions $\|\cdot\|_{\zeta,\mu}$ of these norms, see Definition \ref{Def: weighted norms}.
Our main result on Davie solutions is the following
\begin{theorem}\label{th: a priori branched0}
	Assume that $\Upsilon_\tau$ is globally Lipschitz continuous for all $\tau\in\cT^{\leq N}$. Then, for sufficiently small $\varepsilon:=T\wedge\mu$ and $Z$ any solution of \eqref{eq:davie}, 
	\begin{align*}
		\| Z^{[N+1]} \|_{(N+1)\alpha, \mu} &\lesssim_{\alpha, \mathbb{X}, \sigma} \sum_{i=1}^N \|Z^{[i]}\|_{i\alpha,\mu} 
		 \le 2\sum_{\tau\in\cT^{\le N}} |\tau| \,|\Up_{\tau}(Z_0)| \, \|\mathbb{X}(\tau)\|_{|\tau|\alpha}\,.
	\end{align*}
\end{theorem}
For a Davie solution, the main remainder $Z^{[N+1]}_{st}$ is merely supposed to be $o(t-s)$. 
Theorem \ref{th: a priori branched0} gives an a priori quantitative control over $Z^{[N+1]}$ in terms of the data of the problem:
$\X,\Upsilon,z_0$ (see Theorem \ref{th: a priori branched} for a more precise statement). Moreover it shows that $Z^{[N+1]}$ can be controlled by the lower solution remainders
$(Z^{[i]})_{i=1,\ldots,N}$. As a result one can obtain 
a full well-posedness theory for Davie's solutions and continuity of the map $(z_0,\X)\mapsto Z$ in the appropriate topologies. 
We postpone the proof of the latter results to a forthcoming companion paper. 

In Theorem \ref{th: a priori branched} we only need to assume $\Upsilon_\tau$ is globally Lipschitz continuous for all $\tau\in\cT^{\leq N-1}$
and $\Upsilon_\tau$ is globally Hölder continuous for all $\tau\in\cT^{= N}$, with an exponent $\beta\in(\frac{1}{\alpha}-N,1]$; in this case we
have a bound rather for $\| Z^{[N+1]} \|_{(N+\beta)\alpha, \mu}$, see \eqref{eq: a priori branched}.
In Proposition \ref{lem:bounded} below we also reprove a priori estimates
for bounded coefficients with their derivatives.

\subsectionnotoc{Analytic bounds}


A crucial tool in this setting is the celebrated \emph{Sewing Lemma} \cite{gubinelli2004controlling,FeDe06}, from which we extract
the following \emph{Sewing bound}, see \cite[Thm. 1.9]{CGZ25}. Given a function $R$ on the simplex $\{(s,t)\in[0,T]^2\,:\;s\le t\}$, we define for $s\le u\le t$ the three-point increment $\delta R_{sut}:=R_{st}-R_{su}-R_{ut}$. 
\begin{theorem}[Sewing bound]
	Consider a function $(R_{st})_{s<t\in[0,T]}$ such that $R_{st}=o(t-s)$. Then,  for any $\eta\in(1,\infty)$ the following bound holds
	\begin{equation*}
		\|R\|_{\eta}\leq K_\eta\,\|\delta R\|_\eta\,,\qquad K_\eta:=(1-2^{1-\eta})^{-1}\,.
	\end{equation*}
\end{theorem}

Observe that, by definition, for a Davie solution $Z$ we have $Z^{[N+1]}_{st}=o(t-s)$. Then the Sewing bound allows to control
$\| Z^{[N+1]} \|_{(N+1)\alpha}$ with the H\"older-like seminorm of $ \delta Z^{[N+1]}$, which turns out to have a simpler structure. 
In particular, it is well known \cite[Lemma 8.3]{gubinelli2010ramification} that
	\begin{align*}
		\delta Z_{s u t}^{[N+ 1]}
		= \sum_{\tau\in\CT^{\leq N}} B_{s u}^{[N-|\tau|+1]}(\tau) 
		\, \mathbb{X}_{u t}(\tau)\,,
	\end{align*}
where 
\begin{equation*}
		B^{[m+1]}_{st}(\tau) := 
		\Up_\tau(Z_t) - \sum_{w\in\CF^{\leq m}}\frac{\Up_{w\curvearrowright\tau}(Z_s)}{\pi(w)}  \, \mathbb{X}_{s t}(w)\,,
\end{equation*}
that we interpret as remainders of generalised Taylor expansions of $\Upsilon_\tau(Z_t)$ 
with respect to the family of (generalised) monomials $(\X_{st}(\tau))_{\tau\in\cT^{\le N}}$. This calls into play the notion of a \emph{forest}, namely a disjoint union of finitely many trees. We denote by $\mathcal{F}^{\le m}$ the set of forests with at most $m$ nodes. A key role in our analysis is played by the grafting operation $\graf$ on forests; see Definition~\ref{Def: grafting} and Appendix~\ref{App: grafting}. Finally, $\pi(w)$ denotes a combinatorial factor.

In order to derive global a priori estimates from this formula, we need a representation of $B^{[m+1]}(\tau)$ that is bounded if $\Upsilon_\tau$ is globally Lipschitz continuous for all $\tau\in\cT^{\leq N}$. The following theorem addresses this problem.

\begin{theorem}\label{Thm: metatheorem remainder}
	Assume that $\Upsilon_\tau$ is globally Lipschitz continuous for all $\tau\in\cT^{\leq N}$. Let $Z$ be a Davie solution to \eqref{Eq: controlled ODE}.
Then for all $\tau_0\in\cT^{\le N}$ there exists a function $H_{st}(\cdot;\tau_0):\cF^{\le N-|\tau_0|} \to\R^n\otimes(\R^n)^\ast$ such that for $m\le N-|\tau_0|$\begin{align*}
		B^{[m+1]}_{st}(\tau_0)=\sum_{w\in\cF^{\leq m}}H_{st}(w;\tau_0) \cdot Z_{st}^{[m+1-|w|]}\, \X_{st}(w)\,, 
	\end{align*}
	where the $Z^{[i]}$'s are the solution remainders defined in \eqref{Eq: Z remainders0} and
		\[
	|H_{st}(w;\tau_0)| \lesssim \prod_{\tau\in\cT^{\le |w|+|\tau_0|}}\|\nabla\Upsilon_{\tau}\|_\infty\,.
	\]
\end{theorem}

\subsectionnotoc{Remainders of generalised Taylor expansions}\label{Sec: remainders}

Theorem \ref{Thm: metatheorem remainder} crucially relies on a representation of the coefficients $H_{st}(w;\tau_0)$ that only contains gradients of $\Upsilon_\tau$ for $\tau\in\cT^{\le |w|+|\tau_0|}$, see Definition \ref{Def:H} and Theorem \ref{Thm: Taylor Lipschitz}. In fact these results are stated 
in terms of remainders of generalised Taylor expansions without reference to rough paths
and differential equations, as in the next Theorem. For this we need the following ingredients:
	\begin{itemize}
		\item two points $z_1,z_2\in\R^n$,
		\item a multiplicative functional $\X:\cF^{\le m}\to\R$, namely $\X(w_1w_2)=\X(w_1)\X(w_2)$ where $w_1w_2=w_1\sqcup w_2$ is the juxtaposition of forests, $w_1,w_2\in\cF^{\le m}$,
		\item the abstract remainders for $m\ge 0$
		\begin{align}\label{eq:zk12}
			z^{[m+1]}& :=z_2-z_1 -\sum_{\tau\in\cT^{\le m}}\Upsilon_\tau(z_1) \, \X(\tau)\in\R^n\,,
			\\ B^{[m+1]}_{z_1z_2}(\tau)& :=\Up_\tau(z_2)-\sum_{w\in\cF^{\le m}}\frac{\Up_{w\curvearrowright\tau}(z_1)}{\pi(w)} \, \X(w) \in\R^n\,.
			\label{eq:B12}
		\end{align}
	\end{itemize}
Then we can state the 
\begin{theorem}\label{Thm: metatheorem remainder2}
	Assume that $\Upsilon_\tau$ is globally Lipschitz continuous for all $\tau\in\cT^{\leq N}$. 
	Then for all $\tau_0\in\cT^{\le N}$ there exists a function $H_{z_1z_2}(\cdot;\tau_0):\cF^{\le N-|\tau_0|} \to\R^n\otimes(\R^n)^\ast$ such that
	for $m\le N-|\tau_0|$ 
	\begin{align}
		B^{[m+1]}_{z_1z_2}(\tau_0)=\sum_{w\in\cF^{\leq m}}H_{z_1z_2}(w;\tau_0) \cdot z^{[m+1-|w|]}\, \X(w)\,, \label{Eq: identity remainder H0}
	\end{align}
	where 
		\[
	|H_{z_1z_2}(w;\tau_0)| \lesssim \prod_{\tau\in\cT^{\le |w|+|\tau_0|}}\|\nabla\Upsilon_{\tau}\|_\infty\,.
	\]
\end{theorem}	

Even though our motivation is primarily analytic, namely to obtain a priori estimates under minimal regularity assumptions on the coefficients of the rough equation \eqref{eq:davie}, we emphasize that Theorem \ref{Thm: metatheorem remainder2} is entirely combinatorial and algebraic, see Theorem \ref{Thm: Taylor Lipschitz} and its proof. This provides yet another illustration of the deep interplay between algebraic and analytic methods in rough paths theory. 

We derive representation formulae of $H_{z_1z_2}(w;\tau_0)$ that involve only
$(\nabla\Upsilon_\tau)_{\tau\in \cT^{\le |w|+|\tau_0|}}$. This turns out to be surprisingly difficult: merely stating such formulae requires introducing several nontrivial combinatorial objects, and the final result is Theorem \ref{Thm: Taylor Lipschitz} below. This problem is treated in Section \ref{Sec: Lipschitz elementary differentials}, the longest of the paper, which contains the conceptual core of our argument. See \eqref{Eq: B^2}-\eqref{Eq: B^3} for two simple examples that
are explained in detail.

We give \emph{two} novel representation formulae for the coefficients $H_{z_1z_2}(w;\tau_0)$ that appear in \eqref{Eq: identity remainder H0}. The first expression \eqref{Eq: identity remainder H1} is based on Theorem \ref{Thm: Taylor a priori} and requires boundedness of $\Upsilon_\tau$ and its derivatives. This could be seen as an alternative to the generalised Taylor expansion derived in \cite[Proposition A.1]{hairer2014theory} and used in the context of rough paths in \cite{bonnefoi2022priori}. The second, more complex representation formula 
\eqref{Eq: identity remainder H} allows to obtain a priori bounds for Davie solutions when only Lipschitz continuity of $\Upsilon_\tau$ is assumed. The main tools that we use to prove these representation formulae are combinatorial and algebraic in nature. In particular, we must study in detail the Guin-Oudom extension of grafting \cite{oudom2008lie}, for which we establish a representation formula for $(\tau_1\cdots\tau_k)\graf\tau_0$ as a sum indexed by \emph{planar} binary rooted
trees, see Lemma \ref{lem: GUBinary} and Corollary \ref{Cor: C and J}. Proving this result requires introducing several non-trivial objects, such as
reduction chains (Definition \ref{Def: reduction chain}) and a family of multi-indices $J_{\rho,\ell}$ (Definition \ref{Def: J}) associated with a binary tree $\rho$.
	
\subsectionnotoc{Comparison with other works}
	In \cite{bonnefoi2022priori}, the authors pursue a similar objective, namely the derivation of a priori estimates for solutions to rough differential equations that are independent of the initial condition. They consider two distinct classes of equations: those whose coefficients have bounded derivatives and those whose derivatives exhibit polynomial growth but contain a coercive drift term. Coercivity is the key ingredient that controls the polynomial growth of the combinations of derivatives of the coefficients appearing in the remainder terms of the expansion, both for the solution itself and for its composition with smooth functions. 
	
	A similar approach is adopted in \cite{chevyrev2025large}, where the authors show how the coercive nature of a damping term can be exploited to obtain a priori estimates for solutions, both in the setting of rough differential equations and for singular stochastic PDEs within the framework of regularity structures. In the rough ODE setting, they likewise rely on a generalised Taylor formula for the remainder, which coincides with the one implicitly used in \cite{bonnefoi2022priori}. 
	
	The problem of finding optimal conditions on the coefficient vector field to ensure global in time existence of solutions to rough equations has also been addressed in \cite{bailleul2020non}, where the authors work in the setting of solution flows, see \cite{bailleul2015flows}. Working with weak geometric rough paths, they cover the case of linear growth of the coefficients, yet assuming boundedness of their derivatives. Here we go beyond this regime by only assuming Lipschitz continuity of the 
	elementary differentials that appear in \eqref{eq:davie}.
	
	During the preparation of this manuscript, another related work was published \cite{ying2026non}, in which the problem of deriving a priori bounds for solutions in order to establish the absence of finite-time explosion is investigated. Once again exploiting the separation between drift and diffusion, the author extends the results of \cite{li2025strong} to cover the full range of regularity of the driving signal. The superlinear polynomial growth of the drift is controlled by imposing a suitable sublinear growth condition on the diffusion coefficient and its derivatives.
	
	We stress that a priori estimates for stochastic/rough PDEs have become a major theme recently, see \cite{Mourrat_2017,mourrat2017global,ottoweber1,ottoweber2,otto2024priori,chevyrev2025large}.
	We are currently working on the extension of the techniques of this paper to the more difficult setting of regularity structures and models for SPDEs \cite{hairer2014theory}. In this setting the algebraic framework is more complex \cite{bruned2019algebraic} and, as recently observed, it involves post-Lie structure rather than the pre-Lie setting which is used in this paper for branched rough paths, see \emph{e.g.}
	\cite{kuruschpostLie,linares2023structure,BK23, BM23, JZ26}. 
	
	\subsectionnotoc{Structure of the paper}
	In section \ref{Sec: branched rough paths} we briefly recall Gubinelli's definition of a branched rough path and related notions needed to formulate the Ansatz on the expansion of the solution. We adopt a framework that allows for minimal algebraic and combinatorial machinery. For instance, we prefer to work with a version of grafting that reduces the appearance of combinatorial factors in formulae like \eqref{eq:davie}, see Appendix \ref{App:A} for a discussion. We also avoid the introduction of coproducts.

Section \ref{Sec: generalised Taylor formulae} is devoted to a closed form of the elementary differential remainders well suited for the derivation of a priori estimates under a boundedness assumption on the elementary differentials. Then, in Section \ref{Sec: Lipschitz elementary differentials} we study the problem at the heart of this article, first finding an explicit expression for the grafting of forests in terms of binary trees and then using the latter to describe the algorithm underlying the construction of the generalised Taylor formula for the remainder. In particular, in Section \ref{Sec: Algebraic formula for the remainder} we state and prove a precise version of Theorem \ref{Thm: metatheorem remainder}. Finally, in Section \ref{Sec: well-posedness} we establish global a priori estimates on the solution remainders.

	Appendix \ref{App: grafting} collects remarks on our choice of grafting operation in connection with other prescriptions, as well as a simple proof of the morphism property of elementary differentials. In Appendix \ref{App: smooth rough paths} we make a number of observations on the canonical rough path lift of differentiable paths in connection with the notion of Davie solution, while Appendix \ref{App: analytic tools} summarises useful inequalities and technical lemmas used in the main body.

	\subsectionnotoc{Notation}\label{Sec: notation}
	
	We set notation and conventions that will be adopted throughout the article. 
	For $k\geq 1$ and $T\geq0$ we define the simplices
	\begin{align*}
		&[0, T]^{k+1}_{\leq} :=
		\left\{ (t_0, \ldots, t_k) \in[0,T]^{k+1} :  0 \leq t_0 \leq \ldots \leq
		t_k \leq T \right\}\,,\\
		&[0, T]^{k+1}_{\geq} :=
		\left\{ (t_0, \ldots, t_k) \in[0,T]^{k+1} :  0 \leq t_k \leq \ldots \leq
		t_0 \leq T \right\}\,.
	\end{align*}
	We  will also adopt the shorthand notation $\d \mathbf{t}_k:=\d t_k\ldots\d t_0$ to denote the corresponding Lebesgue measure. 
	
	We introduce a class of norms on continuous functions defined on simplices. Given $\eta
	\in (0, \infty)$, we define for $F \in C ([0, T]^2_{\leq}, \mathbb{R}^n)$
	and $G \in C ([0, T]^3_{\leq}, \mathbb{R}^n)$ 
	\[ \| F \|_{\eta} :=\sup_{0 \leq s < t \leq T} \frac{| F_{s t}
		|}{(t - s)^{\eta}}\,, \qquad \| G \|_{\eta} :=
	\sup_{\substack{0 \leq s \leq u \leq t \leq T\\
			s < t}} \frac{| G_{s u t} |}{(t - s)^{\eta}}\,, \]
	and the spaces
	\[ \mathcal{C}^{\eta}_2 :=\left\{ F \in C ([0, T]^2_{\leq}, \mathbb{R}^n) :
	\| F \|_{\eta} < \infty \right\}\,, \qquad \mathcal{C}^{\eta}_3 :=\left\{ G
	\in C ([0, T]^3_{\leq}, \mathbb{R}^n) : \| G \|_{\eta} < \infty
	\right\} . \]
	We endow the space $C ([0, T], \mathbb{R}^n)$ of continuous functions $f
	: [0, T] \rightarrow \mathbb{R}^n$ with the usual H{\"o}lder semi-norm.
	Defining the increments $\delta f_{s t} := f_t - f_s$ for $s,t\in[0,T]_{\leq}^2$, we denote it by
	\begin{equation*}
		[f]_\alpha:=\|\delta f\|_\alpha=\sup_{0 \leq s < t \leq T}
		\frac{| f_t - f_s |}{(t - s)^{\alpha}}\,,\qquad \alpha \in (0, 1]\,,
	\end{equation*}
	and the corresponding space of $\alpha$-H\"older functions
	\[ \mathcal{C}^{\alpha} := \left\{ f : [0, T] \rightarrow \mathbb{R}^n :
	\quad [ \delta f]_{\alpha} < \infty \right\} \,. \]
	We also consider the operator $\delta:C ([0, T]^2_{\leq}, \mathbb{R}^n)\to C ([0, T]^3_{\leq}, \mathbb{R}^n)$,
	\[
	\delta A_{sut}:=A_{st}-A_{su}-A_{ut}, \qquad 0\le s\le u\le t\le T.
	\]
	Henceforth, we will denote by $d,n\in\N_\ast$ the dimension of Euclidean spaces.
	Throughout the paper, the symbol $A\lesssim_a B$ will denote $A\leq c_a B$ for a constant $c_a\in\R$ which depends on the datum $a$.

	\subsection*{Acknowledgements}
	L.A. has received funding from the European Union’s Horizon Europe research and innovation programme under the Marie Sk\l{}odowska-Curie grant agreement No 101126554.
	A.B. is supported by a Postdoctoral fellowship from the Fondation Sciences Mathématiques de Paris, which is gratefully acknowledged. The work of L.Z. was partially supported by the ANR project 
	RAN-DOP ANR-24-CE40-3377, and by the Institut Universitaire de France.

	\section{Branched rough paths and elementary differentials}\label{Sec: branched rough paths}
	Since the seminal work of Cayley \cite{cayley1857xxviii}, the close correspondence between families of combinatorial trees and vector fields has been well established. To illustrate this correspondence, consider once again the case of a smooth driving path. An expansion of the solution
	to \eqref{integral form} up to order $N=3$ yields (with Einstein's sommation convention)
	\begin{align}\label{Eq: approximation smmoth}
		Z^i_t&=Z^i_s+\sigma^i_{j_1}(Z_s)\int_s^t \d X^{j_1}_{t_0}+\nabla_{j_1}\sigma^i_{ j_2}(Z_s) \,\sigma^{j_1}_{j_3}(Z_s)\int_s^t\int_s^{t_0}\d X^{j_3}_{t_1}\d X^{j_2}_{t_0}\nonumber\\
		&+\nabla_{j_1}\sigma^i_{j_2}(Z_s)\,\nabla_{j_3}\sigma^{j_1}_{j_4}(Z_s)\, \sigma^{j_3}_{j_5}(Z_s)\int_s^t\int_s^{t_0}\int_s^{t_1}\d X^{j_5}_{t_2}\d X^{j_4}_{t_1}\d X^{j_2}_{t_0}\nonumber\\
		&+\frac{1}{2}\nabla^2_{j_1j_2}\sigma^{i}_{j_3}(Z_s)\,\sigma_{j_4}^{j_2}(Z_s)\,\sigma_{j_5}^{j_3}(Z_s)\int_s^t\left(\int_s^{t_0}\d X^{j_5}_{t_1}\right)\left(\int_s^{t_0}\d X^{j_4}_{t_2}\right)\d X^{j_3}_{t_0}+O((t-s)^4)\,.
	\end{align}
	See Proposition \ref{pr:daviesmooth} below for a rigorous proof of this expansion at any order. 
	
	The nested nature of the iterated integrals of the driver path $X$ in \eqref{Eq: approximation smmoth} can be encoded via the hierarchical structure of rooted trees. Interpreting the vertices as integrations against the rough drivers, the natural correspondence between nested iterated integrals and rooted trees for the lower orders reads  
	\begin{align*}
		&\int_s^t\d X^{j_1}_{t_0}=:\X_{st}(\bullet_{j_1}),
		\qquad \int_s^t\int_s^{t_0}\!\int_s^{t_1}\d X^{j_5}_{t_2}\d X^{j_4}_{t_1}\d X^{j_2}_{t_0}=:\X_{st}\left(
		\!\!\begin{tikzpicture}[scale=0.2,baseline=0.1cm]
			\node at (0,-0.2) [dot, label={[label distance=-2pt]right:{\tiny $j_2$}}] (root2) {};
			\node at (0,1) [dot, label=left:{\tiny $j_4$}] (centerl) {};
			\draw (centerl) to
			node [sloped,below] {\small }     (root2);
			\node at (0,2.2) [dot, label={[label distance=-2pt]right:{\tiny $j_5$}}] (centerr) {};
			\draw (centerr) to
			node [sloped,below] {\small }     (centerl);
		\end{tikzpicture}\right)\,,\nonumber\\ 
		&\frac{1}{2}\int_s^t\left(\int_s^{t_0}\d X^{j_5}_{t_1}\right)\left(\int_s^{t_0}\d X^{j_4}_{t_2}\right)\d X^{j_1}_{t_0}=:\X_{st}(
		\begin{tikzpicture}[scale=0.2,baseline=0.1cm]
			\node at (0,0.5) [dot, label={[label distance=-2pt]right:{\tiny $j_1$}}] (root2) {};
			\node at (-0.7,1.7) [dot, label={[label distance=-2pt]left:{\tiny $j_5$}}] (centerl) {};
			\node at (0.7,1.7) [dot, label={[label distance=-2pt]right:{\tiny $j_4$}}] (centerr) {};
			\draw (centerl) to
			node [sloped,below] {\small }     (root2);
			\draw (centerr) to
			node [sloped,below] {\small }     (root2);
		\end{tikzpicture}
		)\,,\label{Eq: example iterated integral}
	\end{align*}
	where, in the case of differentiable $X$, all integrals are well-defined. This section is devoted to realising this correspondence when $X$ is not smooth. First we introduce the set of combinatorial trees used to label the expansion coefficients of the Davie solution. 
	\begin{definition}\label{Def: rooted trees}
		We call \textit{decorated rooted tree} a finite connected acyclic graph $\tau=(\mathcal{E}_\tau,\mathcal{V}_\tau,\mathfrak{n})$ with edge set $\mathcal{E}_\tau$ and vertex set $\mathcal{V}_\tau$, such that
		\begin{itemize}
			\item $\tau$ has a preferred vertex, called the \emph{root},
			\item each pair of vertices is connected by at most one edge, 
			\item $\mathfrak{n}:\mathcal{V}_\tau\to\{1,\ldots, d\}$.
		\end{itemize}
			If two vertices $u,v$ are connected by an edge, and $u$ is closer to the root than $v$, then $u$ is called the \emph{parent}
			and $v$ is the \emph{child}. Leaves are vertices without children, while the root has no parent.
			
		We denote by $\cT$ the set of decorated rooted trees. In addition we introduce the number of vertices $|\tau|:=|\mathcal{V}_\tau|$, which induces a graded structure on $\cT$:
		\begin{equation*}
			\cT=\bigoplus_{k=1}^\infty\cT^{=k}\,,
		\end{equation*}
		where $\cT^{=k}:=\{\tau\in\cT\,:\,|\tau|=k\}$. We also define $\cT^{\leq k}:=\bigcup_{k'\leq k}\cT^{=k'}$ and we denote by $\cF$ the set of decorated forests, obtained by juxtaposition of trees, endowed with the (commutative) \textit{forest product} $\cdot:\cF\times\cF\to\cF$ that amounts to the disjoint union of forests. In addition we denote by $\one$ the empty forest intended as the unit of the forest product. Finally, we denote by $\mathcal{H}$ the linear span of $\cF$ and $\cF^{= m}$ the set of forests with exactly $m$ nodes and $\cF^{\leq k}:=\bigcup_{k'\leq k}\cF^{=k'}$.
	\end{definition}
	
	Note that trees in $\cT$ are drawn on the plane but they are \emph{non-planar}: for example $\<201>= \<102>$.
	We observe that there is a one-to-one correspondence between $\cF$ and the subset of  $\cT$ of rooted trees with root decoration $i\in\{1,\ldots, d\}$, realised by the map $[\cdot]_i:\cF\to\cT$ whose action amounts to attaching the roots of the trees lying in the forest to a common new root $\bullet_i$ via new edges. It will be extensively exploited in the following. When the decoration of the root does not play any role, we will simply denote it by $[\cdot]$.
	
	The fourth term in the truncated expansion of the solution in \eqref{Eq: approximation smmoth} comes equipped with a numerical prefactor linked with the symmetry of the iterated integral or, equivalently, with the symmetry of the corresponding rooted tree. This motivates the following definition. 
	\begin{definition}\label{Def: symmetry factor}
		The \emph{symmetry factor} of a tree is defined recursively by imposing $s(\bullet)=1$ and,
		for distinct $\tau_1,\ldots,\tau_k\in\cT$, $n_1,\ldots, n_k\in\N$ and considering the tree $\tau=[\tau_1^{\cdot n_1}\cdots\tau_k^{\cdot n_k}]$, by
		\begin{equation}\label{Eq: symmetry factor}
			s(\tau):=\prod_{i=1}^ks(\tau_i)^{n_i}n_i!\,.
		\end{equation}
		It extends to forests by identifying $s(\tau_1^{\cdot n_1}\ldots\tau_k^{\cdot n_k})$ with \eqref{Eq: symmetry factor}.
		We also define, for the forest $w=\tau_1^{n_1}\cdots\tau_k^{n_k}\in\cF$, the \emph{permutation factor} 
		\begin{equation}\label{Eq: permutation factor}
			\pi(w):=\prod_{i=1}^kn_i!\,,
		\end{equation}
		which counts the number of permutations of identical trees in the forest $w$. Note that $\pi\equiv 1$ on trees. 
		In addition, we denote by $\pi(\tau_i;w)$ the number of times the tree $\tau_i$ appears in the forest $w$,  i.e. for $w=\hat\tau^{\ell_1}_1\cdots\hat\tau_r^{\ell_r}$ and $\hat\tau_i\ne\hat\tau_j$ for $i\ne j$, we have $\pi(\tau_i;w)=\ell_i$.
Finally $w\setminus\tau_i$ is the forest obtained by removing an instance of $\tau_i$ in $w$. 
	\end{definition}
	\begin{definition}\label{Def: grafting}
		For $\tau\in\cT, w\in\cF$ we recursively define the grafting of trees in $\cH$ as
		\begin{equation}
			\tau\graf\one=0\,,\qquad\tau\curvearrowright [w]_i=[\tau\cdot w+\tau\curvearrowright w]_i\,,\qquad i=1,\ldots, d\,,\label{Eq: grafting iterative}
		\end{equation}
		where we implicitly extend $\graf$ to forests on the right by Leibniz rule, i.e, for $\tau_1,\ldots, \tau_k\in\cT$
		\begin{equation}\label{Eq GU Leibniz}
			\tau\curvearrowright(\tau_1\cdots\tau_k)=\sum_{i=1}^k\tau_1\cdots(\tau\curvearrowright\tau_i)\cdots\tau_k\,.
		\end{equation}
	\end{definition}
		Observe that $\tau\curvearrowright\tau'$ can be seen graphically as the collection of trees obtained by linking the root of $\tau$ to all the vertices of $\tau'$ via a new edge, see Appendix \ref{App: grafting}, where 
		we also discuss two possible choices of grafting and their relation. 
In particular, it is well known that $(\langle\cT\rangle,\graf)$,
		where $\langle\cT\rangle$ is the linear span of $\cT$ in $\cH$, forms a pre-Lie algebra, namely for all $\tau_1,\tau_2,\tau_3\in\cT$
\[
    \tau_1\graf (\tau_2\graf \tau_3) - (\tau_1\graf \tau_2)\graf \tau_3 = \tau_2\graf (\tau_1\graf \tau_3) - (\tau_2 \graf \tau_1)\graf \tau_3\,.
\]
	As shown in  \cite{oudom2008lie}, this property allows to extend canonically the grafting operation to $\graf:\cH\times\cH\to\cH$ by the following inductive prescription, which we will call the Guin-Oudom extension. For $\tau_1,\ldots, \tau_k,\tau\in\cT$ we impose
	\begin{align}\label{Eq: GU definition}
		(\tau_1\cdots \tau_k)\graf\tau=\tau_1\graf((\tau_2\cdots\tau_k)\graf\tau)-(\tau_1\graf(\tau_2\cdots\tau_k))\graf\tau\,.
	\end{align} 
	From this non-associative operation we can build an associative product, see \cite{oudom2008lie}, which encodes a fundamental algebraic property of rough paths, namely the Chen relation, see \eqref{Eq: Chen's relation} below.
	\begin{definition}
		Let $\graf:\cH\times\cH\to\cH$ be the Guin-Oudom extension of grafting. We define the bilinear associative product $\star:\cH\times\cH\to\cH$ as
		\begin{align}\label{Eq: star product}
			w\star v=\sum_{w_1\cdot w_2=w}w_1\cdot(w_2\graf v)\,, \qquad w,v\in\cF.
		\end{align}
	\end{definition}
	
	We define the pairing $\langle\cdot,\cdot\rangle:\cF\times\cF\to\R$ by 
	\begin{equation}\label{Eq: pairing}
		\langle\frac{w}{\pi(w)},u\rangle:=\un{\{w=u\}}.
	\end{equation}
	We denote by $\cH_N$ the linear span of $\cF$ truncated at order $N$ and by $\cH_N^\ast$ its dual, whose elements can be seen as series in $\cH_N$ weighted by the permutation factor, i.e.
	\begin{equation}\label{Eq: dual identification}
		f=\sum_{w\in\cF^{\leq N}}f(w)\,\frac{w}{\pi(w)}\,,
	\end{equation}
	endowed with the duality pairing 
	\[
	\langle f,u\rangle=\sum_{w\in\cF^{\leq N}}f(w)\,\langle\frac{w}{\pi(w)},u\rangle,
	\] 
	for $u\in\cF$, $f\in\mathcal{H}^\ast_N$.
	We naturally restrict the associative product $\star$ to this truncated dual, namely 
	we define $\star:\cH_N^\ast\times \cH_N^\ast\to \cH_N^\ast$, via the identification in \eqref{Eq: dual identification}. 
	Notice that for a tree $\tau$ the permutation factor $\pi(\tau)$ is equal to $1$, so that in particular for $\tau,\tau'\in\cT$ we can rewrite the grafting operation as
	\[
	\tau\curvearrowright\tau'=\sum_{\rho\in\cT}\langle\tau\curvearrowright\tau',\rho\rangle\rho\,.
	\]

	We are in the position to give the definition of a branched rough path, first introduced by Gubinelli in \cite{gubinelli2010ramification} and then
	reformulated by \cite{hairer2015geometric}.
	
	\begin{definition}\label{Def: branched rough paths}
		For $\alpha\in(0,1]$, $N:=\lfloor\frac{1}{\alpha}\rfloor$ and $T>0$, a branched $\alpha$-rough path is a map $\X:[0,T]_\leq^2\to \cH^\ast_N$ such that 
		\begin{itemize}
			\item For any $(s,u,t)\in[0,T]_\leq^3$, $\X$ satisfies Chen's relation in $\cH_N^\ast$
			\begin{equation}\label{Eq: Chen's relation}
				\X_{st}=\X_{su}\star\X_{ut}\,.
			\end{equation}
			\item For any $\tau\in\cT$
			\begin{equation*}
				|\X_{st}(\tau)|\lesssim |t-s|^{|\tau|\alpha}\, ,
			\end{equation*}
			uniformly on $0\le s\le t\le T$.
			\item For any forest $w=\tau_1\cdots\tau_k\in\cF^{\leq N}$,
			\[\X_{st}(w)=\prod_{i=1}^k\X_{st}(\tau_i)\,.\]
		\end{itemize}
		If in addition $\X_{st}(\bullet_i)=\delta X^i_{st}$ for a path $X:[0,T]\to \R^d$, we say that $\X$ is an $\alpha$-branched rough path lift of $X$. 
	\end{definition}
	Note however that our choice for the product $\star$ is not exactly the same as in \cite{gubinelli2010ramification,hairer2015geometric}: see Appendix \ref{App: smooth rough paths} for a discussion.
	
	The family of functions $(\X_{st}(\tau))_{\tau\in\cT^{\le N}}$ plays the role of a set of generalised monomials in \eqref{eq:davie}.
	We still have to define the coefficients $\Upsilon_\tau$ which appear in the expansion and which play the role of generalised derivatives
	of the solution. Borrowing a term from numerical analysis,
	the one-to-one correspondence between rooted trees and the coefficients of the expansion \eqref{eq:davie} is realized by the \emph{elementary differential} defined below. 
	\begin{definition}\label{Def: elementary differential trees}
		For $\sigma\in C^\infty(\R^n,\R^n\otimes (\R^d)^\ast)$ we inductively define the \emph{elementary differential} $\Upsilon:\cT\to C( \R^n;\R^n)$ by imposing that, for all $z\in\R^n$, 
		$[\Upsilon_{\bullet_j}(z)]^i:=\sigma^i_k(z)$, $i=1,\ldots,n$, $j=1,\ldots,d$, and, for all $\tau_1,\ldots, \tau_k\in\cT$,
		\begin{align}\label{Eq: elementary differential tree}
			&[\Upsilon_{[\tau_1\cdots \tau_k]_j}(z)]^i:=\sum_{i_1,\ldots,i_k=1}^n \frac{\partial^k\sigma^i_j(z)}{\partial z_{i_1}
				\cdots \partial z_{i_k}} \, [\Upsilon_{\tau_1}(z)]^{i_1}\cdots [\Upsilon_{\tau_k}(z)]^{i_k}\,.
		\end{align}
		We extend it to a map $\Upsilon:\cF\to C( \R^n;\R^n)$ as a morphism with respect to the forest product such that $\Upsilon_\one=1$.
	\end{definition}	
	\begin{notation}
		To ease the notation we will adopt the following convention to represent matrix products as the one in \eqref{Eq: elementary differential tree}. For $f: \R^n\to\R^n$ and $v_1,\ldots, v_k\in\R^n$ we write
		\begin{align*}
			&[\nabla^k f]_{i_1,\ldots,i_k} :=\frac{\partial^k f}{\partial z_{i_1}
				\cdots \partial z_{i_k}}\,,
			&\nabla^k f \cdot [v_1\cdots v_k]:= \sum_{i_1,\ldots, i_k=1}^n [\nabla^k f]_{i_1,\ldots, i_k} v_1^{i_1}\cdots v_k^{i_k}\,.
		\end{align*} 
		In particular, \eqref{Eq: elementary differential tree} can be rewritten as
		\[
		\Upsilon_{[\tau_1\cdots \tau_k]_j} = \nabla^k \sigma_j \cdot  [\Upsilon_{\tau_1}\cdots\Upsilon_{\tau_k}]\,.
		\]
	\end{notation}
	
	We state a (well-known) morphism property of elementary differentials, whose proof is postponed to Appendix \ref{App: grafting}. 
	\begin{lemma}\label{lem: morphism upsilon}
		Let  $\Up$ be the elementary differential of Definition \ref{Def: elementary differential trees}.	
		For any $\tau, \tau_1,\ldots,\tau_k\in\cT$ we have that 
		\begin{equation}\label{eq: morphism upsilon}
			\Up_{(\tau_1\cdots\tau_k)\graf\tau}=\nabla^k \Up_{\tau}\cdot [\Upsilon_{\tau_1}\cdots\Upsilon_{\tau_k}]\,.
		\end{equation}
	\end{lemma}

	We recall the definition of the solution remainders:
	\begin{equation}\label{Eq: Z remainders}
		Z_{st}^{[m+1]}:=\delta Z_{st}-\sum_{\tau\in\cT^{\le m}}\Up_\tau(Z_s) \, \X_{st}(\tau)\,, \qquad m\ge 0\,.
	\end{equation}
	Note that $Z^{[1]}_{st}=\delta Z_{st}$. We can now introduce the notion of solution \emph{à la} Davie of the rough equation
	\eqref{Eq: controlled ODE}, already anticipated in the introduction.
	\begin{definition}\label{def:aladavie}
		Fix $\alpha\in(0,1]$, $N:=\lfloor \frac1\alpha\rfloor$ and let $\X$ be an $\alpha$-branched rough path lift of $X\in\CC^\alpha([0,T],\R^d)$ in the sense of Definition \ref{Def: branched rough paths}. Consider $\sigma\in C^{N-1}(\R^n,\R^n\otimes(\R^d)^\ast)$.
		A solution \emph{à la} Davie to the rough equation
		\eqref{Eq: controlled ODE} is a map $Z:[0,T]\to\R^n$ that satisfies
		\begin{equation}\label{eq: sol Davie branched}
			\delta Z_{st}=\sum_{\tau\in\cT^{\leq N}}\Up_\tau(Z_s) \, \X_{st}(\tau)+o\left(t-s\right)\,,
		\end{equation}
		uniformly for $0\le s\le t\le T$.
	\end{definition}
	
	The Ansatz that underpins this definition is based on a rigorous control of the solution remainders in the case of
	$X\in C^1$: see Proposition \ref{pr:daviesmooth} and Remark \ref{rem:davie} below.

	\section{Generalized Taylor formulae for branched rough paths}\label{Sec: generalised Taylor formulae}
	The rationale behind the Davie solution to a rough ODE is to locally approximate it with a truncated series over a prescribed set of \emph{generalised monomials} with sufficient control on the truncation error. The goal of this section is to derive explicit expressions for the elementary differential remainders that enter the study of a priori estimates on the Davie solutions.
	
	We fix for the rest of this section $\alpha\in(0,1]$, $N:=\lfloor \frac1\alpha\rfloor$, a driving path $X\in\mathcal{C}^{\alpha}$ and an $\alpha$-branched rough path $\X$ over $X$.	Moreover we consider coefficients $\sigma\in C^{N-1}(\R^n,\R^n\otimes(\R^d)^\ast)$.
	We recall the definition \eqref{Eq: Z remainders} of the solution remainders. Then by Definition \ref{def:aladavie}, $Z$ is a solution \emph{à la} Davie of the rough equation
	\eqref{Eq: controlled ODE} if $Z_{st}^{[N+1]}=o(|t-s|)$.
	
	An important step in the derivation of a priori estimates for the remainders $Z^{[i]}$ consists of finding a convenient expression for the three-point increment $\delta Z^{[N+1]}$. A good control of the latter transfers to $Z^{[N+1]}$ through the Sewing bound, see Section \ref{Sec: sewing bound}.
	
	We set for $\tau\in\CT$ and $m\geq 0$:
	\begin{equation}\label{eq:Bn}
		B^{[m+1]}_{st}(\tau) := 
		\Up_\tau(Z_t) - \sum_{w\in\CF^{\leq m}}\frac{\Up_{w\curvearrowright\tau}(Z_s)}{\pi(w)}  \, \mathbb{X}_{s t}(w)\,.
	\end{equation}
	We interpret $B^{[m+1]}_{st}(\tau)$ as the remainder of a generalised Taylor expansion of 
$\Up_\tau(Z):[0,T]\to\R^n$, with respect to the set of generalised monomials $(\mathbb{X}_{s t}(w))_{w\in\CF^{\leq m}}$.
	\begin{proposition}\label{Lem: delta Z}
		Let $Z$ satisfy \eqref{eq: sol Davie branched} and $Z^{[N+1]}$ be as in \eqref{Eq: Z remainders}. Then
		\begin{align}
			\delta Z_{s u t}^{[N+ 1]} 
			= \sum_{\tau\in\CT^{\leq N}} B_{s u}^{[N-|\tau|+1]}(\tau) 
			\, \mathbb{X}_{u t}(\tau)\,.
			\label{dZN+1}
		\end{align}
	\end{proposition}
	\begin{proof}
		Recalling that $\one\notin\cT$, then the definition of the $(N+1)$-st order remainder of the solution reads
		\[
		Z^{[N+1]}_{st}=\delta Z_{st} - \sum_{\tau\in\CT^{\leq N}} \Up_\tau(Z_s) \, \mathbb{X}_{st}(\tau)\,.
		\]
		Observe that, if $C_{st}=f_sA_{st}$, then
		\[
		\delta C_{sut} = f_s\delta A_{sut} - \delta f_{su} A_{ut}\,,\hspace{1.5cm}(s,u,t)\in[0,T]_\leq^3\,.
		\]
		From this identity and recalling that $\delta\circ \delta=0$, we have
		\[
		\delta Z^{[N+1]}_{sut} = \sum_{\tau\in\CT^{\leq N}} \left[(\Up_\tau(Z_u)-\Up_\tau(Z_s)) \, \mathbb{X}_{ut}(\tau)-\Up_\tau(Z_s) \, \delta \mathbb{X}_{sut}(\tau)
		\right]\,.
		\]
		By Chen's relation \eqref{Eq: Chen's relation}, for all $\tau\in\cT$:
		\begin{align*}
			\delta  \mathbb{X}_{sut}(\tau) 
			&=\sum_{w\in\cF\setminus\{\one\}\,,\,\tau'\in\cT} \frac{1}{\pi(w)}\, \mathbb{X}_{su}(w)\, 
			\mathbb{X}_{ut}(\tau') \, \langle w\curvearrowright\tau',\tau\rangle\,.
		\end{align*}
		Then we obtain the formula
		\begin{align*}
			\delta Z^{[N+1]}_{sut} &= \sum_{\tau\in\CT^{\leq N}} \left(\Up_\tau(Z_u)-\Up_\tau(Z_s)\right) \, \mathbb{X}_{ut}(\tau)
			\\ & \quad -\sum_{\tau\in\CT^{\leq N}}
			\sum_{w\in\cF\setminus\{\one\}\,,\,\tau'\in\cT^{\leq N}}\frac{1}{\pi(w)} \Up_\tau(Z_s) \, \mathbb{X}_{su}(w) \, 
			\mathbb{X}_{ut}(\tau') \, \langle w\curvearrowright\tau',\tau\rangle
			\\ & = \sum_{\tau\in\CT^{\leq N}} \left(\Up_\tau(Z_u)-
			\sum_{w\in\cF^{\leq N-|\tau|}} \frac{1}{\pi(w)}\Up_{w\curvearrowright\tau}(Z_s)\, \mathbb{X}_{su}(w)
			\right) \, \X_{ut}(\tau)\,.
		\end{align*}
		The identification in \eqref{eq:Bn} allows to conclude.
	\end{proof}
	Proposition \ref{Lem: delta Z} clearly shows that, after resorting to the Sewing bound, the remainder norm can be controlled as
	\[
	\begin{split}
		\|Z^{[N+1]}\|_{(N+1)\alpha} & \le K_{(N+1)\alpha} \|\delta Z^{[N+1]}\|_{(N+1)\alpha} 
		\\ & \le K_{(N+1)\alpha} \sum_{\tau\in\CT^{\leq N}}
		\|B^{[N-|\tau|+1]}(\tau)\|_{(N-|\tau|+1)\alpha} \|\mathbb{X}_{u t}(\tau)\|_{|\tau|\alpha}\,.
	\end{split}
	\]
	Since by assumption $\|\mathbb{X}(\tau)\|_{|\tau|\alpha}<+\infty$, the desired bounds directly follow from suitable estimates on
	\[
	\|B^{[N-|\tau|+1]}(\tau)\|_{(N-|\tau|+1)\alpha}\,, \qquad \forall \, \tau\in\CT^{\leq N}.
	\]
	The main aim of this article is to derive closed formulae for the remainders $B^{[i]}$ that would guarantee control of the latter under minimal conditions on the relevant elementary differentials. 
	
	\subsection{Taylor expansions for bounded elementary differentials}\label{Sec: bounded}
	Throughout this section, we set aside the rough path framework and work in a more abstract setting, as the algebraic and analytic identities established below hold for a broad class of expansions sharing a common prescribed structure. For this reason, we adopt the general framework introduced in Section \ref{Sec: remainders}, in particular the definitions \eqref{eq:zk12}-\eqref{eq:B12} of $z^{[m+1]}$ and $B^{[m+1]}_{z_1z_2}$.
		We also use the following shorthand notation: for $w=\tau_1\cdots\tau_k\in\cF$ we write
	\[
	[\Upsilon \X]^{w}:=\bigotimes_{i=1}^{k}\Up_{\tau_i}(z_1) \, \X(\tau_i)\,.
	\] 
	Leveraging on the morphism property of $\Upsilon$, we derive a first explicit formula for the remainder of a generic elementary differential. 	
	\begin{theorem}\label{Thm: Taylor a priori}
		For any $\tau\in\cT$, $m\geq0$, recalling \eqref{eq:zk12}-\eqref{eq:B12}:
		\begin{align}
			B^{[m+1]}_{z_1z_2}(\tau)	=\sum_{k=0}^m\sum_{w=\tau_1\cdots \tau_k\in\cF^{\le m}}\int_{[0,1]}\nabla^{k+1}\Up_\tau(z_1+rz^{[1]})\cdot z^{[m+1-|w|]} \,
			\frac{[\Upsilon\X]^w}{\pi(w)} \, (1-r)^{k}\d r\,.\label{Eq: Taylor}
		\end{align}
	\end{theorem}
	\begin{proof}
		We proceed by induction. Since $z^{[1]}=z^2-z^1$, for $m=0$ we have 
		\[
		\Up_\tau(z_2)-	\Up_\tau(z_1)=\int_{[0,1]}\nabla	\Up_\tau(z_1+rz^{[1]})\cdot z^{[1]} \d r\,,
		\]
		which proves the claim. Now suppose that \eqref{Eq: Taylor} holds true for $m\geq 0$. Then
		\begin{align*}
			&		\underbrace{\Up_\tau(z_2)-\sum_{w\in\cF^{\leq m+1}}\frac{\Up_{w\curvearrowright\tau}(z_1)}{\pi(w)} \, \X(w)}_{A} \\
			&=\underbrace{\Up_\tau(z_2)-\sum_{w\in\cF^{\leq m}}\frac{\Up_{w\curvearrowright\tau}(z_1)}{\pi(w)} \, \X(w)}_{B}-
			\underbrace{\sum_{w\in\cF^{=m+1}}\frac{\Up_{w\curvearrowright\tau}(z_1)}{\pi(w)} \, \X(w)}_{C}\,.
		\end{align*}
		By the morphism nature of the elementary differential, see Lemma \ref{lem: morphism upsilon}, we deduce that
		\begin{align*}
			C& =
			\sum_{k=1}^{m+1}\sum_{w=\tau_1\cdots\tau_k\in\cF^{=m+1}}\nabla^k\Up_\tau(z_1)\cdot\frac{[\Upsilon\X]^{w}}{\pi(w)} 
			=
			\sum_{k=0}^m\sum_{w=\tau_1\cdots\tau_{k+1}\in\cF^{=m+1}}\nabla^{k+1}\Up_\tau(z_1)\cdot\frac{[\Upsilon\X]^{w}}{\pi(w)}\,.
		\end{align*}
		By the induction hypothesis, writing $z^{[m+1-|w|]}=z^{[m+2-|w|]}+\sum_{\gamma\in\CT^{=m+1-|w|}}\Up_\gamma(z_1)\, \X(\gamma)$:
		\begin{align*}
			B		&=\sum_{k=0}^m\sum_{w=\tau_1\cdots\tau_k\in\cF^{\le m}}\int_{[0,1]}\nabla^{k+1}\Up_\tau(z_1+rz^{[1]})\cdot z^{[m+1-|w|]} \, 
			\frac{[\Upsilon\X]^{w}}{\pi(w)} \, (1-r)^{k}\d r \\ 
			& = \underbrace{\sum_{k=0}^m\sum_{w=\tau_1\cdots\tau_k\in\cF^{\le m}}\int_{[0,1]}\nabla^{k+1}\Up_\tau(z_1+rz^{[1]})\cdot z^{[m+2-|w|]} \,
				\frac{[\Upsilon\X]^{w}}{\pi(w)} \, (1-r)^{k}\d r}_{D}
		\end{align*}
		\begin{align*}
		 & + \underbrace{\sum_{k=0}^m\sum_{w=\tau_1\cdots\tau_k\in\cF^{\le m}}\sum_{\gamma\in\cT^{=m+1-|w|}}\int_{[0,1]}\nabla^{k+1}\Up_\tau(z_1+rz^{[1]})\cdot \Upsilon_{\gamma}(z_1) \, \X(\gamma)\,\frac{[\Upsilon\X]^{w}}{\pi(w)} \, (1-r)^{k}\d r}_{E}.
		\end{align*}
		We wish to reshape $E$ in a form that resembles $C$. Care must be taken in including the sum over $\gamma\in\cT^{=m+1-|w|}$ as part of the sum over a larger forest as there is a hidden degeneracy that must be preserved. Indeed 
		\begin{align*}
			E&= \sum_{k=0}^m\sum_{w=\tau_1\cdots\tau_{k+1}\in\cF^{\le m}}\sum_{i=1}^{k+1}\frac{1}{\pi(\tau_i;w)}\int_{[0,1]}\nabla^{k+1}\Up_\tau(z_1+rz^{[1]})\cdot \frac{[\Upsilon\X]^w}{\pi(w\setminus\tau_i)} \, (1-r)^{k}\d r\,,
		\end{align*}
		where $\pi(\tau_i;w)$ has been introduced in Definition \ref{Def: symmetry factor}.
		Upon observing that $\pi(w\setminus\tau_i)\,\pi(\tau_i;w)=\pi(w)$ and that the sum over $i$ yields a factor $k+1$, we obtain
		\begin{align*}
			E&= \sum_{k=0}^m\sum_{w=\tau_1\cdots\tau_{k+1}\in\cF^{\le m}} (k+1) \int_{[0,1]}\nabla^{k+1}\Up_\tau(z_1+rz^{[1]})\cdot \frac{[\Upsilon\X]^w}{\pi(w)} \, (1-r)^{k}\d r\,.
		\end{align*}
		As a result
		\begin{align*}
			&		E - C \\
			& = \sum_{k=0}^m
			\sum_{w=\tau_1\cdots\tau_{k+1}\in\cF^{=m+1}}(k+1)
			\int_{[0,1]}	\left(\nabla^{k+1}\Up_\tau(z_1+rz^{[1]})-\nabla^{k+1}\Up_\tau(z_1)\right)\cdot 
			\frac{[\Upsilon\X]^{w}}{\pi(w)} \, (1-r)^{k}\d r
			\\ & = \sum_{k=0}^m
			\sum_{w=\tau_1\cdots\tau_{k+1}\in\cF^{=m+1}}
			\int_{[0,1]} (k+1)\,(1-r)^{k} \d r \int_{[0,r]}\nabla^{k+2}\Up_\tau(z_1+uz^{[1]})\cdot 
			z^{[1]} \,\frac{[\Upsilon\X]^{w}}{\pi(w)} \d u 
			\\ & = \sum_{k=0}^m
			\sum_{w=\tau_1\cdots\tau_{k+1}\in\cF^{=m+1}}
			\int_{[0,1]}\nabla^{k+2}\Up_\tau(z_1+uz^{[1]})\cdot  z^{[1]}\,
			\frac{[\Upsilon\X]^{w}}{\pi(w)} \, (1-r)^{k+1}	\d u 
			\\ & = \sum_{k=0}^{m+1}
			\sum_{w=\tau_1\cdots\tau_{k}\in\cF^{=m+1}}
			\int_{[0,1]}\nabla^{k+1}\Up_\tau(z_1+uz^{[1]})\cdot z^{[1]}\,
			\frac{[\Upsilon\X]^{w}}{\pi(w)} \, (1-r)^{k}\d u \,.
		\end{align*}
		Note that in the last line we performed the change of summation variable $k'=k+1$ and we added the boundary case $k=0$, which has vanishing argument by the constraint $|w|= m+1>0$. Then
		\[
		\begin{split}
			&	B-C = D+E-C=
			\\ & = \sum_{k=0}^{m+1}\sum_{w=\tau_1\cdots\tau_k\in\cF^{\le m}}\int_{[0,1]}\nabla^{k+1}\Up_\tau(z_1+rz^{[1]})\cdot z^{[m+2-|w|]}  \frac{[\Upsilon\X]^{w}}{\pi(w)} \, (1-r)^{k}\d r
			\\ & \quad +  \sum_{k=0}^{m+1}
			\sum_{w=\tau_1\cdots\tau_{k}\in\cF^{=m+1}}
			\int_{[0,1]}\nabla^{k+1}\Up_\tau(z_1+uz^{[1]})\cdot z^{[1]}\,
			\frac{[\Upsilon\X]^{w}}{\pi(w)} \, (1-r)^{k}\d u = A\,,
		\end{split}
		\]
		where in $D$ we added to the sum the boundary case $k=m+1$ since the corresponding term is equal to zero by the constraint $w=\tau_1\cdots\tau_{k}\in\cF^{\le m}$.	This concludes the proof.
	\end{proof}
	Choosing $\X=\X_{st}$, $z_1=Z_s$ and $z_2=Z_t$, 
	Equation \eqref{dZN+1} and Theorem \ref{Thm: Taylor a priori} for $m=0,\ldots,N-1$ specialise to
	\begin{equation}\label{eq:boubou}
		B^{[m+1]}_{st}(\tau)
		=\sum_{k=0}^{m}\sum_{w=\tau_1\cdots\tau_k\in\cF^{\le m}}\int_{[0,1]}\nabla^{k+1}\Up_\tau(Z_s+rZ^{[1]}_{st})\cdot  Z^{[m+1-|w|]}_{st}\,\frac{[\Upsilon\X]_{st}^{w}}{\pi(w)}\, (1-r)^{k}\d r\,.
	\end{equation}
	This is a first exact formula for the remainders $B^{[m+1]}_{st}(\tau)$ that appear in \eqref{dZN+1}. 
	Under suitable boundedness assumptions on the elementary differentials $\Upsilon_\tau$ and their derivatives, the explicit form of the remainders in 
	\eqref{eq:boubou} allows to bound them in terms of the H\"older norm of the errors $Z^{[i]}$: see 
	Proposition \ref{lem:bounded} below.
	
	In the next section we address the (much more difficult) problem of a different exact formula for the remainders $B^{[m+1]}_{st}(\tau)$ that allows to
	bound them assuming only Lipschitz continuity and not boundedness of $\Upsilon_\tau$.
	
	\begin{remark}
The results of \cite[Proposition 3.11]{bonnefoi2022priori} and \cite[Proposition A.1]{hairer2014theory} are similar in spirit to \eqref{Eq: Taylor}. The key difference is that, despite being less general, our formula is more precise since it includes also the intermediate solution remainders $Z^{[i]}$ for $i=2,\ldots,N$. Notice that all these formulae are well suited for the study of global a priori estimates on Davie's solutions when all relevant derivatives of $\Up$ are bounded. 
	\end{remark}
	
	\begin{remark}
		A natural question is whether the generalised Taylor formulae of Theorem \ref{Thm: Taylor a priori} can also be derived within the framework of (weakly) geometric rough paths. Although we obtained an analogous formula, its derivation involved additional, and somewhat unexpected, difficulties. At first sight, one might expect that the richer structure of geometric rough paths, most notably the availability of an integration-by-parts formula, would simplify the analysis. However, the corresponding Ansatz for the solution expansion is no longer indexed by trees, resulting in the loss of the convenient algebraic structures associated with them, such as the grafting operation and the morphism property of elementary differentials.
		
		More importantly, in the geometric setting, the multiplicativity of characters is formulated with respect to the shuffle product. As a consequence, every decomposition of the rough path appearing in the preceding calculations gives rise to additional combinatorial terms that must be tracked carefully. Developing this approach would require a substantial detour from the main objectives of the present work, and we therefore leave it for future investigation.
	\end{remark}

	\section{Taylor expansions for Lipschitz elementary differentials}\label{Sec: Lipschitz elementary differentials}
	
	As showed in \eqref{eq:boubou}, the expansion of the remainders involved in the calculation of $\delta Z^{[N+1]}$ as derived in Theorem \ref{Thm: Taylor a priori} can be controlled by lower order solution remainders when sufficiently many derivatives of the elementary differentials are bounded. 
	Using again the general framework introduced in Section \ref{Sec: remainders}, in particular the definitions \eqref{eq:zk12}-\eqref{eq:B12} of 
	$z^{[m+1]}$ and $B^{[m+1]}_{z_1z_2}$, we can interpret \eqref{Eq: Taylor} as follows:
	\[
	B^{[m+1]}_{z_1z_2}(\tau)=\sum_{w\in\cF^{\leq m}}
	H_{z_1z_2}(w;\tau)\cdot z^{[m+1-|w|]} \, \X(w)
	\]
	where, for $w=\tau_1\cdots \tau_k\in\cF^{\le m}$ and $z^{[1]}=z^2-z^1$
	\begin{equation}\label{Eq: identity remainder H1}
	H_{z_1z_2}(w;\tau):= 
	\int_{[0,1]}\nabla^{k+1}\Up_\tau(z_1+r z^{[1]})\cdot \Upsilon_w(z_1)\,\frac{(1-r)^{k}}{k!}\d r\,.
	\end{equation}
	This clearly highlights how this formula is effective only under boundedness assumptions on 
	$\Upsilon_\tau$ and its derivatives.
On the other hand, if we only assume that $\Upsilon_\tau$ is globally Lipschitz continuous for all $\tau\in\cT^{\leq N}$, it would be necessary to derive alternative formulae that only involve gradients of the elementary differentials. An additional requirement when studying a priori estimates, which will become transparent in the proof of Theorem \ref{th: a priori branched}, is that the Taylor formulae we are looking for must be linear in the solution remainders $z^{[i]}$. 
	
	The main result of this section is an alternative, explicit expression for the coefficients $H_{z_1z_2}(w;\tau_0)$, given in Theorem \ref{Thm: Taylor Lipschitz} below, which naturally allows for a bound of the form
	\[
	|H_{z_1z_2}(w;\tau_0)| \lesssim \prod_{\tau\in\cT^{\le |w|+|\tau_0|}}\|\nabla\Upsilon_{\tau}\|_\infty\,.
	\]	

	To highlight the key ideas underlying our method, we shall elaborate on the remainder in the left-hand side of \eqref{eq:B12} at low orders. 
	For $z_1,z_2\in\R^n$ we introduce the notation 
	\begin{equation*}
		u_i:=z_1+t_i z^{[1]}\,,\qquad i\in \N,\ t_i\in[0,1]\,.
	\end{equation*}
	  where $z^{[1]}=z^2-z^1$. We also restrict to the case of scalar equations as this allows to drop the tree decorations, even if everything we are presenting works exactly the same in higher dimensions. 
	
	 Recalling the notation \eqref{eq:zk12}-\eqref{eq:B12} for the abstract remainders $z^{[m+1]}$ and 
	$B^{[m+1]}_{z_1z_2}(\tau)$, the first relevant case $m+1=1$ is simple and yields
	\begin{align*}
		B_{z_1z_2}^{[1]}(\tau)= \Up_{\tau}(z_2)-\Up_{\tau}(z_1)=\int_0^1\nabla\Up_\tau(u_0) \, z^{[1]} \d t_0\,,
	\end{align*}
	which indeed satisfies our requirements. Moving to the case $m+1=2$, the only forest $w\in\cF^{=1}$ is the single vertex itself, entailing 
	\begin{align*}
		B_{z_1z_2}^{[2]}(\tau)=	B_{z_1z_2}^{[1]}(\tau)-\Up_{\bullet\graf\tau}(z_1) \, \X(\bullet)\,,
	\end{align*}
	which is a special instance of the relation
	\begin{align}\label{Eq: inductive B}
		B^{[m+2]}_{z_1z_2}(\tau)=B^{[m+1]}_{z_1z_2}(\tau) - \sum_{w\in\CF^{= m+1}}\frac{\Up_{w\curvearrowright\tau}(z_1)}{\pi(w)}  \, \mathbb{X}(w)\,.
	\end{align}
	Analogously we have for $z^{[m+2]}$ by \eqref{eq:zk12}
	\begin{align}\label{Eq: inductive z}
		z^{[m+1]}=z^{[m+2]} + \sum_{\tau\in\cT^{= m+1}} \Up_{\tau}(z_1)  \, \mathbb{X}(\tau)\,.
	\end{align}
	Our strategy now is to increase the order of the remainder $z^{[1]}$ in $B_{z_1z_2}^{[1]}(\tau)$ in order to absorb the term $\Up_{\bullet\graf\tau}(z_1) \, \X(\bullet)$, which is not yet in gradient form. Indeed we can rewrite $B_{z_1z_2}^{[2]}(\tau)$ as
	\begin{align*}
		B_{z_1z_2}^{[2]}(\tau)&=\int_0^1\nabla\Up_\tau(u_0) \, z^{[1]} \d t_0-\Up_{\bullet\graf\tau}(z_1) \, \X(\bullet)
		\\&=\int_0^1\nabla\Up_\tau(u_0) \, z^{[2]} \d t_0 +\left[{\color{red} \int_0^1\nabla\Up_\tau(u_0) \Up_{\bullet}(z_1) \d t_0-\Up_{\bullet\graf\tau}(z_1)}\right] \, \X(\bullet)\,,
	\end{align*}
	where we have used \eqref{Eq: inductive z} in the second line.
	We can massage the red terms to a nicer form by summing and subtracting a term that would help us forming new gradients:
	\begin{align*}
		&\int_0^1\left[{\color{red}\nabla\Up_\tau(u_0) \Up_{\bullet}(z_1)} -\nabla\Up_\tau(u_0) \Up_{\bullet}(u_0)\right]\d t_0 +	\int_0^1
		\left[\nabla\Up_\tau(u_0) \Up_{\bullet}(u_0){\color{red}-\Up_{\bullet\graf\tau}(z_1)}\right] \d t_0
		 \\
		&= -\int_0^1\int_0^{t_0}\nabla\Up_\tau(u_0) \, \nabla\Up_{\bullet}(u_1) \, z^{[1]}\d t_0\d t_1+\int_0^1(\Up_{\bullet\graf\tau}(u_0)-\Up_{\bullet\graf\tau}(z_1)) \d t_0\\
		&=\int_0^1\int_0^{t_0}\left(\nabla \Up_{\bullet\graf\tau}(u_1)-\nabla\Up_\tau(u_0) \, \nabla\Up_{\bullet}(u_1)\right) \, z^{[1]} \d t_1\d t_0\,.
	\end{align*}
	In the second line we used the morphism property of $\Upsilon$ with respect to the grafting product. As a result
	\begin{align}\label{Eq: B^2}
	B_{z_1z_2}^{[2]}(\tau)&=\int_0^1\nabla\Up_\tau(u_0) \, z^{[2]} \d t_0 +\nonumber
	\\ & \quad +\int_0^1\int_0^{t_0}{\color{red}\left(\nabla \Up_{\bullet\graf\tau}(u_1)-\nabla\Up_\tau(u_0) \, \nabla\Up_{\bullet}(u_1)\right)}\,z^{[1]} \, 
	\X(\bullet)\d t_1\d t_0\,.
	\end{align}
	Observe that this expression satisfies our requirements since it only contains gradients of $\Upsilon$. Most importantly, the procedure adopted to create new gradients does not depend on the specific trees considered and will be the cornerstone of our method.

	Unfortunately, this procedure gets rapidly more complicated. The key issues already arise when our approach is carried over to $B_{z_1z_2}^{[3]}(\tau)$. Recalling \eqref{Eq: inductive B} we express the remainder as
	\begin{align*}
&B_{z_1z_2}^{[3]}(\tau)=B^{[2]}_{z_1z_2}(\tau)-\Upsilon_{\<1>\,\graf\tau}(z_1)\,\X(\,\<1>\,)-\frac{1}{2}\Upsilon_{\bullet\bullet\graf\tau}(z_1)\,\X(\bullet\bullet)
		\\&=\int_0^1\nabla\Upsilon_{\tau}(u_0) \, z^{[2]} \d t_0 
		+\int_0^1\int_0^{t_0}\left(\nabla \Upsilon_{\bullet\graf\tau}(u_1)-\nabla\Upsilon_{\tau}(u_0) \, \nabla\Upsilon_{\bullet}(u_1)\right)\,z^{[1]}\,\X(\bullet)\d t_1\d t_0\\
		&\quad-\Upsilon_{\<1>\,\graf\tau}(z_1)\,\X(\,\<1>\,)-\frac{1}{2}\Upsilon_{(\bullet\bullet)\graf\tau}(z_1)\,\X(\bullet\bullet)\,.
	\end{align*}
	First we raise the order of all the errors $z^{[i]}$ using \eqref{Eq: inductive z}, obtaining
	\begin{align}
	B_{z_1z_2}^{[3]}(\tau)=&\int_0^1\nabla\Upsilon_{\tau}(u_0) \, z^{[3]} \d t_0 +\int_0^1{\color{blue}\nabla\Upsilon_{\tau}(u_0)\,\Upsilon_{\<1>}(z_1)}\,\X(\,\<1>\,)\d t_0\nonumber
	\\ & + \int_0^1\int_0^{t_0}\left(\nabla \Upsilon_{\bullet\graf\tau}(u_1)-\nabla\Upsilon_{\tau}(u_0) \, \nabla\Upsilon_{\bullet}(u_1)\right)\,z^{[2]}\,\X(\bullet)\d t_1\d t_0\nonumber\\
		&+\int_0^1\int_0^{t_0}{\color{red}\left(\nabla \Upsilon_{\bullet\graf\tau}(u_1)-\nabla\Upsilon_{\tau}(u_0) \, \nabla\Upsilon_{\bullet}(u_1)\right)\Upsilon_{\bullet}(z_1)}\,\X(\bullet)\,\X(\bullet)\d t_1\d t_0\nonumber\\
		&-\int_0^1{\color{blue}\Upsilon_{\<1>\,\graf\tau}(z_1)}\,\X(\,\<1>\,)\d t_0-\int_0^1\int_0^{t_0}{\color{red}\Upsilon_{(\bullet\bullet)\graf\tau}(z_1)}\,\X(\bullet\bullet)\d t_1\d t_0\,.\label{Eq: B^3 intermedia}
	\end{align}
	In the last line we converted the combinatorial prefactor $\pi(w)$ into an integral over a simplex. 
	Since the first and third terms in the right-hand side already have the desired form, we focus our attention on the remaining ones. An argument identical to the one performed to obtain \eqref{Eq: B^2} allows to group the blue terms as follows 
	\[
	\begin{split}
		\int_{[0,1]_\geq^2}{\color{blue}\left[ \nabla\Up_{\<1> \curvearrowright\tau}(u_1)- \nabla\Up_{\tau}(u_0)\, \nabla\Up_{\<1>}(u_1)\right]} \, z^{[1]} \, \X(\,\<1>\,) \d t_0
		\d t_1\,.
	\end{split}
	\]
	
	To deal with the red terms, more care is required. We show the main steps of the argument, leaving to the interested reader the task of checking the intermediate calculations. 	First, in order to merge the red term in the last line of \eqref{Eq: B^3 intermedia} with any of the other contributions, we must explicitly write down the terms of the forest grafting, see \eqref{Eq: GU definition}. The result reads (ignoring the common factor $\X(\bullet\bullet)$)
	\begin{equation*}
	-\int_0^1\int_0^{t_0}{\color{red}\Upsilon_{(\bullet\bullet)\graf\tau}(z_1)} \d t_1\d t_0=-	\int_0^1\int_0^{t_0}{\color{red}\left(\Upsilon_{\bullet\graf(\bullet\graf\tau)}(z_1)-\Upsilon_{(\bullet\graf\bullet)\graf\tau}(z_1)\right)} \d t_1\d t_0\,.
	\end{equation*}
	The second step consists of adding and subtracting the right expressions to turn differences of elementary differentials into their gradient. After iterating this procedure multiple times the red terms take the form
	\begin{align*}
		&\int_0^1\int_0^{t_0}{\color{red}\left(\nabla \Upsilon_{\bullet\graf\tau}(u_1)-\nabla\Upsilon_{\tau}(u_0) \, \nabla\Upsilon_{\bullet}(u_1)\right)\Upsilon_{\bullet}(z_1)} \d t_1\d t_0
		-\int_0^1\int_0^{t_0}{\color{red}\Upsilon_{(\bullet\bullet)\graf\tau}(z_1)} \d t_1\d t_0
		\\&=-\int_{[0,1]_\geq^3}{\color{violet}\left(\nabla \Upsilon_{\bullet\graf\tau}(u_1)-\nabla\Upsilon_{\tau}(u_0) \, \nabla\Upsilon_{\bullet}(u_1)\right)\, \nabla\Upsilon_{\bullet}(u_2)} \,z^{[1]}\d t_2\d t_1\d t_0\\
		&+\int_{[0,1]_\geq^3}{\color{magenta}\nabla\Upsilon_{\bullet\graf(\bullet\graf\tau)}(u_2)} \,z^{[1]}\d t_2\d t_1\d t_0\\
		&+\int_0^1\left[\int_0^{t_0}\int_{t_1}^{t_0}{\color{orange}\nabla\Upsilon_{\tau}(u_0) \, \nabla\Upsilon_{\bullet\graf\bullet}(u_2)}
		\d t_2 \d t_1
		+\int_0^1\int_0^{t_0}\int_0^{t_0}{\color{orange}\nabla\Upsilon_{(\bullet\graf\bullet)\graf\tau}(u_2)} \d t_2\d t_1\right]z^{[1]} \d t_0.
	\end{align*}
	This expression is almost in the desired form, but it displays a new issue: we encounter integrals that are not over simplexes anymore, e.g. in the orange contributions. This difficulty can be overcome by a manipulation of the nested integrals: in the first case, one exchanges the
	variables $t_1$ and $t_2$, yielding $\color{brown}\nabla\Upsilon_{\tau}(u_0) \, \nabla\Upsilon_{\bullet\graf\bullet}(u_1)$; in the second case, one breaks the integration domain into the two sets $\{t_1<t_2\}\cup\{t_2<t_1\}$ and on $\{t_1<t_2\}$ one
	 exchanges the two variables; the net result is that the second term appears \emph{twice}: 
	 $\color{ForestGreen}\nabla\Upsilon_{(\bullet\graf\bullet)\graf\tau}(u_1)+\nabla\Upsilon_{(\bullet\graf\bullet)\graf\tau}(u_2)$.
	Collecting all contributions, we obtain the desired formula for $B_{z_1z_2}^{[3]}(\tau)$:
	\begin{align}
	&B_{z_1z_2}^{[3]}(\tau)=\int_0^1\nabla\Upsilon_{\tau}(u_0) \, z^{[3]} \d t_0 +\nonumber
	\\ & + \int_{[0,1]_\geq^2}\left(\nabla \Upsilon_{\bullet\graf\tau}(u_1)-\nabla\Upsilon_{\tau}(u_0) \, \nabla\Upsilon_{\bullet}(u_1)\right)\,z^{[2]}\, \X(\bullet)\d t_1\d t_0\nonumber
	\\&+\int_{[0,1]_\geq^2}{\color{blue}\left( \nabla\Up_{\<1> \curvearrowright\tau}(u_1)- \nabla\Up_{\tau}(u_0)\, \nabla\Up_{\<1>}(u_1)\right)} \, z^{[1]} \, \X(\,\<1>\,)\d t_0\d t_1\nonumber
	\\&+\int_{[0,1]_\geq^3}\left({\color{magenta}\nabla\Upsilon_{\bullet\graf(\bullet\graf\tau)}(u_2)}
	-{\color{ForestGreen} \nabla\Upsilon_{(\bullet\graf\bullet)\graf\tau}(u_1)}-{\color{ForestGreen} \nabla\Upsilon_{(\bullet\graf\bullet)\graf\tau}(u_2)}{\color{violet}-\nabla\Upsilon_{\bullet\graf\tau}(u_1)\, \nabla\Upsilon_{\bullet}(u_2) }\right.\nonumber\\
	&\quad+\left.{\color{brown}\nabla\Upsilon_{\tau}(u_0) \, \nabla\Upsilon_{\bullet\graf\bullet}(u_1)}
	{\color{violet}+\nabla\Upsilon_{\tau}(u_0) \, \nabla\Upsilon_{\bullet}(u_1)\, \nabla\Upsilon_{\bullet}(u_2)} \right)\,\X(\bullet\bullet)\,z^{[1]}\d t_2\d t_1\d t_0\,.\label{Eq: B^3}
	\end{align}
	\begin{remark}\label{rem:problems}
	This example showcases the problems that one faces when trying to give a systematic formula account for this construction:
	\begin{enumerate}
		\item an explicit description of all terms in the Guin-Oudom expression $(\tau_1\cdots\tau_k)\graf\tau$ is needed,
		\item the green terms in \eqref{Eq: B^3} have the same form in the Guin-Oudom sum, 
		yet they are evaluated at different points ($u_1$ and $u_2$ in this case); finding the 
		correct evaluation point for each derivative in expressions like \eqref{Eq: B^3} is far from trivial.
	\end{enumerate}
	A solution to both problems is proposed in the forthcoming sections \ref{Sec: Binary trees and Guin-Oudom grafting} and
	\ref{Sec: evaluation points}. 
	\end{remark}
	
	\subsection{Binary trees and Guin-Oudom grafting}\label{Sec: Binary trees and Guin-Oudom grafting}
	A crucial step in the construction of the generalized Taylor expansions is to derive a suitable and explicit representation of the contributions arising from the Guin–Oudom extension of the grafting operation to forests of rooted trees. The latter is defined inductively and hides the actual shape of each term, hence preventing us from making important identifications. Thus, in this section, we provide an alternative representation of the grafting of forests that, to the best of our knowledge, is not present in the literature and has its own relevance. 
	
	We recall that the grafting operation $\graf$ introduced in Definition \ref{Def: grafting} can be extended to forests according to the aforementioned Guin-Oudom prescription, see \cite{oudom2008lie}, as per \eqref{Eq: GU definition}.
	As an example, consider the grafting of the forest $w=\tau_1 \tau_2\tau_3$. An iterative application of \eqref{Eq GU Leibniz} and\eqref{Eq: GU definition} leads to
	\begin{align*}
		(\tau_1\tau_2\tau_3)\graf\tau=&\tau_1\graf(\tau_2\graf(\tau_3\graf \tau))-\tau_1\graf((\tau_2\graf\tau_3)\graf\tau)\\
		&-(\tau_1\graf\tau_2)\graf(\tau_3\graf\tau)+((\tau_1\graf\tau_2)\graf\tau_3)\graf\tau\\
		&-\tau_2\graf((\tau_1\graf\tau_3)\graf\tau)+(\tau_2\graf(\tau_1\graf\tau_3))\graf\tau\,.
	\end{align*}
	In what follows, we need to encode each term on the right-hand side of this expression.
	This problem is well known in the literature on enumerative combinatorics for non-commutative operations and is usually addressed by adopting a representation in term of binary trees \cite{stanley1999enumerative,chapoton2007operads}. However, as an additional complication, the order of trees within the forest is not preserved, while only specific permutations of the latter are present. 
	
	Inspired by the extensive literature cited above, we adopt an analogous point of view. We introduce \emph{planar binary trees} with a decoration on 
	leaves, in such a way that for example we can encode
	\[
	\tau_1\graf\tau_2 \longleftrightarrow \begin{tikzpicture}[scale=0.2,baseline=0.1cm]
			\node at (0,0) [dot] (root2) {};
			\node at (-3,3) [empty] (centerl) {$1$};
			\node at (3,3) [empty] (centerr) {$2$};
			\draw[kernel] (centerl) to
			node [sloped,below] {\small }     (root2);
			\draw[kernel] (centerr) to
			node [sloped,below] {\small }     (root2);
		\end{tikzpicture}, \qquad 
		(\tau_1\graf\tau_2)\graf\tau_3 \longleftrightarrow 
				\begin{tikzpicture}[scale=0.2,baseline=0.1cm]
			\node at (0,0) [dot] (root) {};
			\node at (-3,3) [dot] (root1) {};
			\node at (-6,6) [empty] (topl) {$1$};
			\node at (0,6) [empty] (topr) {$2$};
			\node at (3,3) [empty] (root2) {$3$};
			\draw[kernel] (topl) to
			node [sloped,below] {\small }     (root1);
			\draw[kernel] (topr) to
			node [sloped,below] {\small }     (root1);
			\draw[kernel] (root1) to
			node [sloped,below] {\small }     (root);
			\draw[kernel] (root2) to
			node [sloped,below] {\small }     (root);
		\end{tikzpicture}\,,
	\]
	\[
	(\tau_1\graf\tau_2)\graf(\tau_3\graf\tau_0) \longleftrightarrow \!\!
\begin{tikzpicture}[scale=0.2,baseline=0.1cm]
				\node at (0,0) [dot] (root) {};
				\node at (-4,3) [dot] (root1) {};
				\node at (-7,6) [empty] (topl) {$1$};
				\node at (-1,6) [empty] (topr) {$2$};
				\node at (4,3) [dot] (root2) {};
				\node at (1,6) [empty] (centerl) {$3$};
				\node at (7,6) [empty] (centerr) {$0$};
				\draw[kernel] (topl) to
				node [sloped,below] {\small }     (root1);
				\draw[kernel] (topr) to
				node [sloped,below] {\small }     (root1);
				\draw[kernel] (centerl) to
				node [sloped,below] {\small }     (root2);
				\draw[kernel] (centerr) to
				node [sloped,below] {\small }     (root2);
				\draw[kernel] (root1) to
				node [sloped,below] {\small }     (root);
				\draw[kernel] (root2) to
				node [sloped,below] {\small }     (root);
			\end{tikzpicture}\,, \quad
	(\tau_3\graf(\tau_1\graf\tau_2))\graf\tau_0 \longleftrightarrow \!\!
		\begin{tikzpicture}[scale=0.2,baseline=0.1cm]
		\node at (-3,9) [empty] (3) {$1$};
		\node at (3, 9) [empty] (2) {$2$};
		\node at (0,0) [dot] (root) {};
		\node at (-3,3) [dot] (root1) {};
		\node at (-6,6) [empty] (topl) {$3$};
		\node at (0,6) [dot] (topr) {};
		\node at (3,3) [empty] (root2) {$0$};
		\draw[kernel] (topl) to
		node [sloped,below] {\small }     (root1);
		\draw[kernel] (topr) to
		node [sloped,below] {\small }     (root1);
		\draw[kernel] (root1) to
		node [sloped,below] {\small }     (root);
		\draw[kernel] (root2) to
		node [sloped,below] {\small }     (root);
		\draw[kernel] (3) to
		node [sloped,below] {\small }     (topr);
		\draw[kernel] (2) to
		node [sloped,below] {\small }     (topr);
	\end{tikzpicture}\,.
	\]
	\begin{definition}[Planar binary trees] \label{def:planar}
	A \emph{planar binary} rooted tree $(V,E,\mathfrak{v})$ is a rooted tree such that
	every vertex has either two or zero children and endowed with
	\begin{itemize}
	\item a left-to-right total order between children of every vertex
	\item a decoration $\mathfrak{v}:L\to\N$, where $L\subset V$ are the leaves of the tree.
	\end{itemize}
	We denote by $\cT_{b}$ the set of \emph{planar binary} rooted trees. 
		If $k\in\N$ and $I\in(\N\setminus\{0\})^k$ is a string of $k$ positive integers, then we denote by $\cT_b^{I}$ the set of $\rho\in\cT_{b}$ with 
		$k+1$ leaves and such that 
		\begin{itemize}
		\item the rightmost leaf has $\mathfrak{v}$-decoration $0$,
		\item the $i$-th leaf from the {\it right} receives as $\mathfrak{v}$-decoration the $i$-th element of $I$.
\end{itemize}		
 When $I=\emptyset$, $\cT_b^I=\{\bullet_0\}$.		We also adopt the notation
		\[
		I_k:=(1,\ldots,k)\in\N^k
		\]
		and, for $\Sigma_k$ the set of permutations of $\{1,\ldots,k\}$ and for $p\in\Sigma_k$, $p(I_k):=(p(1),\ldots,p(k))$. 
		
		Finally, in analogy with the grafting of rooted trees on a new common root,  we denote by $[\cdot,\cdot]:\cT_b\times\cT_b\to\cT_b$ the operation of grafting an ordered pair of planary binary trees on a new root, with the first tree on the
	left and the second on the right. 
	\end{definition}
	
	When we draw a tree in $\cT_b$ on the plane, we respect the left-to-right order of its vertices. For example, unlike what happens in 
$\cT$, we have $\<301>\ne\<103>$ in $\cT_b$.
	If $k=3$ and $I=(2,1,3)$, an example of planar binary tree lying in $\cT_b^{I}$ is
	\[
	\begin{tikzpicture}[scale=0.2,baseline=0.1cm]
		\node at (-3,9) [empty] (3) {$1$};
		\node at (3, 9) [empty] (2) {$2$};
		\node at (0,0) [dot] (root) {};
		\node at (-3,3) [dot] (root1) {};
		\node at (-6,6) [empty] (topl) {$3$};
		\node at (0,6) [dot] (topr) {};
		\node at (3,3) [empty] (root2) {$0$};
		\draw[kernel] (topl) to
		node [sloped,below] {\small }     (root1);
		\draw[kernel] (topr) to
		node [sloped,below] {\small }     (root1);
		\draw[kernel] (root1) to
		node [sloped,below] {\small }     (root);
		\draw[kernel] (root2) to
		node [sloped,below] {\small }     (root);
		\draw[kernel] (3) to
		node [sloped,below] {\small }     (topr);
		\draw[kernel] (2) to
		node [sloped,below] {\small }     (topr);
	\end{tikzpicture} =[ [\bullet_3,[\bullet_1,\bullet_2]],\bullet_0]\,.
	\]
	The notion of rightmost and leftmost leaf must be understood in the sense of the latter representation: in this example, the leafs with decoration $0$ and $3$, respectively.
	
	Our interest in planar binary trees resides in the possibility of using them to represent the composition of noncommutative and nonassociative operations, such as the grafting of forests. This link is spelled out in the following definition. 
	\begin{definition}\label{Def: link binary to rooted}
		For $(\tau_k)_{k\in\N}$,
		we define a map $G:\cT_b\to\langle \cT\rangle $ that to any planar binary tree in $\cT_b$ associates the corresponding grafting of rooted trees according to the following rule: 
		\begin{equation}
			G(\bullet_{i})=\tau_i\,,\quad\forall i\in\N\,, \qquad
			G([\rho_1,\rho_2]):=G(\rho_1)\graf G(\rho_2)\,.\label{Eq: definition G}
		\end{equation}
		When we are only interested in trees lying in $\cT_b^I$, then only $(\tau_k)_{k\in I\sqcup\{0\}}$ need to be specified.
	\end{definition}
	On account of Definition \ref{Def: link binary to rooted},
	when two leaves of a planar binary tree share a common root, they are placeholders of a grafting $(\tau_i\graf\tau_j)$ within the system of parenthesis representing a generic contribution to $(\tau_1\cdots\tau_k)\graf\tau$. As these objects play a special role in the sequel, we assign them a specific name.
	\begin{definition}\label{Def: eff leaves}
		The \emph{effective cherries} 
		of a decorated binary tree $\rho\in\cT_b$ are the pairs of edges that share a common parent such that
		\begin{enumerate}
			\item either both edges end with a leaf whose decoration is not $0$
			\item or the left edge stems from the root and ends with a leaf.
		\end{enumerate}
	\end{definition}
		\begin{example}\label{ex:3.7}
		In the following tree
		
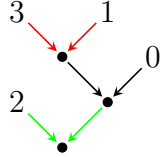
\[
		\begin{tikzpicture}[scale=0.2,baseline=0.1cm]
				\node at (0,0) [dot] (root) {};
				\node at (-3,3) [empty] (root1) {$2$};
				\node at (3,3) [dot] (root2) {};
				\node at (0,6) [dot] (centerl) {};
				\node at (-3,9) [empty] (topl) {$3$};
				\node at (3,9) [empty] (topr) {$1$};
				\node at (6,6) [empty] (centerr) {$0$};
				\draw[kernel,red] (topl) to
				node [sloped,below] {\small }     (centerl);
				\draw[kernel,red] (topr) to
				node [sloped,below] {\small }     (centerl);
				\draw[kernel] (centerl) to
				node [sloped,below] {\small }     (root2);
				\draw[kernel] (centerr) to
				node [sloped,below] {\small }     (root2);
				\draw[kernel,green] (root1) to
				node [sloped,below] {\small }     (root);
				\draw[kernel,green] (root2) to
				node [sloped,below] {\small }     (root);
			\end{tikzpicture}
\]
the red pair of edges is an effective cherry by rule 1 of Definition \ref{Def: eff leaves}; the green pair is an effective cherry by rule 2;
the black pair is not an effective cherry because the right edge ends with a leaf decorated by $0$ and the left edge does not end with a leaf, and therefore it violates twice rule 1.
		\end{example}

	\begin{definition}\label{Def: operations binary trees1}
		Let $\rho\in\cT_b$. We define the family of binary trees $\cR^-(\rho)\subset\cT_b$ obtained from $\rho$ by the following algorithm. For any effective cherry of $\rho$:
		\begin{enumerate}
			\item we contract the effective cherry to its parent vertex;
			\item if the effective cherry contains two leaves, after contraction it becomes a new leaf which inherits the rightmost decoration of the pair;
			\item if, instead, the effective cherry contains only one leaf grafted from the left on the root, after contraction the new vertex is not a leaf and does not require any decoration. 
		\end{enumerate}
	\end{definition}
	
	\begin{example}
		We consider the same planar binary tree as in Example \ref{ex:3.7}. If we contract the red effective cherry, then by rule 2 of Definition 
		\ref{Def: operations binary trees1} we obtain the tree
		
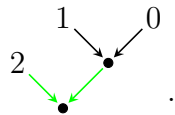
\[
			\begin{tikzpicture}[scale=0.2,baseline=0.1cm]
				\node at (0,0) [dot] (root) {};
				\node at (-3,3) [empty] (root1) {$2$};
				\node at (3,3) [dot] (root2) {};
				\node at (0,6) [empty] (centerl) {$1$};
				\node at (6,6) [empty] (centerr) {$0$};
				\draw[kernel] (centerl) to
				node [sloped,below] {\small }     (root2);
				\draw[kernel] (centerr) to
				node [sloped,below] {\small }     (root2);
				\draw[kernel,green] (root1) to
				node [sloped,below] {\small }     (root);
				\draw[kernel,green] (root2) to
				node [sloped,below] {\small }     (root);
			\end{tikzpicture}.
		\]
		If instead we contract the green effective cherry, then by rule 3 we obtain the tree
		
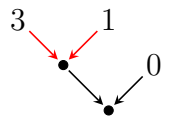
\[
				\begin{tikzpicture}[scale=0.2,baseline=0.1cm]
			\node at (0,0) [dot] (root) {};
			\node at (-3,3) [dot] (root1) {};
			\node at (-6,6) [empty] (topl) {$3$};
			\node at (0,6) [empty] (topr) {$1$};
			\node at (3,3) [empty] (root2) {$0$};
			\draw[kernel,red] (topl) to
			node [sloped,below] {\small }     (root1);
			\draw[kernel,red] (topr) to
			node [sloped,below] {\small }     (root1);
			\draw[kernel] (root1) to
			node [sloped,below] {\small }     (root);
			\draw[kernel] (root2) to
			node [sloped,below] {\small }     (root);
		\end{tikzpicture}\,.
		\]
Hence $\cR^-(\rho)$ only contains these two trees.
	\end{example}
	\begin{definition}\label{Def: operations binary trees2}
		Let $k\in\N$, $i_1,\ldots,i_k$ be distinct positive integers and set $I:=(i_1,\ldots,i_k)$. Consider $\rho\in\cT_b^I$ and $v$ a leaf of $\rho$ with $\mathfrak{v}(v)\ne 0$. For $j\in\N$, we introduce the tree $R^+_{j,\mathfrak{v}(v)}(\rho)\in\cT_b$ 
		obtained from $\rho$ by replacing $v$ with an effective cherry decorated by $(\,j,\mathfrak{v}(v))$. 
		
		We also define the set $\cR_j^+(\rho)\subset\cT_b$ containing all planar binary trees $R^+_{j,\mathfrak{v}(v)}(\rho)$, $v\in L(\rho)$ such that
		$\mathfrak{v}(v)\ne 0$, plus the binary tree $[\bullet_{j},\rho]$, see \eqref{eq: binary blocks}.
	\end{definition}
	
	\begin{example}
		We consider the binary tree 
				\begin{align*}
			\rho=	\begin{tikzpicture}[scale=0.2,baseline=0.1cm]
				\node at (0,0) [dot] (root) {};
				\node at (-3,3) [empty] (root1) {$2$};
				\node at (3,3) [dot] (root2) {};
				\node at (0,6) [empty] (centerl) {$1$};
				\node at (6,6) [empty] (centerr) {$0$};
				\draw[kernel] (centerl) to
				node [sloped,below] {\small }     (root2);
				\draw[kernel] (centerr) to
				node [sloped,below] {\small }     (root2);
				\draw[kernel] (root1) to
				node [sloped,below] {\small }     (root);
				\draw[kernel] (root2) to
				node [sloped,below] {\small }     (root);
			\end{tikzpicture}.
		\end{align*}
			For $j=3$ we have 
		\begin{align*}
			[\bullet_3,\rho]=		
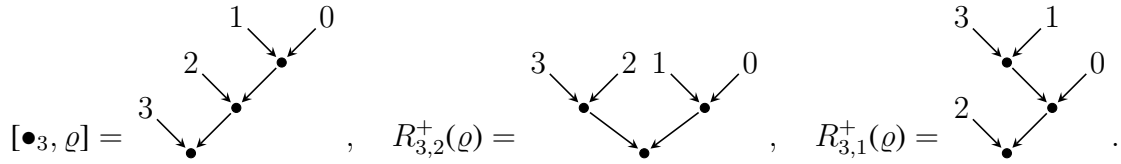
\begin{tikzpicture}[scale=0.2,baseline=0.1cm]
				\node at (0,0) [dot] (root) {};
				\node at (-3,3) [empty] (root1) {$3$};
				\node at (3,3) [dot] (root2) {};
				\node at (0,6) [empty] (centerl) {$2$};
				\node at (6,6) [dot] (centerr) {};
				\node at (3,9) [empty] (topl) {$1$};
				\node at (9,9) [empty] (topr) {$0$};
				\draw[kernel] (topr) to
				node [sloped,below] {\small }     (centerr);
				\draw[kernel] (topl) to
				node [sloped,below] {\small }     (centerr);
				\draw[kernel] (centerl) to
				node [sloped,below] {\small }     (root2);
				\draw[kernel] (centerr) to
				node [sloped,below] {\small }     (root2);
				\draw[kernel] (root1) to
				node [sloped,below] {\small }     (root);
				\draw[kernel] (root2) to
				node [sloped,below] {\small }     (root);
			\end{tikzpicture}\,, \quad
			R^+_{3,2}(\rho)=		\begin{tikzpicture}[scale=0.2,baseline=0.1cm]
				\node at (0,0) [dot] (root) {};
				\node at (-4,3) [dot] (root1) {};
				\node at (-7,6) [empty] (topl) {$3$};
				\node at (-1,6) [empty] (topr) {$2$};
				\node at (4,3) [dot] (root2) {};
				\node at (1,6) [empty] (centerl) {$1$};
				\node at (7,6) [empty] (centerr) {$0$};
				\draw[kernel] (topl) to
				node [sloped,below] {\small }     (root1);
				\draw[kernel] (topr) to
				node [sloped,below] {\small }     (root1);
				\draw[kernel] (centerl) to
				node [sloped,below] {\small }     (root2);
				\draw[kernel] (centerr) to
				node [sloped,below] {\small }     (root2);
				\draw[kernel] (root1) to
				node [sloped,below] {\small }     (root);
				\draw[kernel] (root2) to
				node [sloped,below] {\small }     (root);
			\end{tikzpicture}\,, \quad
			R^+_{3,1}(\rho)	=	\begin{tikzpicture}[scale=0.2,baseline=0.1cm]
				\node at (0,0) [dot] (root) {};
				\node at (-3,3) [empty] (root1) {$2$};
				\node at (3,3) [dot] (root2) {};
				\node at (0,6) [dot] (centerl) {};
				\node at (-3,9) [empty] (topl) {$3$};
				\node at (3,9) [empty] (topr) {$1$};
				\node at (6,6) [empty] (centerr) {$0$};
				\draw[kernel] (topl) to
				node [sloped,below] {\small }     (centerl);
				\draw[kernel] (topr) to
				node [sloped,below] {\small }     (centerl);
				\draw[kernel] (centerl) to
				node [sloped,below] {\small }     (root2);
				\draw[kernel] (centerr) to
				node [sloped,below] {\small }     (root2);
				\draw[kernel] (root1) to
				node [sloped,below] {\small }     (root);
				\draw[kernel] (root2) to
				node [sloped,below] {\small }     (root);
			\end{tikzpicture}.
		\end{align*}
		Moreover $\cR_3^+(\rho)$ is the set that contains the trees in the latter display.
		The sets $\cR^-([\bullet_3,\rho])$ and $\cR^-(R^+_{3,2}(\rho))$ contain only $\rho$, while $\cR^-(R^+_{3,1}(\rho))$ contains $\rho$ and
		
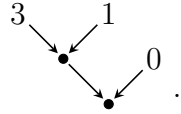
\[
		\begin{tikzpicture}[scale=0.2,baseline=0.1cm]
			\node at (0,0) [dot] (root) {};
			\node at (-3,3) [dot] (root1) {};
			\node at (-6,6) [empty] (topl) {$3$};
			\node at (0,6) [empty] (topr) {$1$};
			\node at (3,3) [empty] (root2) {$0$};
			\draw[kernel] (topl) to
			node [sloped,below] {\small }     (root1);
			\draw[kernel] (topr) to
			node [sloped,below] {\small }     (root1);
			\draw[kernel] (root1) to
			node [sloped,below] {\small }     (root);
			\draw[kernel] (root2) to
			node [sloped,below] {\small }     (root);
		\end{tikzpicture}\,.
		\]
	\end{example}

	The operation that characterises $\mathcal{R}^-(\cdot)$ allows to introduce a key combinatorial coefficient associated to binary trees. 
	\begin{definition}\label{redchain}
		We inductively define the map $C:\cT_b\to \N$ by setting $C(\bullet)=1$ and
		\begin{align}\label{Eq: C}
			C(\rho)=\sum_{\rho'\in\cR^-(\rho)}C(\rho')\,,
		\end{align}
		where the set of binary trees $\cR^-(\rho)$ has been introduced in Definition \ref{Def: operations binary trees2}.
	\end{definition}
	
	
	Here is the main result of this section, which concerns an explicit description of the contributions to the grafting of forests on trees as prescribed by the inductive Guin-Oudom construction. To the best of our knowledge, this is the first time a formula of this kind has been derived in this setting.  
	\begin{lemma}\label{lem: GUBinary}
		For all $\tau_0,\ldots, \tau_k\in\cT$:
		\begin{equation}\label{Eq: GU binary}
			(\tau_k\cdots \tau_1)\curvearrowright\tau_0= \frac1{k!} \sum_{p\in\Sigma_k} \sum_{\rho\in\cT_b^{p(I_k)}} C(\rho)
			\, (-1)^{k-m_R(\rho)} \, G(\rho)\,,
		\end{equation}
		where $G$ is the mapping from planar binary trees to rooted trees of Definition \ref{Def: link binary to rooted}, $m_R(\rho)\in\N$ denotes the distance of the rightmost leaf from the root and $C(\rho)$ is the combinatorial factor defined in \eqref{Eq: C}.
	\end{lemma}
	\begin{proof}
		We proceed by induction on the number of trees composing the forest. For $k=1$, \eqref{Eq: GU definition} is trivially satisfied since the only planar binary tree in $\cT_b^{I_1}$ is $[\bullet_1,\bullet_0]$, for which $G([\bullet_1,\bullet_0])=\tau_1\graf\tau$, $C([\bullet_1,\bullet_0])=1$ and $m_R([\bullet_1,\bullet_0])=1$. Then we assume that \eqref{Eq: GU binary} holds true for forest of at most $k-1$ rooted trees. Let $\tau_1,\ldots, \tau_k\in\cT$. By definition
		\begin{align*}
			(\tau_k\cdots \tau_1)\curvearrowright\tau_0 &=\tau_k\curvearrowright((\tau_{k-1}\cdots \tau_1)\curvearrowright\tau_0)\nonumber
			-\sum_{i=1}^{k-1}(\tau_{k-1}\cdots(\tau_k\curvearrowright \tau_{i})\cdots \tau_1)
			\curvearrowright\tau_0 \,.\label{Eq: GU}
		\end{align*}
		Using the inductive hypothesis on the first term and the linearity of $\graf$ we write
		\begin{align*}
			\tau_k\curvearrowright((\tau_{k-1}\cdots \tau_1)\curvearrowright\tau_0)
			&= \frac1{(k-1)!} \sum_{p\in\Sigma_{k-1}} \sum_{\rho\in\cT_b^{p(I_{k-1})}} C(\rho)(-1)^{k-1-m_R(\rho)} \, \tau_k\curvearrowright G(\rho)\\
			&=  \frac1{(k-1)!} \sum_{p\in\Sigma_{k-1}} \sum_{\rho\in\cT_b^{p(I_{k-1})}} C(\rho) \, (-1)^{k-m_R([\bullet_k,\rho])} \, G([\bullet_k,\rho])\, .
		\end{align*}
		The identification $\tau_k\curvearrowright G(\rho)=G([\bullet_k,\rho])$ follows from \eqref{Eq: definition G}. 
		We can now use once again the inductive hypothesis on the second term: for all $i=1,\ldots,k-1$ and recalling the action of $R^+_{k,i}$, see Definition \ref{Def: operations binary trees2}, 
		\begin{align*}
			& (\tau_{k-1}\cdots(\tau_k\curvearrowright \tau_{i})\cdots \tau_1)
			\curvearrowright\tau_0 \\
			&\hspace{2cm}= \frac1{(k-1)!} \sum_{p\in\Sigma_{k-1}} \sum_{\rho'\in\cT_b^{p(I_{k-1})}} C(\rho) \, (-1)^{k-1-m_R(\rho)} \, G(R^+_{k,i}(\rho))\,.
		\end{align*}
		As a result
		\begin{align*}
			(\tau_k\cdots \tau_1)\curvearrowright\tau_0&= \frac1{(k-1)!} \sum_{p\in\Sigma_{k-1}} \sum_{\rho\in\cT_b^{p(I_{k-1})}} 
			C(\rho) \, (-1)^{k-m_R([\bullet_k,\rho])} \, G([\bullet_k,\rho]) \\
			&\qquad-\frac1{(k-1)!}\sum_{i=1}^{k-1} \sum_{p\in\Sigma_{k-1}} \sum_{\rho\in\cT_b^{p(I_{k-1})}}  C(\rho) \, (-1)^{k-1-m_R(\rho)} \, G\left(R^+_{k,i}(\rho)\right)\\
			&= \frac1{(k-1)!} \sum_{p\in\Sigma_{k-1}} \sum_{\rho\in\cT_b^{p(I_{k-1})}} \sum_{\rho'\in\cR^+_k(\rho)} (-1)^{k-m_R(\rho')} \, C(\rho) \, G(\rho')\,.
		\end{align*}
		In the last line we used that $[\bullet_k,\rho]$ and $R^+_{k,i}(\rho)$ span all the contributions to $\cR^+_k(\rho)$ according to Definition \ref{Def: operations binary trees2}. Observe that $R^+_{k,i}(\rho)$ does not affect the depth of $\tau_0$ in $\rho$, so that $m_R(\rho)=m_R(R^+_{k,i}(\rho))$. 	
		Now we note that the symmetric nature of the forest product entails
		\[
		(\tau_k\cdots \tau_1)\curvearrowright\tau_0
		= \frac1{k!}\sum_{q\in\Sigma_k} (\tau_{q(k)}\cdots \tau_{q(1)})\curvearrowright\tau_0\,.
		\]
		Hence, iterating the above argument for every $q\in\Sigma_k$ we obtain
		\begin{align*}
			&(\tau_{q(k)}\cdots \tau_{q(1)})\curvearrowright\tau_0 = 
			\frac1{(k-1)!} \sum_{p\in\Sigma_{k-1}} \sum_{\rho\in\cT_b^{q\circ p(I_{k-1})}}\sum_{\rho'\in\cR^+_{q(k)}(\rho)} (-1)^{k-m_R(\rho')} \, C(\rho) \, G(\rho')\,,
		\end{align*}
		and therefore 
		\begin{align*}
			(\tau_k\cdots \tau_1)\curvearrowright\tau_0 &= \frac1{k!}\sum_{q\in\Sigma_k} 
			\frac1{(k-1)!} \sum_{\substack{\bar p\in\Sigma_k\\ \bar p(k)=k}} \sum_{\rho\in\cT_b^{q\circ \bar p(I_{k-1})}}\sum_{\rho'\in\cR^+_{q\circ \bar p(k)}(\rho)} (-1)^{k-m_R(\rho')} \, C(\rho) \, G(\rho')\,
			\\ &= \frac1{k!}\sum_{p\in\Sigma_k} 
			\sum_{\rho\in\cT_b^{p(I_{k-1})}}\sum_{\rho'\in\cR^+_{p(k)}(\rho)} (-1)^{k-m_R(\rho')} \, C(\rho) \, G(\rho')\,.
		\end{align*}
		In order to obtain a sum over $\rho'\in\cT_b^{I_k}$ we remark that there is a bijection between $(p,\rho,\rho')\in \Sigma_k\times
		\cT_b^{p(I_{k-1})}\times\cR^+_{p(k)}(\rho)$ and $(p',\rho',\rho)\in \Sigma_k\times\cT_b^{p'(I_k)}\times \cR^{-}(\rho')$, given as follows:
		\begin{itemize}
			\item given $(p,\rho,\rho')\in \Sigma_k\times\cT_b^{p(I_{k-1})}\times\cR^+_{p(k)}(\rho)$, we know that $\rho\in\cR^{-}(\rho')$ and
			we must define $p'\in \Sigma_k$. Let $v$ be the vertex in $\rho$ to which a new leaf with decoration $p(k)$ is grafted in order to obtain $\rho'$ from $\rho$. If $v$ is the
			root, then we set $p'=p$. Else, $v$ is a leaf with decoration $p(i)\ne 0$; in this case we define $p'(\,j)=p(\,j)$ for $j<i$; $p'(i)=p(k)$; $p'(\,j)=p(\,j-1)$ for $j\in\{i+1,\ldots,k\}$.
			\item given $(p',\rho',\rho)\in \Sigma_k\times\cT_b^{p'(I_k)}\times \cR^{-}(\rho')$, we must define $p$. If $\rho$ is obtained from $\rho'$ by contracting
			an effective cherry at the root, then $p'=p$; else, let $p'(i)$ be the decoration at the left leaf of such effective cherry; then we set 
			$p(\,j)=p'(\,j)$ for $j<i$, $p(k)=p'(i)$ and $p(\,j)=p'(\,j+1)$ for $j\in\{i,\ldots,k-1\}$.
		\end{itemize}
		Then we conclude that
		\begin{align*}
			(\tau_k\cdots \tau_1)\curvearrowright\tau_0 & = \frac1{k!}\sum_{p'\in\Sigma_k} 
			\sum_{\rho'\in\cT_b^{p'(I_k)}} (-1)^{k-m_R(\rho')} \left(\sum_{\rho\in\cR^{-}(\rho')} C(\rho) \right) G(\rho')  
			\\ & = \frac1{k!}\sum_{p'\in\Sigma_k} 
			\sum_{\rho'\in\cT_b^{p'(I_k)}} (-1)^{k-m_R(\rho')} \, C(\rho') \, G(\rho') \, .
		\end{align*}
		The proof is complete. 
	\end{proof}

	\subsection{Combinatorial operations on planar binary trees}

		Every planar binary tree $\rho\in\cT_b$ can be uniquely represented as 
	\begin{equation}\label{eq: binary blocks}
		\rho=[[[\ldots[\bullet_i,\rho_{m_L(\rho)}]\ldots],\rho_2],\rho_1]\,,
	\end{equation}
	where $i\in \N$ is the decoration of the leftmost leaf and $\rho_1,\ldots,\rho_{m_L(\rho)}$ are again binary trees, possibly consisting of a single vertex. 
	The integer $m_L(\rho)\in\N$ is the distance of the leftmost leaf from the root in $\rho$, while $m_R(\rho)\in\N$ denotes the distance of the rightmost leaf from the root.
	\begin{example}
	Consider the planar binary tree
		\begin{align*}
			\rho=
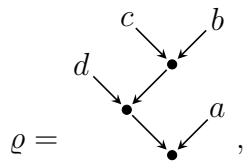
\begin{tikzpicture}[scale=0.2,baseline=0.1cm]
				\node at (-3,9) [empty] (3) {$c$};
				\node at (3, 9) [empty] (2) {$b$};
				\node at (0,0) [dot] (root) {};
				\node at (-3,3) [dot] (root1) {};
				\node at (-6,6) [empty] (topl) {$d$};
				\node at (0,6) [dot] (topr) {};
				\node at (3,3) [empty] (root2) {$a$};
				\draw[kernel] (topl) to
				node [sloped,below] {\small }     (root1);
				\draw[kernel] (topr) to
				node [sloped,below] {\small }     (root1);
				\draw[kernel] (root1) to
				node [sloped,below] {\small }     (root);
				\draw[kernel] (root2) to
				node [sloped,below] {\small }     (root);
				\draw[kernel] (3) to
				node [sloped,below] {\small }     (topr);
				\draw[kernel] (2) to
				node [sloped,below] {\small }     (topr);
			\end{tikzpicture}\,,
		\end{align*}
		for which $m_L(\rho)=2$, $m_R(\rho)=1$. The unique decomposition of the form \eqref{eq: binary blocks} reads
		\[
		\rho_1 = \bullet_a\,, \qquad \rho_2 = \begin{tikzpicture}[scale=0.2,baseline=0.1cm]
			\node at (0,0) [dot] (root2) {};
			\node at (-3,3) [empty] (centerl) {$c$};
			\node at (3,3) [empty] (centerr) {$b$};
			\draw[kernel] (centerl) to
			node [sloped,below] {\small }     (root2);
			\draw[kernel] (centerr) to
			node [sloped,below] {\small }     (root2);
		\end{tikzpicture},
		\qquad \rho = [[\bullet_d,\rho_2],\rho_1]\,.\,
		\]
		Analogously, if
		\begin{align*}
			\rho=
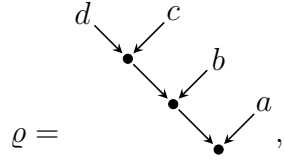
\begin{tikzpicture}[scale=0.2,baseline=0.1cm]
				\node at (-3,9) [empty] (3) {$c$};
				\node at (-9, 9) [empty] (4) {$d$};
				\node at (0,0) [dot] (root) {};
				\node at (-3,3) [dot] (root1) {};
				\node at (-6,6) [dot] (topl) {};
				\node at (0,6) [empty] (topr) {$b$};
				\node at (3,3) [empty] (root2) {$a$};
				\draw[kernel] (topl) to
				node [sloped,below] {\small }     (root1);
				\draw[kernel] (topr) to
				node [sloped,below] {\small }     (root1);
				\draw[kernel] (root1) to
				node [sloped,below] {\small }     (root);
				\draw[kernel] (root2) to
				node [sloped,below] {\small }     (root);
				\draw[kernel] (4) to
				node [sloped,below] {\small }     (topl);
				\draw[kernel] (3) to
				node [sloped,below] {\small }     (topl);
			\end{tikzpicture},
		\end{align*}
		then $m_L(\rho)=3$, $m_R(\rho)=1$ and
		\[
		\rho_1 = \bullet_a\,, \qquad \rho_2 =\bullet_b\,, \qquad \rho_3=\bullet_c\,, 
		\qquad \rho = [[[\bullet_d,\rho_3],\rho_2],\rho_1]\,.
		\]
		Finally if
		\[
		\rho=
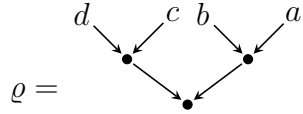
\begin{tikzpicture}[scale=0.2,baseline=0.1cm]
				\node at (0,0) [dot] (root) {};
				\node at (-4,3) [dot] (root1) {};
				\node at (-7,6) [empty] (topl) {$d$};
				\node at (-1,6) [empty] (topr) {$c$};
				\node at (4,3) [dot] (root2) {};
				\node at (1,6) [empty] (centerl) {$b$};
				\node at (7,6) [empty] (centerr) {$a$};
				\draw[kernel] (topl) to
				node [sloped,below] {\small }     (root1);
				\draw[kernel] (topr) to
				node [sloped,below] {\small }     (root1);
				\draw[kernel] (centerl) to
				node [sloped,below] {\small }     (root2);
				\draw[kernel] (centerr) to
				node [sloped,below] {\small }     (root2);
				\draw[kernel] (root1) to
				node [sloped,below] {\small }     (root);
				\draw[kernel] (root2) to
				node [sloped,below] {\small }     (root);
			\end{tikzpicture}
			\]
			then 
			\[
			\rho_1 = \begin{tikzpicture}[scale=0.2,baseline=0.1cm]
				\node at (0,0) [dot] (root) {};
				\node at (-3,3) [empty] (root1) {$b$};
				\node at (3,3) [empty] (root2) {$a$};
				\draw[kernel] (root1) to
				node [sloped,below] {\small }     (root);
				\draw[kernel] (root2) to
				node [sloped,below] {\small }     (root);
			\end{tikzpicture}, \qquad \rho_2 = \bullet_c, \qquad \rho = [[\bullet_d,\rho_2],\rho_1].
			\]
	\end{example}
	
	Let $\rho\in\cT_b$ be represented as in \eqref{eq: binary blocks}: $\rho=[[[\ldots[\bullet_i,\rho_{m_L(\rho)}]\ldots],\rho_2],\rho_1]$.
	If $m_L(\rho)\ge 1$, set 
	\begin{equation}\label{V-}
		V^-(\rho):=[\ldots[\rho_{m_L(\rho)},\rho_{m_L(\rho)-1}],\ldots,\rho_1]\,.
	\end{equation}
	\begin{example} 
		Here are two examples that clarify the action of $V^-$:
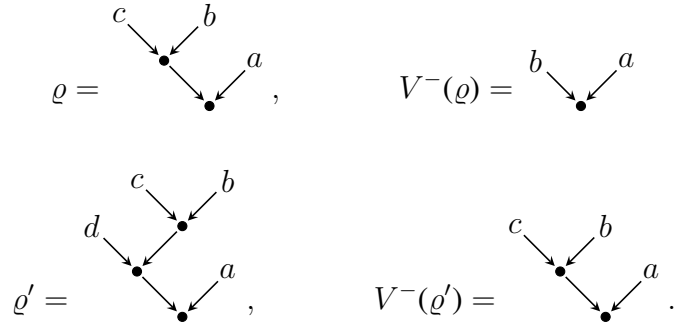
		\begin{align*}
			\rho=\begin{tikzpicture}[scale=0.2,baseline=0.1cm]
				\node at (0,0) [dot] (root) {};
				\node at (-3,3) [dot] (root1) {};
				\node at (-6,6) [empty] (topl) {$c$};
				\node at (0,6) [empty] (topr) {$b$};
				\node at (3,3) [empty] (root2) {$a$};
				\draw[kernel] (topl) to
				node [sloped,below] {\small }     (root1);
				\draw[kernel] (topr) to
				node [sloped,below] {\small }     (root1);
				\draw[kernel] (root1) to
				node [sloped,below] {\small }     (root);
				\draw[kernel] (root2) to
				node [sloped,below] {\small }     (root);
			\end{tikzpicture}\,,\hspace{1.5cm}
			V^-(\rho)=\begin{tikzpicture}[scale=0.2,baseline=0.1cm]
				\node at (0,0) [dot] (root) {};
				\node at (-3,3) [empty] (root1) {$b$};
				\node at (3,3) [empty] (root2) {$a$};
				\draw[kernel] (root1) to
				node [sloped,below] {\small }     (root);
				\draw[kernel] (root2) to
				node [sloped,below] {\small }     (root);
			\end{tikzpicture}
		\end{align*}
		\begin{align*}
			\rho'=\begin{tikzpicture}[scale=0.2,baseline=0.1cm]
				\node at (-3,9) [empty] (3) {$c$};
				\node at (3, 9) [empty] (2) {$b$};
				\node at (0,0) [dot] (root) {};
				\node at (-3,3) [dot] (root1) {};
				\node at (-6,6) [empty] (topl) {$d$};
				\node at (0,6) [dot] (topr) {};
				\node at (3,3) [empty] (root2) {$a$};
				\draw[kernel] (topl) to
				node [sloped,below] {\small }     (root1);
				\draw[kernel] (topr) to
				node [sloped,below] {\small }     (root1);
				\draw[kernel] (root1) to
				node [sloped,below] {\small }     (root);
				\draw[kernel] (root2) to
				node [sloped,below] {\small }     (root);
				\draw[kernel] (3) to
				node [sloped,below] {\small }     (topr);
				\draw[kernel] (2) to
				node [sloped,below] {\small }     (topr);
			\end{tikzpicture}\,,\hspace{1.5cm}
			V^-(\rho')=\begin{tikzpicture}[scale=0.2,baseline=0.1cm]
				\node at (0,0) [dot] (root) {};
				\node at (-3,3) [dot] (root1) {};
				\node at (-6,6) [empty] (topl) {$c$};
				\node at (0,6) [empty] (topr) {$b$};
				\node at (3,3) [empty] (root2) {$a$};
				\draw[kernel] (topl) to
				node [sloped,below] {\small }     (root1);
				\draw[kernel] (topr) to
				node [sloped,below] {\small }     (root1);
				\draw[kernel] (root1) to
				node [sloped,below] {\small }     (root);
				\draw[kernel] (root2) to
				node [sloped,below] {\small }     (root);
			\end{tikzpicture}\,.
		\end{align*}
		In words, $V^-$ removes the cherry containing the leftmost leaf.
	\end{example}

	\begin{definition}\label{Def: trunk}
	For $\rho=[[[\ldots[\bullet_i,\rho_{m_L(\rho)}]\ldots],\rho_2],\rho_1]\in\cT_b$ as in \eqref{eq: binary blocks},
we set
		\[
		P_{\rho,\ell}:= \left\{ \begin{array}{ll} [\ldots[\bullet_i,\rho_{m_L(\rho)}],\ldots, \rho_\ell] \qquad &\text{ if} \quad \ell\in\{1,\ldots, m_L(\rho)\}\,,
			\\ \bullet_i \qquad \qquad &\text{ if} \quad \ell= m_L(\rho)+1\,.
		\end{array}\right.
		\]
	\end{definition}
		Note that $P_{\rho,\ell}$ is a binary tree obtained as follows: on the path from the root to the leftmost leaf, one finds the only node at distance
		$\ell-1$ from the root; then $P_{\rho,\ell}$ is the largest subtree of $\rho$ that has this node as its root.
	\begin{example}\label{ex:4.14}
		Consider the binary tree
		\begin{align*}
			\rho=
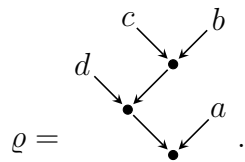
\begin{tikzpicture}[scale=0.2,baseline=0.1cm]
				\node at (-3,9) [empty] (3) {$c$};
				\node at (3, 9) [empty] (2) {$b$};
				\node at (0,0) [dot] (root) {};
				\node at (-3,3) [dot] (root1) {};
				\node at (-6,6) [empty] (topl) {$d$};
				\node at (0,6) [dot] (topr) {};
				\node at (3,3) [empty] (root2) {$a$};
				\draw[kernel] (topl) to
				node [sloped,below] {\small }     (root1);
				\draw[kernel] (topr) to
				node [sloped,below] {\small }     (root1);
				\draw[kernel] (root1) to
				node [sloped,below] {\small }     (root);
				\draw[kernel] (root2) to
				node [sloped,below] {\small }     (root);
				\draw[kernel] (3) to
				node [sloped,below] {\small }     (topr);
				\draw[kernel] (2) to
				node [sloped,below] {\small }     (topr);
			\end{tikzpicture}\,.
		\end{align*}
		Here $m_L(\rho)=2$, so that we have for $\ell\in\{2,3\}$
		\[
		P_{\rho,2}=
		\begin{tikzpicture}[scale=0.2,baseline=0.1cm]
			\node at (0,0) [dot] (root) {};
			\node at (-3,3) [empty] (root1) {$d$};
			\node at (3,3) [dot] (root2) {};
			\node at (0,6) [empty] (centerl) {$c$};
			\node at (6,6) [empty] (centerr) {$b$};
			\draw[kernel] (centerl) to
			node [sloped,below] {\small }     (root2);
			\draw[kernel] (centerr) to
			node [sloped,below] {\small }     (root2);
			\draw[kernel] (root1) to
			node [sloped,below] {\small }     (root);
			\draw[kernel] (root2) to
			node [sloped,below] {\small }     (root);
		\end{tikzpicture},
		\qquad	P_{\rho,3}= \bullet_d\,.
		\]
	\end{example}

	We define an additional operation on binary trees that consists of inserting a new leaf in a specific position of the binary tree, whose usefulness will be clarified below. 
	\begin{definition}[Insertion]\label{Def: insertion}
		For a given $j\in\N$, $\rho\in\cT_b$ and $\ell\in\{1,\ldots, m_L(\rho)+1\}$, we define $\bullet_j\to_\ell\rho\in\cT_b$ as the planar binary tree obtained by inserting $\bullet_j$ in $\rho$ as follows: for $\rho=[[[\ldots[\bullet_i,\rho_{m_L(\rho)}]\ldots],\rho_2],\rho_1]$ as in \eqref{eq: binary blocks},
		\[
		\bullet_j\to_\ell\rho:=[\ldots[[\bullet_j,P_{\rho,\ell}],\rho_{\ell-1}],\ldots, \rho_1]\,, 
		\]
		where $P_{\rho,\ell}$ is defined in Definition \ref{Def: trunk}.
	\end{definition}
	With this definition, $\bullet_j\to_\ell\rho$ is the binary tree obtained from $\rho$ by grafting from the left a new leaf with decoration $j$; the leaf is grafted on a vertex which lies on the
	shortest path from the root to the leftmost leaf, at distance $\ell-1$ from the root.
	Notice that for $\ell=1$ we have $P_{\rho,1}=\rho$ and $\bullet_j\to_1\rho=[\bullet_j,\rho]$.
	
	\begin{example}\label{ex:Prho}
		To make the reader more familiar with the insertion operation introduced in Definition \ref{Def: insertion} we consider the following binary tree
		\begin{align*}
			\rho=\begin{tikzpicture}[scale=0.2,baseline=0.1cm]
				\node at (0,0) [dot] (root) {};
				\node at (-3,3) [dot] (root1) {};
				\node at (-6,6) [empty] (topl) {$c$};
				\node at (0,6) [empty] (topr) {$b$};
				\node at (3,3) [empty] (root2) {$a$};
				\draw[kernel] (topl) to
				node [sloped,below] {\small }     (root1);
				\draw[kernel] (topr) to
				node [sloped,below] {\small }     (root1);
				\draw[kernel] (root1) to
				node [sloped,below] {\small }     (root);
				\draw[kernel] (root2) to
				node [sloped,below] {\small }     (root);
			\end{tikzpicture}\,.
		\end{align*}
		Since $\rho=[[\bullet_{c},\rho_2],\rho_1]$ for $\rho_1=\bullet_a$ and $\rho_2=\bullet_b$, then $m_L(\rho)=2$ and we have three spots for the insertion of a new leaf, indexed by $\ell=1,2,3$. Hence 
		\begin{align*}
			\bullet_j\to_1\rho=\begin{tikzpicture}[scale=0.2,baseline=0.1cm]
				\node at (0,0) [dot] (root) {};
				\node at (-3,3) [empty] (root1) {$j$};
				\node at (3,3) [dot] (root2) {};
				\node at (0,6) [dot] (centerl) {};
				\node at (-3,9) [empty] (topl) {$c$};
				\node at (3,9) [empty] (topr) {$b$};
				\node at (6,6) [empty] (centerr) {$a$};
				\draw[kernel] (topl) to
				node [sloped,below] {\small }     (centerl);
				\draw[kernel] (topr) to
				node [sloped,below] {\small }     (centerl);
				\draw[kernel] (centerl) to
				node [sloped,below] {\small }     (root2);
				\draw[kernel] (centerr) to
				node [sloped,below] {\small }     (root2);
				\draw[kernel] (root1) to
				node [sloped,below] {\small }     (root);
				\draw[kernel] (root2) to
				node [sloped,below] {\small }     (root);
			\end{tikzpicture},
			\qquad
			\bullet_j\to_2\rho=\begin{tikzpicture}[scale=0.2,baseline=0.1cm]
				\node at (-3,9) [empty] (3) {$c$};
				\node at (3, 9) [empty] (2) {$b$};
				\node at (0,0) [dot] (root) {};
				\node at (-3,3) [dot] (root1) {};
				\node at (-6,6) [empty] (topl) {$j$};
				\node at (0,6) [dot] (topr) {};
				\node at (3,3) [empty] (root2) {$a$};
				\draw[kernel] (topl) to
				node [sloped,below] {\small }     (root1);
				\draw[kernel] (topr) to
				node [sloped,below] {\small }     (root1);
				\draw[kernel] (root1) to
				node [sloped,below] {\small }     (root);
				\draw[kernel] (root2) to
				node [sloped,below] {\small }     (root);
				\draw[kernel] (3) to
				node [sloped,below] {\small }     (topr);
				\draw[kernel] (2) to
				node [sloped,below] {\small }     (topr);
			\end{tikzpicture}\,,
			\qquad
			\bullet_j\to_3\rho=\begin{tikzpicture}[scale=0.2,baseline=0.1cm]
				\node at (-3,9) [empty] (3) {$c$};
				\node at (-9, 9) [empty] (4) {$j$};
				\node at (0,0) [dot] (root) {};
				\node at (-3,3) [dot] (root1) {};
				\node at (-6,6) [dot] (topl) {};
				\node at (0,6) [empty] (topr) {$b$};
				\node at (3,3) [empty] (root2) {$a$};
				\draw[kernel] (topl) to
				node [sloped,below] {\small }     (root1);
				\draw[kernel] (topr) to
				node [sloped,below] {\small }     (root1);
				\draw[kernel] (root1) to
				node [sloped,below] {\small }     (root);
				\draw[kernel] (root2) to
				node [sloped,below] {\small }     (root);
				\draw[kernel] (4) to
				node [sloped,below] {\small }     (topl);
				\draw[kernel] (3) to
				node [sloped,below] {\small }     (topl);
			\end{tikzpicture}.
		\end{align*}
		At the same time, we have
		\[
		P_{\rho,1}=\rho, \qquad P_{\rho,2} = \begin{tikzpicture}[scale=0.2,baseline=0.1cm]
			\node at (0,0) [dot] (root) {};
			\node at (-3,3) [empty] (root1) {$c$};
			\node at (3,3) [empty] (root2) {$b$};
			\draw[kernel] (root1) to
			node [sloped,below] {\small }     (root);
			\draw[kernel] (root2) to
			node [sloped,below] {\small }     (root);
		\end{tikzpicture}, \qquad P_{\rho,3}=\bullet_c.
		\]
	\end{example}
	
		The operation of removing an effective cherry of $\rho\in\cT_b^{I_k}$ in all possible ways, that lies at the heart of the combinatorial coefficient $C(\rho)$, induces a reduction operation that stops when the tree is reduced to a single vertex. 
	\begin{definition}[Reduction chains]\label{Def: reduction chain}
		For any $k\in\N$ and $\rho\in\cT _b$ with $k+1$ leaves we define a \emph{reduction chain} as a collection of $k$ binary trees $(\rho^0,\rho^1,\ldots \rho^k)$ such that 
		\begin{equation*}
			\rho^0=\rho\,,\qquad\rho^i\in\cR^-(\rho^{i-1})\,,\qquad i=1,\ldots, k\,,
		\end{equation*}
		see Definition \ref{Def: operations binary trees1} for $\cR^-$.
	\end{definition}
	In a reduction chain $(\rho^0,\rho^1,\ldots \rho^k)$, every $\rho^{i}$ is obtained by removing an effective cherry from $\rho^{i-1}$.
	\begin{example}\label{ex:reducc}
	If we consider
	\[
	\rho=
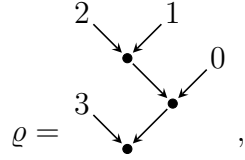
\begin{tikzpicture}[scale=0.2,baseline=0.1cm]
				\node at (0,0) [dot] (root) {};
				\node at (-3,3) [empty] (root1) {$3$};
				\node at (3,3) [dot] (root2) {};
				\node at (0,6) [dot] (centerl) {};
				\node at (-3,9) [empty] (topl) {$2$};
				\node at (3,9) [empty] (topr) {$1$};
				\node at (6,6) [empty] (centerr) {$0$};
				\draw[kernel] (topl) to
				node [sloped,below] {\small }     (centerl);
				\draw[kernel] (topr) to
				node [sloped,below] {\small }     (centerl);
				\draw[kernel] (centerl) to
				node [sloped,below] {\small }     (root2);
				\draw[kernel] (centerr) to
				node [sloped,below] {\small }     (root2);
				\draw[kernel] (root1) to
				node [sloped,below] {\small }     (root);
				\draw[kernel] (root2) to
				node [sloped,below] {\small }     (root);
			\end{tikzpicture}\, ,
\]
then we have two possible reduction chains:
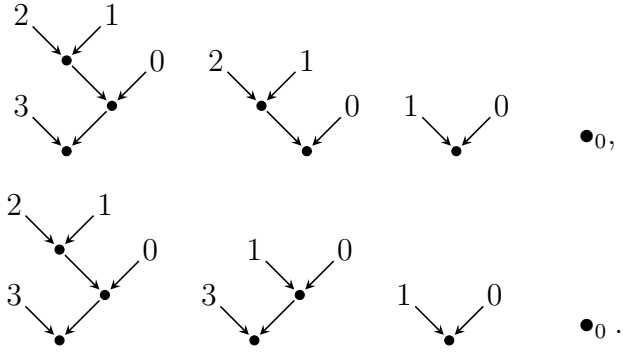
\[
\begin{tikzpicture}[scale=0.2,baseline=0.1cm]
				\node at (0,0) [dot] (root) {};
				\node at (-3,3) [empty] (root1) {$3$};
				\node at (3,3) [dot] (root2) {};
				\node at (0,6) [dot] (centerl) {};
				\node at (-3,9) [empty] (topl) {$2$};
				\node at (3,9) [empty] (topr) {$1$};
				\node at (6,6) [empty] (centerr) {$0$};
				\draw[kernel] (topl) to
				node [sloped,below] {\small }     (centerl);
				\draw[kernel] (topr) to
				node [sloped,below] {\small }     (centerl);
				\draw[kernel] (centerl) to
				node [sloped,below] {\small }     (root2);
				\draw[kernel] (centerr) to
				node [sloped,below] {\small }     (root2);
				\draw[kernel] (root1) to
				node [sloped,below] {\small }     (root);
				\draw[kernel] (root2) to
				node [sloped,below] {\small }     (root);
			\end{tikzpicture} \quad 
			\begin{tikzpicture}[scale=0.2,baseline=0.1cm]
				\node at (0,0) [dot] (root) {};
				\node at (-3,3) [dot] (root1) {};
				\node at (-6,6) [empty] (topl) {$2$};
				\node at (0,6) [empty] (topr) {$1$};
				\node at (3,3) [empty] (root2) {$0$};
				\draw[kernel] (topl) to
				node [sloped,below] {\small }     (root1);
				\draw[kernel] (topr) to
				node [sloped,below] {\small }     (root1);
				\draw[kernel] (root1) to
				node [sloped,below] {\small }     (root);
				\draw[kernel] (root2) to
				node [sloped,below] {\small }     (root);
			\end{tikzpicture} \quad
			 \begin{tikzpicture}[scale=0.2,baseline=0.1cm]
			\node at (0,0) [dot] (root2) {};
			\node at (-3,3) [empty] (centerl) {$1$};
			\node at (3,3) [empty] (centerr) {$0$};
			\draw[kernel] (centerl) to
			node [sloped,below] {\small }     (root2);
			\draw[kernel] (centerr) to
			node [sloped,below] {\small }     (root2);
		\end{tikzpicture} \qquad \bullet_0,
\]
\[
\begin{tikzpicture}[scale=0.2,baseline=0.1cm]
				\node at (0,0) [dot] (root) {};
				\node at (-3,3) [empty] (root1) {$3$};
				\node at (3,3) [dot] (root2) {};
				\node at (0,6) [dot] (centerl) {};
				\node at (-3,9) [empty] (topl) {$2$};
				\node at (3,9) [empty] (topr) {$1$};
				\node at (6,6) [empty] (centerr) {$0$};
				\draw[kernel] (topl) to
				node [sloped,below] {\small }     (centerl);
				\draw[kernel] (topr) to
				node [sloped,below] {\small }     (centerl);
				\draw[kernel] (centerl) to
				node [sloped,below] {\small }     (root2);
				\draw[kernel] (centerr) to
				node [sloped,below] {\small }     (root2);
				\draw[kernel] (root1) to
				node [sloped,below] {\small }     (root);
				\draw[kernel] (root2) to
				node [sloped,below] {\small }     (root);
			\end{tikzpicture} \quad
				\begin{tikzpicture}[scale=0.2,baseline=0.1cm]
				\node at (0,0) [dot] (root) {};
				\node at (-3,3) [empty] (root1) {$3$};
				\node at (3,3) [dot] (root2) {};
				\node at (0,6) [empty] (centerl) {$1$};
				\node at (6,6) [empty] (centerr) {$0$};
				\draw[kernel] (centerl) to
				node [sloped,below] {\small }     (root2);
				\draw[kernel] (centerr) to
				node [sloped,below] {\small }     (root2);
				\draw[kernel] (root1) to
				node [sloped,below] {\small }     (root);
				\draw[kernel] (root2) to
				node [sloped,below] {\small }     (root);
			\end{tikzpicture} \quad 
			\begin{tikzpicture}[scale=0.2,baseline=0.1cm]
			\node at (0,0) [dot] (root2) {};
			\node at (-3,3) [empty] (centerl) {$1$};
			\node at (3,3) [empty] (centerr) {$0$};
			\draw[kernel] (centerl) to
			node [sloped,below] {\small }     (root2);
			\draw[kernel] (centerr) to
			node [sloped,below] {\small }     (root2);
		\end{tikzpicture} \qquad \bullet_0.
\]

	\end{example}
	
	\begin{lemma}\label{reduch}
		The set of reduction chains of $\rho\in\cT _b$ has cardinality $C(\rho)$, see \eqref{Eq: C}.
	\end{lemma}
	\begin{proof}
		By recurrence in \eqref{Eq: C}, we can write
		\[
		C(\rho) = \sum_{\rho^0,\ldots,\rho^k\in\cT_b} \prod_{i=1}^k \un{(\rho^i\in\cR^-(\rho^{i-1}))}\,.
		\]
		This implies the claim.
	\end{proof}
We define
	\begin{equation}\label{eq:<-}
	\{0,1,\ldots,k\}^\ell_<:=\{(i_1,\ldots,i_\ell)\in\N^\ell: 0\le i_1 < \cdots <i_\ell\le k\}.
	\end{equation}
	\begin{definition}\label{Def: reduction steps}
		Given $\rho\in\cT_b^{I_k}$ we fix a reduction chain $(\rho^0,\rho^1,\ldots,\rho^k)$ of $\rho$ as in Definition \ref{Def: reduction chain}.
		Then we set 
		\begin{itemize}
			\item $r^i$ as the maximum decoration of the leaves of $\rho^i$
			(note that $k=r^0\ge r^1\ge\cdots\ge r^k=0$),
			\item $q^{m_L(\rho)+1}:=0$,
			\item For $a\in\{1,\ldots,m_L(\rho)\}$, we define $q^a:=\min\{i>q^{a+1}: r^{i}<r^{i-1}\}$, 
			namely the smallest $i>q^{a+1}$ such that the maximum decoration of $\rho^i$ is strictly less than the maximum decoration of $\rho^{i-1}$,
			\item for $a\in\{1,\ldots,m_L(\rho)\}$, $p^a:=k-q^a$,
		\end{itemize}
		Finally, for $\boldsymbol{j}_{m_L(\rho)}\in\{0,1,\ldots,k-1\}^{m_L(\rho)}_<$, 
		define $H_{\rho}(\,\bj_{m_L(\rho)})$ as the number of 
	reduction chains $(\rho,\rho^1,\ldots,\rho^{k})$ of $\rho\in\cT_b^{I_k}$ such that
	$\bj_{m_L(\rho)} = (p^1,\ldots,p^{m_L(\rho)})$. 
	\end{definition}
	
	In this definition it is important that the decoration of the leaves are ordered from right to left (since $\rho\in\cT_b^{I_k}$),
	and therefore the maximal decoration is always on the leftmost leaf. Given a reduction chain $(\rho^0,\rho^1,\ldots,\rho^k)$, 
	$1\le q^{m_L(\rho)}<\cdots<q^1\le k$
	is the sequence of the first $m_L(\rho)$ steps when the maximal decoration strictly decreases, and $p^a=k-q^a$.
		
	\begin{example} In the first reduction chain of Example \ref{ex:reducc}, the sequence of maximum decorations is 
		$r^0=3$, $r^1=2$, $r^2=1$, $r^3=0$; $m_L(\rho)=1$; then $q^1=1$, and $p^1=2$.
	 In the second reduction chain, we have $r^0=r^1=3$, $r^2=1$, $r^3=0$; then $q^1=2$, and $p^1=1$. We conclude that
	 $H_\rho(1)=H_\rho(2)=1$. See Example \ref{ex:redduc2} for more on reduction chains.
	\end{example}

	\subsection{Evaluation points}\label{Sec: evaluation points}
	In Section \ref{Sec: Binary trees and Guin-Oudom grafting} we found a description of the Guin-Oudom expression $(\tau_1\cdots\tau_k)\graf\tau$,
	which solves the first problem stated in Remark \ref{rem:problems}. Now we need to address the second problem highlighted there:
	how to predict the evaluation points for each derivative in expressions like \eqref{Eq: B^3} for $B^{[m+1]}(\tau)$.
	
	In Definition \ref{Def: J} we introduce multi-indices $J_{\rho,\ell}:\{0,1,\ldots,k\}^\ell_<\to\N$ over the ordered sets of integers defined in \eqref{eq:<-}.
	Recalling the notation introduced in the two examples \eqref{Eq: B^2}-\eqref{Eq: B^3} at the beginning of Section 
	\ref{Sec: Lipschitz elementary differentials}, the desired expression for the remainders $B^{[m+1]}$, see \eqref{eq:e(A)}, will involve 
	$\ell$ gradients evaluated on $(u_{i_1},\ldots,u_{i_\ell})$ and then multiplied

	by the coefficient $J_{\rho,\ell}(i_1,\ldots,i_\ell)$ as $\ell$ varies. The precise construction is again based on planar binary trees.
		\begin{definition}\label{Def: J}
		For any $\rho\in\cT_b$ with $k+1$ leaves and $\ell\in\{1,\ldots,m_L(\rho)+1\}$ we define recursively a multi-index  
		$J_{\rho,\ell}:\{0,1,\ldots,k\}^\ell_<\to\N$ as follows: 
		\begin{itemize} 
		\item $J_{\bullet_i,1}=\un{(0)}$ and, given 
		$J_{V^-(\rho),m_L(\rho)}: \{0,\ldots,k-1\}^{m_L(\rho)}_<\to\N$, recall \eqref{V-},
			\item for $j\in\{0,\ldots,k\}$
			\[
			J_{\rho,1}(\,j) := \sum_{i_1<\cdots <i_{m_L(\rho)}<k} 
			\un{(i_1<j )}\, J_{V^-(\rho),m_L(\rho)}(i_1,\ldots,i_{m_L(\rho)})\,,
			\]
			\item for $\ell\in\{2,\ldots,m_L(\rho)\}$, $\bj_\ell=(\,j_1,\ldots,j_{\ell})\in\{0,\ldots,k\}^\ell_<$:
			\begin{equation}\label{Eq: second case multi}
			J_{\rho,\ell}(\,\bj_\ell) := \sum_{i_{\ell}<\cdots <i_{m_L(\rho)}<k} 
			\un{(\, j_\ell \le i_\ell)}\, J_{V^-(\rho),m_L(\rho)}(\,j_1,\ldots,j_{\ell-1},i_{\ell},\ldots ,i_{m_L(\rho)})\,,
			\end{equation}
			\item for $\bj_{m_L(\rho)+1}=(\,j_1,\ldots,j_{m_L(\rho)+1})\in\{0,\ldots, k\}^{m_L(\rho)+1}_<$:
			\begin{equation}\label{Eq: third case multi}
				J_{\rho,m_L(\rho)+1}(\,\bj_{m_L(\rho)+1}) := 
				J_{V^-(\rho),m_L(\rho)}(\,j_1,\ldots,j_{m_L(\rho)})\,.
			\end{equation}

		\end{itemize}
	\end{definition}

	\begin{example}\label{Example: J def}
		To give the rather abstract Definition \ref{Def: J} a more intuitive interpretation, we first consider the case
		\begin{equation}\label{eq:exrho}
			\rho = \begin{tikzpicture}[scale=0.2,baseline=0.1cm]
				\node at (0,0) [dot] (root) {};
				\node at (-3,3) [empty] (root1) {$1$};
				\node at (3,3) [empty] (root2) {$0$};
				\draw[kernel] (root1) to
				node [sloped,below] {\small }     (root);
				\draw[kernel] (root2) to
				node [sloped,below] {\small }     (root);
			\end{tikzpicture}=[\bullet_1,\bullet_0]\,.
		\end{equation}
		Here $m_L(\rho)=1$ and $k=1$, so that the only relevant multi-index components are $J_{\rho,1}(0)$, $J_{\rho,1}(1)$ and $J_{\rho,2}(0,1)$. 
		Moreover $V^-(\rho)=\bullet_0$, so that $J_{V^-(\rho),1}(0)=1$ by definition. Recall that
		\begin{align*}
			&J_{\rho,1}(\,j) = \sum_{i_1<\cdots <i_{m_L(\rho)}<k} 
			\un{(i_1<j )}\, J_{V^-(\rho),m_L(\rho)}(i_1,\ldots,i_{m_L(\rho)})=\un{(0<j)}J_{\bullet_0,1}(0)\,.
		\end{align*}
		Since the indicator function differs from zero only for $j=1$, we deduce that  $J_{\rho,1} (0)= 0, J_{\rho,1} (1)= 1$. Finally, recalling \eqref{Eq: third case multi}
		\[
		J_{\rho,2}(0,1)=
		J_{\bullet_0,1}(0)=1\,.
		\]
		Now consider the forest $w=\bullet\bullet=\tau_1\tau_2\in\cF$ and $\tau_0\in\cT$. Drawing a parallel with the computation of $B^{[3]}_{z_1z_2}(\tau)$, see \eqref{Eq: B^3}, we expect contributions composed of one, two and three gradients, respectively. According to Definition \ref{Def: link binary to rooted}, the grafting 
		\begin{equation*}
			(\tau_1\tau_2)\graf\tau_0=\tau_2\graf(\tau_1\graf\tau_0)-(\tau_2\graf\tau_1)\graf\tau_0
		\end{equation*}
		is equal to $G(\rho)-G(\rho')$, where
		\begin{equation*}\label{eq:exrho2}
			\rho=	\begin{tikzpicture}[scale=0.2,baseline=0.1cm]
				\node at (0,0) [dot] (root) {};
				\node at (-3,3) [empty] (root1) {$2$};
				\node at (3,3) [dot] (root2) {};
				\node at (0,6) [empty] (centerl) {$1$};
				\node at (6,6) [empty] (centerr) {$0$};
				\draw[kernel] (centerl) to
				node [sloped,below] {\small }     (root2);
				\draw[kernel] (centerr) to
				node [sloped,below] {\small }     (root2);
				\draw[kernel] (root1) to
				node [sloped,below] {\small }     (root);
				\draw[kernel] (root2) to
				node [sloped,below] {\small }     (root);
			\end{tikzpicture}=[\bullet_{2},[\bullet_1,\bullet_0]]\,,\qquad
			\rho'=\begin{tikzpicture}[scale=0.2,baseline=0.1cm]
				\node at (0,0) [dot] (root) {};
				\node at (-3,3) [dot] (root1) {};
				\node at (-6,6) [empty] (topl) {$2$};
				\node at (0,6) [empty] (topr) {$1$};
				\node at (3,3) [empty] (root2) {$0$};
				\draw[kernel] (topl) to
				node [sloped,below] {\small }     (root1);
				\draw[kernel] (topr) to
				node [sloped,below] {\small }     (root1);
				\draw[kernel] (root1) to
				node [sloped,below] {\small }     (root);
				\draw[kernel] (root2) to
				node [sloped,below] {\small }     (root);
			\end{tikzpicture}=[[\bullet_{2},\bullet_1],\bullet_0]\,.
		\end{equation*}
		We take our moves from $\rho$, for which $m_L(\rho)=1$ and $k=2$. By definition, the only non-vanishing value for one element is
		\begin{align*}
			J_{\rho,1}(2)=J_{[\bullet_1,\bullet_0],1}(1)=1\,.
		\end{align*}
		This means that, upon identifying $\tau_1=\tau_2=\bullet$, in the calculation of $B^{[3]}_{z_1z_2}(\tau_0)$ the term $\nabla\Up_{\bullet\graf(\bullet\graf\tau_0)}$ will be evaluated at $u_2$, as it is the case. For $\ell=2$ we have
		\begin{align*}
			J_{\rho,2}(1,2)=J_{[\bullet_1,\bullet_0],1}(1)=1\,.
		\end{align*}
		This exactly corresponds to the evaluation of $\nabla\Up_{\<1>}(u_1) \, \nabla\Up_{\tau_2}(u_2)$ in $B^{[3]}_{z_1z_2}(\tau_0)$. 
		
		Considering $\rho'$, for which $m_L(\rho')=2$ and $k=2$, there are more possibilities.
		Note that $V^-(\rho')$ is equal to the tree in \eqref{eq:exrho}, for which the only non-vanishing multi-indices are $J_{[\bullet_1,\bullet_0],1}(1)$ and $J_{[\bullet_1,\bullet_0],2}(0,1)$. 
		\[
		J_{\rho',1}(\,j) = \sum_{i_1<\cdots <i_{m_L(\rho)}<k} 
		\un{(i_1<j )}\, J_{[\bullet_1,\bullet_0],2}(i_1,i_2),
		\]
		which is non-vanishing only for $(i_1,i_2)=(0,1)$ and $j=1,2$, yielding $J_{\rho',1}(1)=J_{\rho',1}(2)=1$. 	This reflects the fact that the contribution $\nabla\Up_{(\tau_2\graf\tau_1)\graf\tau_0}$ in \eqref{Eq: B^3} is evaluated both at $u_1$ and $u_2$.
		
		For $\ell=2$ we apply the definition:
		\[
		J_{\rho',2}(\,j_1,j_2) = \sum_{i_{2}<2} 
		\un{(\, j_2 \le i_2)}\, J_{[\bullet_1,\bullet_0],2}(\,j_1,i_{2})
		\]
		which is non-vanishing only for $(\,j_1,j_2)=(0,1)$, yielding $J_{\rho',2}(0,1)=1 $. This corresponds to the evaluation points of 
		$\nabla\Up_{\tau_0}(u_0) \, \nabla\Up_{\tau_2\graf\tau_1}(u_1)$. 
		
		Lastly we consider the multi-index $J_{\rho',3}$, whose definition is adapted to the present case as
		\[
		J_{\rho',3}(\,j_1,j_2,j_{3}) = 
		J_{[\bullet_1,\bullet_0],2}(\,j_1,j_2)=\un{(\,j_1=0,j_2=1)}\,.
		\]
		It entails $J_{\rho',3}(0,1,2)=1$, while all the other contributions vanish, matching the evaluation of 
		$\nabla\Up_{\tau_0}(u_0) \, \nabla\Up_{\tau_1}(u_1) \, \nabla\Up_{\tau_2}(u_2)$.
	\end{example}

	We show now how the multi-indices $J$ of Definition \ref{Def: J} are connected to the insertion operation and the reduction chains of Definitions \ref{Def: insertion},
	respectively \ref{Def: reduction steps}.
	\begin{lemma}\label{lem: tuple correspondence}
If $\rho\in\cT_b^{I_k}$, then $J_{\rho,\ell}(\,\bj_\ell)$ is equal to $H_{\bullet_{k+1}\to_\ell \rho}(\,\bj_\ell)$ \, for all $\ell\in\{1,\ldots, m_L(\rho)+1\}$.
	\end{lemma}
	
	\begin{proof}
		We want to show that $H_{\bullet_{k+1}\to_\ell \rho}$ satisfies the same recursive characterisation as $J_{\rho,\ell}$ for all $\ell\in\{1,\ldots, m_L(\rho)+1\}$.
		We consider first \eqref{Eq: third case multi}, i.e. the case $\ell=m_L(\rho)+1$. 	
		We can write a reduction chain of $\bullet_{k+1}\to_\ell \rho$ for which $\bj_\ell=(p^1,\ldots, p^\ell)$ as
		\begin{equation}\label{Eq: specific reduction chain0}
			(\bullet_{k+1}\to_\ell\rho,\bullet_{k+1}\to_\ell\bar\rho^1,\ldots, \bullet_{k+1}\to_\ell\bar\rho^b,\bar\rho^b,\bar\rho^{b+1}\ldots,\bar\rho^k),
		\end{equation}
		where by Definition \ref{Def: reduction steps} $b=q^\ell-1$, since $b+1=q^\ell$ is the step in the reduction chain of $\bullet_{k+1}\to_\ell \rho$ in which the decoration $k+1$ is removed. 
		Then we define uniquely a reduction chain of $\rho$:
		\begin{equation*}
			(\rho,\bar\rho^{1},\ldots,\bar\rho^b,\ldots\bar\rho^k).
		\end{equation*}
		Then the sequence $(\,j_1,\ldots,j_{\ell-1})$ is equal to the sequence $(p^1,\ldots,p^{\ell-1})$ associated to this reduction chain of $\rho$.
		Viceversa, given a reduction chain of $\rho$ such that $(p^1,\ldots,p^{\ell-1})=(\,j_1,\ldots,j_{\ell-1})$, for any $j\in\{\,j_\ell+1,\ldots,k\}$ we 
		can find uniquely a reduction chain as in \eqref{Eq: specific reduction chain0} such that $k-j=b+1$. These applications are inverse to each other.
		As a consequence, for any $\bj_{m_L(\rho)+1}$ we have
		\begin{equation*}\label{Eq: reduction chain}
			H_{\bullet_{k+1}\to_{m_L(\rho)+1} \rho}(\,\bj_{m_L(\rho)+1}) = 
			H_{\rho}(\,\bj_{m_L(\rho)})\,.
		\end{equation*}
		
		Now we consider \eqref{Eq: second case multi}, namely $\ell=2,\ldots, m_L(\rho)$ and we fix $\bj_{\ell}=(\,j_1,\ldots, j_\ell)=(p^1,\ldots, p^\ell)$ as the indices of Definition \ref{Def: reduction steps} for $\bullet_{k+1}\to_\ell \rho$.  Henceforth we consider a generic reduction chain of $\rho\in\cT_b^{I_k}$:
		\begin{align*}
			(\rho, \bar\rho^1,\ldots, \bar\rho^{\bar q^{m_L(\rho)}},\ldots, \bar\rho^{\bar q^\ell},\ldots, \bar\rho^{\bar q^1},\ldots, \bar\rho^k)\,,
		\end{align*}
		where $(\bar q^{m_L(\rho)},\ldots, \bar q^1)$ are the steps from Definition \ref{Def: reduction steps} for $\rho$, while $(\bar p^1,\ldots, \bar p^{m_L(\rho)})$ are defined accordingly.
		To each reduction chain of this form we can associate multiple reduction chains of $\bullet_{k+1}\to_\ell \rho$ that satisfy the constraint $\bj_\ell=(p^1,\ldots, p^\ell)$ by observing that $\bullet_{k+1}$ can be extracted at all steps strictly after $\bar q^{\ell}$ until $\bar q^{\ell-1}$. This amounts to considering $\bar q^{\ell}<q^{\ell}\leq \bar q^{\ell-1}$.
		As a motivation, recall that $\bar q^\ell$ is the step in which the inserted leaf in $\bullet_{k+1}\to_\ell \rho$ becomes part of an effective cherry, while $\bar q^{\ell-1}$ represents the step in which the leftmost leaf of $\bar\rho^{\bar q^{\ell}}$ is removed. This can be done only after $\bullet_{k+1}$ is removed in the inserted tree. This condition translates to $k+1-\bar q^{\ell-1}\leq k+1-q^{\ell}< k+1-\bar q^{\ell}$, which implies $p^{\ell-1}<p^{\ell}\leq\bar p^{\ell}$ or equivalently $\bar p^{\ell-1}<j_{\ell}\leq\bar p^{\ell}$.
		
		In addition, all subsequent steps of the reduction chain of $\bullet_{k+1}\to_\ell\rho$ are chosen as the ones of the reduction chain of $\rho$, which translates to $j_r=\bar p^r$ for all $r\in\{1,\ldots,\ell-1\}$. This observation amounts to the identification 
		\[
		H_{\bullet_{k+1}\to_\ell\rho,\ell}(\,\bj_\ell) = \sum_{i_{\ell}<\cdots <i_{m_L(\rho)}<k} 
		\un{(\, j_\ell \leq i_\ell)}\, H_{\rho,m_L(\rho)}(\,j_1,\ldots,j_{\ell-1},i_{\ell},\ldots ,i_{m_L(\rho)})\,.
		\]
		
		Finally we consider the remaining case $\ell=1$. To each reduction chain of $\rho$ we can associate multiple reduction chains of $\bullet_{k+1}\to_1\rho$ by observing that, this time, $\bullet_{k+1}$ can be extracted in all possible ways before the step $\bar q^{1}$. This translates to the constraint $j_1=p^1>\bar p^1$, which entails the formula
		\[
		H_{\bullet_{k+1}\to_1\rho,1}(\,j_1) = \sum_{i_1<\cdots <i_{m_L(\rho)}<k} 
		\un{(i_1<j_1 )}\, H_{\rho,m_L(\rho)}(i_1,\ldots,i_{m_L(\rho)})\,.
		\]
		This concludes the proof.
	\end{proof}
	
\begin{example}\label{ex:redduc2}
	To illustrate Lemma \ref{lem: tuple correspondence}, consider the binary tree
	\begin{align*}
		\rho=\begin{tikzpicture}[scale=0.2,baseline=0.1cm]
			\node at (0,0) [dot] (root) {};
			\node at (-4,3) [dot] (root1) {};
			\node at (-7,6) [empty] (topl) {$3$};
			\node at (-1,6) [empty] (topr) {$2$};
			\node at (4,3) [dot] (root2) {};
			\node at (1,6) [empty] (centerl) {$1$};
			\node at (7,6) [empty] (centerr) {$0$};
			\draw[kernel] (topl) to
			node [sloped,below] {\small }     (root1);
			\draw[kernel] (topr) to
			node [sloped,below] {\small }     (root1);
			\draw[kernel] (centerl) to
			node [sloped,below] {\small }     (root2);
			\draw[kernel] (centerr) to
			node [sloped,below] {\small }     (root2);
			\draw[kernel] (root1) to
			node [sloped,below] {\small }     (root);
			\draw[kernel] (root2) to
			node [sloped,below] {\small }     (root);
		\end{tikzpicture}
		\in\cT_b^{I_3}\,.
	\end{align*}
	Since $m_L(\rho)=2$ and $|\rho|=k+1=4$, by the definition of $J_{\rho,\ell}$ for $\ell\in\{1,\ldots,m_L(\rho)+1\}$, see Definition \ref{Def: J}, one can verify that
	\begin{align*}
		&J_{\rho,1}(2)=J_{\rho,1}(3)= J_{\rho,2}(1,2)= J_{\rho,3}(1,2,3)=1\,,
	\end{align*}
	while all the other components of the multi-indices vanish. To exemplify the construction of $J_{\rho,\ell}$, $\ell=1,2,3$  according to Lemma \ref{lem: tuple correspondence} we consider all reduction chains of the binary trees 
	\begin{align}\label{Eq: drawing trees}
		\bullet_4\to_1\rho=	\hspace{-0.3cm}\begin{tikzpicture}[scale=0.2,baseline=0.1cm]
			\node at (0,3) [dot] (root) {};
			\node at (-3,0) [dot] (trueroot) {};
			\node at (-4,7) [dot] (root1) {};
			\node at (-7,10) [empty] (topl) {$3$};
			\node at (-1,10) [empty] (topr) {$2$};
			\node at (4,7) [dot] (root2) {};
			\node at (1,10) [empty] (centerl) {$1$};
			\node at (7,10) [empty] (tau) {$0$};
			\node at (-6,3) [empty] (rootr) {$4$};
			\draw[kernel] (topl) to
			node [sloped,below] {\small }     (root1);
			\draw[kernel] (topr) to
			node [sloped,below] {\small }     (root1);
			\draw[kernel] (centerl) to
			node [sloped,below] {\small }     (root2);
			\draw[kernel] (tau) to
			node [sloped,below] {\small }     (root2);
			\draw[kernel] (root1) to
			node [sloped,below] {\small }     (root);
			\draw[kernel] (root2) to
			node [sloped,below] {\small }     (root);
			\draw[kernel,red] (rootr) to
			node [sloped,below] {\small }     (trueroot);
			\draw[kernel] (root) to
			node [sloped,below] {\small }     (trueroot);			
		\end{tikzpicture}
		\bullet_4\to_2\rho=	\hspace{-0.4cm}\begin{tikzpicture}[scale=0.2,baseline=0.1cm]
			\node at (-4,10) [empty] (3) {$3$};
			\node at (2, 10) [empty] (2) {$2$};
			\node at (0,0) [dot] (root) {};
			\node at (-4,4) [dot] (root1) {};
			\node at (-7,7) [empty] (topl) {$4$};
			\node at (-1,7) [dot] (topr) {};
			\node at (3,3) [dot] (root2) {};
			\node at (0.5,6) [empty] (centerl) {$1$};
			\node at (6,6) [empty] (centerr) {$0$};
			\draw[kernel,blue] (topl) to
			node [sloped,below] {\small }     (root1);
			\draw[kernel] (topr) to
			node [sloped,below] {\small }     (root1);
			\draw[kernel] (centerl) to
			node [sloped,below] {\small }     (root2);
			\draw[kernel] (centerr) to
			node [sloped,below] {\small }     (root2);
			\draw[kernel,red] (root1) to
			node [sloped,below] {\small }     (root);
			\draw[kernel] (root2) to
			node [sloped,below] {\small }     (root);
			\draw[kernel] (3) to
			node [sloped,below] {\small }     (topr);
			\draw[kernel] (2) to
			node [sloped,below] {\small }     (topr);
		\end{tikzpicture} \ \
		\bullet_4\to_3\rho=\hspace{-0.8cm}\begin{tikzpicture}[scale=0.2,baseline=0.1cm]
			\node at (-4,10) [empty] (3) {$3$};
			\node at (-10, 10) [empty] (4) {$4$};
			\node at (0,0) [dot] (root) {};
			\node at (-4,4) [dot] (root1) {};
			\node at (-7,7) [dot] (topl) {};
			\node at (-1,7) [empty] (topr) {$2$};
			\node at (4,4) [dot] (root2) {};
			\node at (1,7) [empty] (centerl) {$1$};
			\node at (7,7) [empty] (centerr) {$0$};
			\draw[kernel,blue] (topl) to
			node [sloped,below] {\small }     (root1);
			\draw[kernel] (topr) to
			node [sloped,below] {\small }     (root1);
			\draw[kernel] (centerl) to
			node [sloped,below] {\small }     (root2);
			\draw[kernel] (centerr) to
			node [sloped,below] {\small }     (root2);
			\draw[kernel,red] (root1) to
			node [sloped,below] {\small }     (root);
			\draw[kernel] (root2) to
			node [sloped,below] {\small }     (root);
			\draw[kernel,green] (4) to
			node [sloped,below] {\small }     (topl);
			\draw[kernel] (3) to
			node [sloped,below] {\small }     (topl);
		\end{tikzpicture}
	\end{align}
	\begin{enumerate}
		\item We take our moves from $\bullet_4\to_1\rho$. According to Definition \ref{Def: reduction chain}, a reduction chain is obtained by contracting one effective cherry at each step. We want to record the step at which the edge that connects the node labelled by $4$ to the root is removed, corresponding to $q^1$ in Definition \ref{Def: reduction steps}.
		
		To this end we explicitly represent the two possible reduction chains of $\bullet_4\to_1\rho$. First we consider 
		
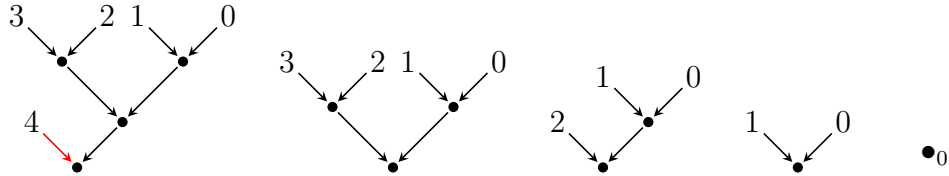
\begin{align*}
			\begin{tikzpicture}[scale=0.2,baseline=0.1cm]
				\node at (0,3) [dot] (root) {};
				\node at (-3,0) [dot] (trueroot) {};
				\node at (-4,7) [dot] (root1) {};
				\node at (-7,10) [empty] (topl) {$3$};
				\node at (-1,10) [empty] (topr) {$2$};
				\node at (4,7) [dot] (root2) {};
				\node at (1,10) [empty] (centerl) {$1$};
				\node at (7,10) [empty] (tau) {$0$};
				\node at (-6,3) [empty] (rootr) {$4$};
				\draw[kernel] (topl) to
				node [sloped,below] {\small }     (root1);
				\draw[kernel] (topr) to
				node [sloped,below] {\small }     (root1);
				\draw[kernel] (centerl) to
				node [sloped,below] {\small }     (root2);
				\draw[kernel] (tau) to
				node [sloped,below] {\small }     (root2);
				\draw[kernel] (root1) to
				node [sloped,below] {\small }     (root);
				\draw[kernel] (root2) to
				node [sloped,below] {\small }     (root);
				\draw[kernel,red] (rootr) to
				node [sloped,below] {\small }     (trueroot);
				\draw[kernel] (root) to
				node [sloped,below] {\small }     (trueroot);			
			\end{tikzpicture}
			\quad
			\begin{tikzpicture}[scale=0.2,baseline=0.1cm]
				\node at (0,0) [dot] (root) {};
				\node at (-4,4) [dot] (root1) {};
				\node at (-7,7) [empty] (topl) {$3$};
				\node at (-1,7) [empty] (topr) {$2$};
				\node at (4,4) [dot] (root2) {};
				\node at (1,7) [empty] (centerl) {$1$};
				\node at (7,7) [empty] (centerr) {$0$};
				\draw[kernel] (topl) to
				node [sloped,below] {\small }     (root1);
				\draw[kernel] (topr) to
				node [sloped,below] {\small }     (root1);
				\draw[kernel] (centerl) to
				node [sloped,below] {\small }     (root2);
				\draw[kernel] (centerr) to
				node [sloped,below] {\small }     (root2);
				\draw[kernel] (root1) to
				node [sloped,below] {\small }     (root);
				\draw[kernel] (root2) to
				node [sloped,below] {\small }     (root);
			\end{tikzpicture}
			\quad
			\begin{tikzpicture}[scale=0.2,baseline=0.1cm]
				\node at (0,0) [dot] (root) {};
				\node at (-3,3) [empty] (root1) {$2$};
				\node at (3,3) [dot] (root2) {};
				\node at (0,6) [empty] (centerl) {$1$};
				\node at (6,6) [empty] (tau) {$0$};
				\draw[kernel] (centerl) to
				node [sloped,below] {\small }     (root2);
				\draw[kernel] (tau) to
				node [sloped,below] {\small }     (root2);
				\draw[kernel] (root1) to
				node [sloped,below] {\small }     (root);
				\draw[kernel] (root2) to
				node [sloped,below] {\small }     (root);
			\end{tikzpicture}
			\quad
			\begin{tikzpicture}[scale=0.2,baseline=0.1cm]
				\node at (0,0) [dot] (root1) {};
				\node at (-3,3) [empty] (topl) {$1$};
				\node at (3,3) [empty] (topr) {$0$};
				\draw[kernel] (topl) to
				node [sloped,below] {\small }     (root1);
				\draw[kernel] (topr) to
				node [sloped,below] {\small }     (root1);
			\end{tikzpicture}
			\qquad
			\bullet_0
		\end{align*}
		In the first case $q^1=1$, hence $p^1=3$. The second possible reduction chain is
		
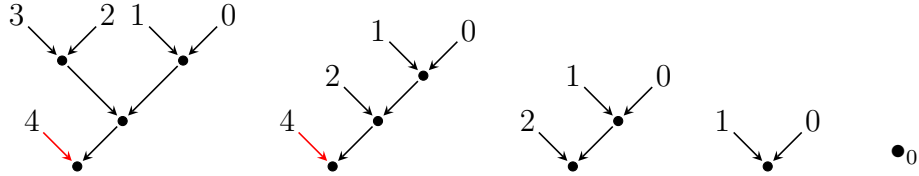
\begin{align*}
			\begin{tikzpicture}[scale=0.2,baseline=0.1cm]
				\node at (0,3) [dot] (root) {};
				\node at (-3,0) [dot] (trueroot) {};
				\node at (-4,7) [dot] (root1) {};
				\node at (-7,10) [empty] (topl) {$3$};
				\node at (-1,10) [empty] (topr) {$2$};
				\node at (4,7) [dot] (root2) {};
				\node at (1,10) [empty] (centerl) {$1$};
				\node at (7,10) [empty] (tau) {$0$};
				\node at (-6,3) [empty] (rootr) {$4$};
				\draw[kernel] (topl) to
				node [sloped,below] {\small }     (root1);
				\draw[kernel] (topr) to
				node [sloped,below] {\small }     (root1);
				\draw[kernel] (centerl) to
				node [sloped,below] {\small }     (root2);
				\draw[kernel] (tau) to
				node [sloped,below] {\small }     (root2);
				\draw[kernel] (root1) to
				node [sloped,below] {\small }     (root);
				\draw[kernel] (root2) to
				node [sloped,below] {\small }     (root);
				\draw[kernel,red] (rootr) to
				node [sloped,below] {\small }     (trueroot);
				\draw[kernel] (root) to
				node [sloped,below] {\small }     (trueroot);			
			\end{tikzpicture}
			\quad
			\begin{tikzpicture}[scale=0.2,baseline=0.1cm]
				\node at (0,3) [dot] (root) {};
				\node at (-3,0) [dot] (trueroot) {};
				\node at (-3,6) [empty] (root1) {$2$};
				\node at (3,6) [dot] (root2) {};
				\node at (0,9) [empty] (centerl) {$1$};
				\node at (6,9) [empty] (tau) {$0$};
				\node at (-6,3) [empty] (rootr) {$4$};
				\draw[kernel] (centerl) to
				node [sloped,below] {\small }     (root2);
				\draw[kernel] (tau) to
				node [sloped,below] {\small }     (root2);
				\draw[kernel] (root1) to
				node [sloped,below] {\small }     (root);
				\draw[kernel] (root2) to
				node [sloped,below] {\small }     (root);
				\draw[kernel,red] (rootr) to
				node [sloped,below] {\small }     (trueroot);
				\draw[kernel] (root) to
				node [sloped,below] {\small }     (trueroot);
			\end{tikzpicture}
			\quad
			\begin{tikzpicture}[scale=0.2,baseline=0.1cm]
				\node at (0,0) [dot] (root) {};
				\node at (-3,3) [empty] (root1) {$2$};
				\node at (3,3) [dot] (root2) {};
				\node at (0,6) [empty] (centerl) {$1$};
				\node at (6,6) [empty] (tau) {$0$};
				\draw[kernel] (centerl) to
				node [sloped,below] {\small }     (root2);
				\draw[kernel] (tau) to
				node [sloped,below] {\small }     (root2);
				\draw[kernel] (root1) to
				node [sloped,below] {\small }     (root);
				\draw[kernel] (root2) to
				node [sloped,below] {\small }     (root);
			\end{tikzpicture}
			\quad
			\begin{tikzpicture}[scale=0.2,baseline=0.1cm]
				\node at (0,0) [dot] (root1) {};
				\node at (-3,3) [empty] (topl) {$1$};
				\node at (3,3) [empty] (topr) {$0$};
				\draw[kernel] (topl) to
				node [sloped,below] {\small }     (root1);
				\draw[kernel] (topr) to
				node [sloped,below] {\small }     (root1);
			\end{tikzpicture}
			\qquad
			\bullet_0
		\end{align*}
		This time $p^1=4-q^1=2$. Recall that Lemma \ref{lem: tuple correspondence} states that the number $H_{\bullet_{4}\to_1 \rho}(\,j_1)$ of reduction chains such that $p^1=j_1$ coincides with $J_{\rho,1}(\,j_1)$. This example corroborates the result, since $H_{\bullet_{4}\to_1 \rho}(3)=J_{\rho,1}(3)=1$ and $H_{\bullet_{4}\to_1 \rho}(2)=J_{\rho,1}(2)=1$.
		
		\item For the choice $\ell=2$, which corresponds to $\bullet_4\to_2\rho$ in \eqref{Eq: drawing trees}, we need to record the steps at which the blue, respectively the red edges are removed, which correspond to $q^2$, resp. $q^1$. This time we can only build a single reduction chain
		
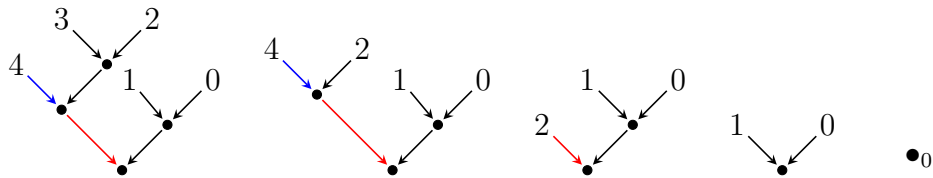
\begin{align*}
			\begin{tikzpicture}[scale=0.2,baseline=0.1cm]
				\node at (-4,10) [empty] (3) {$3$};
				\node at (2, 10) [empty] (2) {$2$};
				\node at (0,0) [dot] (root) {};
				\node at (-4,4) [dot] (root1) {};
				\node at (-7,7) [empty] (topl) {$4$};
				\node at (-1,7) [dot] (topr) {};
				\node at (3,3) [dot] (root2) {};
				\node at (0.5,6) [empty] (centerl) {$1$};
				\node at (6,6) [empty] (centerr) {$0$};
				\draw[kernel,blue] (topl) to
				node [sloped,below] {\small }     (root1);
				\draw[kernel] (topr) to
				node [sloped,below] {\small }     (root1);
				\draw[kernel] (centerl) to
				node [sloped,below] {\small }     (root2);
				\draw[kernel] (centerr) to
				node [sloped,below] {\small }     (root2);
				\draw[kernel,red] (root1) to
				node [sloped,below] {\small }     (root);
				\draw[kernel] (root2) to
				node [sloped,below] {\small }     (root);
				\draw[kernel] (3) to
				node [sloped,below] {\small }     (topr);
				\draw[kernel] (2) to
				node [sloped,below] {\small }     (topr);
			\end{tikzpicture}
			\quad
			\begin{tikzpicture}[scale=0.2,baseline=0.1cm]
				\node at (0,0) [dot] (root) {};
				\node at (-5,5) [dot] (root1) {};
				\node at (-8,8) [empty] (topl) {$4$};
				\node at (-2,8) [empty] (topr) {$2$};
				\node at (3,3) [dot] (root2) {};
				\node at (0.5,6) [empty] (centerl) {$1$};
				\node at (6,6) [empty] (centerr) {$0$};
				\draw[kernel,blue] (topl) to
				node [sloped,below] {\small }     (root1);
				\draw[kernel] (topr) to
				node [sloped,below] {\small }     (root1);
				\draw[kernel] (centerl) to
				node [sloped,below] {\small }     (root2);
				\draw[kernel] (centerr) to
				node [sloped,below] {\small }     (root2);
				\draw[kernel,red] (root1) to
				node [sloped,below] {\small }     (root);
				\draw[kernel] (root2) to
				node [sloped,below] {\small }     (root);
			\end{tikzpicture}
			\quad
			\begin{tikzpicture}[scale=0.2,baseline=0.1cm]
				\node at (0,0) [dot] (root) {};
				\node at (-3,3) [empty] (root1) {$2$};
				\node at (3,3) [dot] (root2) {};
				\node at (0,6) [empty] (centerl) {$1$};
				\node at (6,6) [empty] (tau) {$0$};
				\draw[kernel] (centerl) to
				node [sloped,below] {\small }     (root2);
				\draw[kernel] (tau) to
				node [sloped,below] {\small }     (root2);
				\draw[kernel,red] (root1) to
				node [sloped,below] {\small }     (root);
				\draw[kernel] (root2) to
				node [sloped,below] {\small }     (root);
			\end{tikzpicture}
			\quad
			\begin{tikzpicture}[scale=0.2,baseline=0.1cm]
				\node at (0,0) [dot] (root1) {};
				\node at (-3,3) [empty] (topl) {$1$};
				\node at (3,3) [empty] (topr) {$0$};
				\draw[kernel] (topl) to
				node [sloped,below] {\small }     (root1);
				\draw[kernel] (topr) to
				node [sloped,below] {\small }     (root1);
			\end{tikzpicture}
			\qquad
			\bullet_0
		\end{align*}
		and we have $q^2=2$, $q^1=3$, so that $(p^1,p^2)=(1,2)$. This entails that the only non-vanishing component is $J_{\rho,2}(1,2)$ and it coincides with $H_{\bullet_4\to_2\rho}(1,2)=1$.

		\item We iterate the procedure for $\bullet_4\to_3\rho$, this time having a third colored edge which corresponds to the third entry of the multi-index. Once again we have a single reduction chain:
		
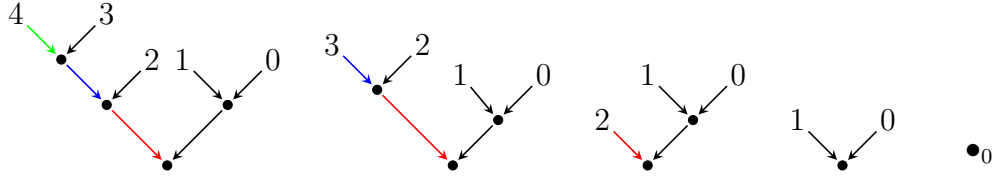
\begin{align*}
			\begin{tikzpicture}[scale=0.2,baseline=0.1cm]
				\node at (-4,10) [empty] (3) {$3$};
				\node at (-10, 10) [empty] (4) {$4$};
				\node at (0,0) [dot] (root) {};
				\node at (-4,4) [dot] (root1) {};
				\node at (-7,7) [dot] (topl) {};
				\node at (-1,7) [empty] (topr) {$2$};
				\node at (4,4) [dot] (root2) {};
				\node at (1,7) [empty] (centerl) {$1$};
				\node at (7,7) [empty] (centerr) {$0$};
				\draw[kernel,blue] (topl) to
				node [sloped,below] {\small }     (root1);
				\draw[kernel] (topr) to
				node [sloped,below] {\small }     (root1);
				\draw[kernel] (centerl) to
				node [sloped,below] {\small }     (root2);
				\draw[kernel] (centerr) to
				node [sloped,below] {\small }     (root2);
				\draw[kernel,red] (root1) to
				node [sloped,below] {\small }     (root);
				\draw[kernel] (root2) to
				node [sloped,below] {\small }     (root);
				\draw[kernel,green] (4) to
				node [sloped,below] {\small }     (topl);
				\draw[kernel] (3) to
				node [sloped,below] {\small }     (topl);
			\end{tikzpicture}
			\quad
			\begin{tikzpicture}[scale=0.2,baseline=0.1cm]
				\node at (0,0) [dot] (root) {};
				\node at (-5,5) [dot] (root1) {};
				\node at (-8,8) [empty] (topl) {$3$};
				\node at (-2,8) [empty] (topr) {$2$};
				\node at (3,3) [dot] (root2) {};
				\node at (0.5,6) [empty] (centerl) {$1$};
				\node at (6,6) [empty] (centerr) {$0$};
				\draw[kernel,blue] (topl) to
				node [sloped,below] {\small }     (root1);
				\draw[kernel] (topr) to
				node [sloped,below] {\small }     (root1);
				\draw[kernel] (centerl) to
				node [sloped,below] {\small }     (root2);
				\draw[kernel] (centerr) to
				node [sloped,below] {\small }     (root2);
				\draw[kernel,red] (root1) to
				node [sloped,below] {\small }     (root);
				\draw[kernel] (root2) to
				node [sloped,below] {\small }     (root);
			\end{tikzpicture}
			\quad
			\begin{tikzpicture}[scale=0.2,baseline=0.1cm]
				\node at (0,0) [dot] (root) {};
				\node at (-3,3) [empty] (root1) {$2$};
				\node at (3,3) [dot] (root2) {};
				\node at (0,6) [empty] (centerl) {$1$};
				\node at (6,6) [empty] (tau) {$0$};
				\draw[kernel] (centerl) to
				node [sloped,below] {\small }     (root2);
				\draw[kernel] (tau) to
				node [sloped,below] {\small }     (root2);
				\draw[kernel,red] (root1) to
				node [sloped,below] {\small }     (root);
				\draw[kernel] (root2) to
				node [sloped,below] {\small }     (root);
			\end{tikzpicture}
			\quad
			\begin{tikzpicture}[scale=0.2,baseline=0.1cm]
				\node at (0,0) [dot] (root1) {};
				\node at (-3,3) [empty] (topl) {$1$};
				\node at (3,3) [empty] (topr) {$0$};
				\draw[kernel] (topl) to
				node [sloped,below] {\small }     (root1);
				\draw[kernel] (topr) to
				node [sloped,below] {\small }     (root1);
			\end{tikzpicture}
			\qquad
			\bullet_0
		\end{align*}
		Here $p^3=4-1=3$, $p^2=4-2=2$ and $p^1=4-3=1$. Hence $H_{\bullet_4\to_3\rho}(1,2,3)=1$, which coincides as expected with the only non-vanishing component $J_{\rho,3}(1,2,3)$.
	\end{enumerate}
\end{example}
	
	For later use, see the proof of Lemma \ref{lem: Upsilon descent}, we prove three useful identities involving the multi-indices $J$ for binary trees of the form $\bullet_k\to_\ell\rho$ with $\rho\in\cT_b^{I_{k-1}}$.
	\begin{lemma}\label{lem: J sum} 
		Fix $\rho\in\cT_b^{I_{k-1}}$ and any $\ell\in\{1,\ldots, m_L(\rho)+1\}$. Then 
		\begin{itemize}
			\item For $f:\{0,1,\ldots, k\}\to\R^n$,
			\begin{equation}\label{Eq: 1}
				\sum_{j\le k} J_{\bullet_k\to_\ell\rho,1}(\,j) \, f(\,j)= \sum_{\boldsymbol{i}_\ell\in\{0,\ldots,k-1\}^\ell_<} 
				J_{\rho,\ell}(\boldsymbol{i}_\ell) \sum_{j=i_1+1}^k f(\,j)\,.	
			\end{equation}
			
			\item	For $m\in\{2,\ldots,\ell\}$, and $f:\{0,1,\ldots, k\}^{m}_<\to\R^n$,
			\begin{align}
				&\sum_{\boldsymbol{j}_{m}\in\{0,\ldots,k\}^{m}_<} J_{\bullet_k\to_\ell\rho,m}(\,\boldsymbol{j}_{m}) \, f(\,\boldsymbol{j}_{m})= 
				\sum_{\boldsymbol{i}_\ell\in\{0,\ldots,k-1\}^\ell_<}\,\sum_{i=i_{m-1}+1}^{i_{m}}
				J_{\rho,\ell}(\boldsymbol{i}_\ell)  \, f(\,\boldsymbol{i}_{m-1}, i)\,.\label{Eq: 2}
			\end{align}
			\item Finally for $f:\{0,1,\ldots, k\}^{\ell+1}_<\to\R^n$,
			\begin{equation}
				\sum_{\boldsymbol{j}_{\ell+1}\in\{0,\ldots,k\}^{\ell+1}_<} J_{\bullet_k\to_\ell\rho,\ell+1}(\,\boldsymbol{j}_{\ell+1})\, f(\,\boldsymbol{j}_{\ell+1})= \sum_{\boldsymbol{j}_{\ell}\in\{0,\ldots,k\}^{\ell}_<} J_{\rho,\ell}(\\,\boldsymbol{j}_{\ell}) \sum_{i=j_\ell+1}^k f(\,\boldsymbol{j}_{\ell},i)\,.\label{Eq: 3}
\end{equation}
		\end{itemize}
	\end{lemma}
	\begin{proof}
		Since $V^-(\bullet_k\to_\ell\rho)=\rho$ and $m_L(\bullet_k\to_\ell\rho)=\ell$, the identities in \eqref{Eq: 1} and \eqref{Eq: 3} follow easily from the recursive definitions of $J_{\rho,\cdot}$. To derive \eqref{Eq: 2} we observe that by \eqref{Eq: second case multi}
		\[
		J_{\bullet_k\to_\ell\rho,m}(\,\bj_{m}) = \sum_{i_{m}<\cdots <i_\ell<k} 
		\un{(\, j_{m} \le i_{m})}\, J_{\rho,\ell}(\,j_1,\ldots,j_{m-1},i_{m},\ldots ,i_\ell)\,,
		\]
		so that, for $f:\{0,1,\ldots, k\}^{m}_<\to\R^n$, \eqref{Eq: 2} follows by
		renaming first $i$ as $j_{m}$ and then $i_{m},\ldots, i_\ell$ as $j_{m},\ldots, j_\ell$.
	\end{proof}
	
	\begin{definition}
		For a given $\rho\in\cT_b^{I_k}$ and the associated multi-index $J_{\rho,\ell}$ as per Definition \ref{Def: J} we define its cardinality $|J_{\rho,\ell}|\in\N$ by
		\[
		|J_{\rho,\ell}|:=\sum_{\bj_\ell\in\{0,\ldots,k\}^\ell_<}J_{\rho,\ell}(\,\bj_\ell)\,.
		\]
	\end{definition}
	\begin{remark}\label{Rem: cardinality}	
		The coefficient $C(\rho)$ in \eqref{Eq: C} coincides with the number of reduction chains of $\rho$ by Lemma \ref{reduch}. Then Lemma \ref{lem: tuple correspondence} tells us something more, namely that 
		\begin{equation*}
			|J_{\rho,\ell}|=C(\bullet_{k+1}\to_\ell\rho)\,,
		\end{equation*}
		for all $\ell\in\{1,\ldots, m_L(\rho)+1\}$.
	\end{remark}
	Going back to the explicit characterisation of the Guin-Oudom grafting of forests proved in \ref{lem: GUBinary} , we explore now the possibility of writing the right-hand side of \eqref{Eq: GU binary} in terms of the multi-indices $J_{\rho,\cdot}$. First, we show that it can be written solely in terms of binary trees of the form $\bullet\to_\ell\rho$. 
	
	\begin{lemma}\label{Lem: construction binary trees}
		For any fixed $k\in\N_\ast$ let us consider $I_k:=\{1,\ldots,k\}$ and $F:\cT_b^{I_k}\to\R$. Then 
		\begin{align*}
			\sum_{\rho'\in\cT_b^{I_k}} F(\rho')&=\sum_{\rho\in\cT_b^{I_{k-1}}}\sum_{\ell=1}^{m_L(\rho)+1} F( \bullet_k\to_\ell\rho),
		\end{align*}
		where $\bullet_k\to_\ell$ is the insertion operator of Definition \ref{Def: insertion}.
	\end{lemma}
	\begin{proof}
	For $\rho\in\cT_b^{I_{k-1}}$ and $\ell\in\{1,\ldots,m_L(\rho)+1\}$, we set $\rho':=\bullet_k\to_\ell\rho\in\cT_b^{I_k}$.
	For $\rho'\in\cT_b^{I_k}$, we set $\rho=V^-(\rho')$ and $\ell:=m_L(\rho')$. The two maps are the inverse of each other and give a bijection between 
	the two sets.
	\end{proof}
		
	The link between the cardinality of the reduction chains and the combinatorial coefficient $C$ spelled out in Lemma \ref{reduch} and Remark \ref{Rem: cardinality} leads to a useful corollary of Lemma \ref{lem: GUBinary}.
	\begin{corollary}\label{Cor: C and J}
		For all $\tau_0,\ldots, \tau_k\in\cT$:
		\begin{equation*}
			(\tau_k\cdots \tau_1)\curvearrowright\tau_0=\frac{1}{k!}\sum_{p\in\Sigma_k} \sum_{\rho\in\cT_b^{p(I_{k-1})}}\sum_{\ell=1}^{m_L(\rho)+1} \sum_{\bj_\ell\in\{0,\ldots,k-1\}^\ell_<}(-1)^{k-m_R(\bullet_k\to_\ell\rho)} J_{\rho,\ell}(\,\bj_\ell) \, G\left( \bullet_k\to_\ell\rho\right)\,,
		\end{equation*}
		where $G$ is the mapping from planar binary trees to rooted trees of Definition \ref{Def: link binary to rooted} and 
		$m_R(\rho)\in\N$ denotes the distance of the rightmost leaf from the root.
	\end{corollary}
	\begin{proof}
		The result follows from Lemma \ref{lem: GUBinary} by a direct application of Lemma \ref{Lem: construction binary trees} for $F(\rho)=(-1)^{k-m_R(\rho)}C(\rho)G(\rho)$ in conjunction with Lemma \ref{lem: tuple correspondence}, see Remark \ref{Rem: cardinality}.
	\end{proof}
	
	\subsection{Algebraic formula for the remainder}\label{Sec: Algebraic formula for the remainder}
We are now going to state and prove our generalised Taylor formulae for the remainders $B^{[m+1]}(\cdot)$, using
the combinatorial machinery introduced in the previous section. 
\begin{definition}	
	For any binary tree $\rho\in\cT_b$ and for $\ell\in\{1,\ldots, m_L(\rho)+1\}$, we define
	\begin{equation}\label{eq:rhoelli}
	\rho^{(\ell)}_i:=\rho_i\,, \quad i=1,\ldots,\ell-1, \qquad \rho^{(\ell)}_\ell:=P_{\rho,\ell}\,,
	\end{equation}
see Definition \ref{Def: trunk}. Then we introduce the linear operation $\mathcal{A}:\cT_b\to\bigoplus_{i\ge 1} \cT_b^{\otimes i}$:
	\begin{align}\label{Eq: operation A}
		&\mathcal{A}(\rho):=\sum_{\ell=1}^{m_L(\rho)+1}(-1)^{\un{(\ell>1)}}\bigotimes_{i=1}^{\ell}\rho^{(\ell)}_i\,.
	\end{align}
\end{definition}

	\begin{example}
		Consider the same binary tree $\rho$ as in Example \ref{ex:4.14}.
		The action of $\cA$ defined in \eqref{Eq: operation A} consists of sequentially detaching the trees $\rho_1,\rho_2$ from $\rho=[[\bullet_d,\rho_2],\rho_1]$. In particular
		$\cA(\rho)=\rho-\rho_1^{(2)}\otimes\rho_{2}^{(2)}-\rho_1^{(3)}\otimes\rho_2^{(3)}\otimes\rho_3^{(3)}$
		for
		\begin{align*}
			&\rho=
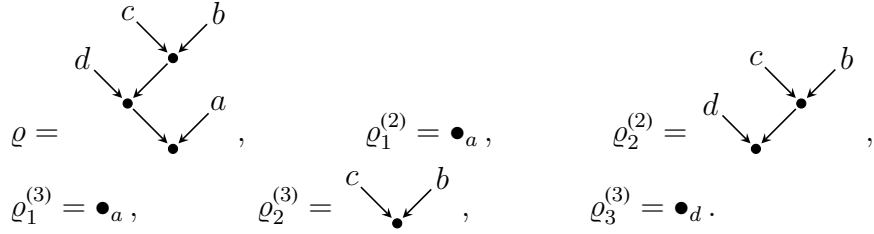
			\begin{tikzpicture}[scale=0.2,baseline=0.1cm]
				\node at (-3,9) [empty] (3) {$c$};
				\node at (3, 9) [empty] (2) {$b$};
				\node at (0,0) [dot] (root) {};
				\node at (-3,3) [dot] (root1) {};
				\node at (-6,6) [empty] (topl) {$d$};
				\node at (0,6) [dot] (topr) {};
				\node at (3,3) [empty] (root2) {$a$};
				\draw[kernel] (topl) to
				node [sloped,below] {\small }     (root1);
				\draw[kernel] (topr) to
				node [sloped,below] {\small }     (root1);
				\draw[kernel] (root1) to
				node [sloped,below] {\small }     (root);
				\draw[kernel] (root2) to
				node [sloped,below] {\small }     (root);
				\draw[kernel] (3) to
				node [sloped,below] {\small }     (topr);
				\draw[kernel] (2) to
				node [sloped,below] {\small }     (topr);
			\end{tikzpicture}\,,
			\hspace{1.5cm}\rho_1^{(2)}=\bullet_a\,,\hspace{1.5cm}\rho_2^{(2)}=\begin{tikzpicture}[scale=0.2,baseline=0.1cm]
				\node at (0,0) [dot] (root) {};
				\node at (-3,3) [empty] (root1) {$d$};
				\node at (3,3) [dot] (root2) {};
				\node at (0,6) [empty] (centerl) {$c$};
				\node at (6,6) [empty] (centerr) {$b$};
				\draw[kernel] (centerl) to
				node [sloped,below] {\small }     (root2);
				\draw[kernel] (centerr) to
				node [sloped,below] {\small }     (root2);
				\draw[kernel] (root1) to
				node [sloped,below] {\small }     (root);
				\draw[kernel] (root2) to
				node [sloped,below] {\small }     (root);
			\end{tikzpicture}\,,
		\\&\rho_1^{(3)}=\bullet_a\,,\hspace{1.5cm}\rho_2^{(3)}=\begin{tikzpicture}[scale=0.2,baseline=0.1cm]
			\node at (0,0) [dot] (root2) {};
			\node at (-3,3) [empty] (centerl) {$c$};
			\node at (3,3) [empty] (centerr) {$b$};
			\draw[kernel] (centerl) to
			node [sloped,below] {\small }     (root2);
			\draw[kernel] (centerr) to
			node [sloped,below] {\small }     (root2);
		\end{tikzpicture}
		\,,\hspace{1.5cm}\rho_3^{(3)}=\bullet_d\,.
		\end{align*}
	\end{example}
	\vspace{0.2cm}
	
	The multi-indices $J_{\rho,\cdot}$ of Definition \ref{Def: J} were introduced precisely to cope with the problem of understanding the evaluation points of the integrand of the desired formula for the remainders. Hence we introduce an evaluation map that involves such objects. We fix two points $z_1, z_2\in\R^n$ and we recall the notation for the linear interpolation
	\begin{equation*}
		u_i:=z_1+t_iz^{[1]}\,,\qquad i\in \N\,,\qquad t_i\in[0,1]\,, \qquad  z^{[1]}:=z_2-z_1.
	\end{equation*}
	We define a map
	\[
	e: \bigoplus_{i=1}^{+\infty}\cT_b^{\otimes i}\to \bigoplus_{k=0}^{+\infty}C([0,T]_{\geq}^{k+1},\R^d),
	\] 
	such that, for $\rho_1,\ldots,\rho_r\in\cT_b$ and for  $\gamma:=[[\rho_r,\rho_{r-1}],\ldots\rho_1]$ having $k+1$ leaves, 
	\begin{align}
	&e(\rho_1\otimes\cdots\otimes\rho_r)\in C([0,T]_{\geq}^{k+1},\R^d), \nonumber \\
		&e(\rho_1\otimes\cdots\otimes\rho_r)(\mathbf{t}_{k}):=\sum_{\bj_r\in\{0,\ldots,k\}^r_<}J_{\rho,r}(\,\bj_r)\prod_{i=1}^r\nabla\Up_{G(\rho_i)}(u_{j_i})\,,\label{eq: e}
	\end{align}
	where the factors in the product of gradients of the elementary differential components are contracted with each other, namely
	\begin{equation*}
		[\nabla\Upsilon_{G(\rho_i)}(u_{j_i})\nabla\Upsilon_{G(\rho_{i+1})}(u_{j_{i+1}})]^a_b=\sum_{c=1}^n \partial_c[\Upsilon_{G(\rho_i)}]^a(u_{j_i})\,\partial_b[\Upsilon_{G(\rho_{i+1})}]^c(u_{j_{i+1}})\,.
	\end{equation*}  
	 In addition we recall the notation introduced in Section \ref{Sec: notation}
\[
[0, T]^{k+1}_{\geq} :=
		\left\{ (t_0, \ldots, t_k) \in[0,T]^{k+1} :  0 \leq t_k \leq \ldots \leq
		t_0 \leq T \right\}\,, \qquad \d \mathbf{t}_k:=\d t_k\ldots\d t_0\,,
		\]
	Then, recalling \eqref{eq:rhoelli}, the evaluation map $e$ and the operation $\mathcal{A}$ of \eqref{Eq: operation A} interact as follows: 
	\begin{align}
		&e(\mathcal{A}(\rho))(\mathbf{t}_{k})= 
		\sum_{\ell=1}^{m_L(\rho)+1}(-1)^{\un{(\ell>1)}}\sum_{\bj_\ell\in\{0,\ldots,k\}^\ell_<}J_{\rho,\ell}(\,\bj_\ell) \, \prod_{i=1}^{\ell}\nabla \Up_{G\left(\rho^{(\ell)}_i\right)}(u_{j_i}) .
		\label{eq:e(A)}
	\end{align}
		\begin{definition}\label{Def: sym H}
		Let $I$ be a finite ordered set and $\bt\in\cT_b^I$ an ordered set of trees. We define $\cP(I,\bt)$ as the set of permutations of $I$ such that $p(i)\neq j$ if $\tau_i=\tau_j$ and $i\neq j$, for all $i,j\in I$.
\end{definition}\label{Def: H}
	Everything is in place to introduce an analytic expression of paramount importance to rewrite the remainders $B^{[m+1]}(\cdot)$ solely in terms of gradients of the elementary differentials.
	\begin{definition}\label{Def:H}
		We define $H_{z_1z_2} :\cF\times \cT \to \R^n\otimes(\R^n)^\ast$ as
		\begin{equation*}\label{Eq: def H(one)}
			H_{z_1z_2}(\one;\tau_0) :=\int_0^1 \nabla\Up_{\tau_0}(u_0)\d t_0\, ,
		\end{equation*}
		and, for $w=\tau_1\cdots\tau_k\in\cF\setminus\{\one\}$, $\tau_0,\tau_1,\ldots,\tau_k\in\cT$, denoting $I_k=\{1,\ldots,k\}$ and $\bt=(\tau_1,\ldots,\tau_k)\in\cT_b^{I_k}$,
		\begin{align}
			&H_{z_1z_2}(w;\tau_0) \label{Eq: def H}
			:=\sum_{p\in\cP(I_k,\bt)}\sum_{\rho\in\cT_b^{p(I_k)}}\int_{[0,1]_\ge^{k+1}}(-1)^{k-m_R(\rho)}\,e(\mathcal{A}(\rho))(\mathbf{t}_k)\d\mathbf{t}_k\,.
		\end{align}
	\end{definition}
	
		We are finally in the position to state the main result of this section, namely a generalised Taylor formula for the remainder of the finite difference solution to the rough equation that only involves gradients of the elementary differentials $\Upsilon_\tau$. This feature will be crucial when deriving a priori estimates on the solution only assuming that $\Upsilon_\tau$ is globally Lipschitz continuous for all $\tau\in\cT^{\leq N}$. 
	\begin{theorem}\label{Thm: Taylor Lipschitz}
		For $m\ge 0$ and $\tau_0\in\cT$, let $H_{z_1z_2}$ be the expression defined in \eqref{Eq: def H}. In addition let $\X:\cF^{\le m}\to\R$ be a multiplicative functional. Then
		\begin{align}
			&\Up_{\tau_0}(z_2)-\sum_{w\in\cF^{\leq m}}\frac{\Up_{w\curvearrowright\tau_0}(z_1)}{\pi(w)} \, \X(w)=\sum_{w\in\cF^{\leq m}}H_{z_1z_2}(w;\tau_0) \cdot z^{[m+1-|w|]}\, \X(w),\label{Eq: identity remainder H}
		\end{align}
		where $\pi(w)$ is defined in \eqref{Eq: permutation factor} and $z^{[i]}$ in \eqref{eq:zk12}.
	\end{theorem}
	
	\begin{example}
		To convince the reader that Theorem \ref{Thm: Taylor Lipschitz} encodes the generalised Taylor remainders conjectured at the beginning of Section \ref{Sec: Lipschitz elementary differentials}, we explicitly compute the contributions to $B^{[3]}_{z_1z_2}(\tau)$ as per \eqref{Eq: identity remainder H}. 
		\begin{align}
			B^{[3]}_{z_1z_2}(\tau)=&H_{z_1z_2}(\one;\tau)\cdot z^{[3]}+H_{z_1z_2}(\bullet;\tau)\cdot z^{[2]}\,\X(\bullet)\nonumber
			\\&+ H_{z_1z_2}(\bullet\bullet;\tau)\cdot z^{[1]}\,\X(\bullet\bullet)+ H_{z_1z_2}(\,\<1>\,;\tau)\cdot z^{[1]}\,\X(\,\<1>\,)\,.\label{Eq: B^3 ex}
		\end{align}
		Definition \ref{Def: H} entails
		\begin{align*}
			&H_{z_1z_2}(\one;\tau)=\int_0^1 \nabla\Up_{\tau_0}(u_0)\d t_0\,,
			\\&H_{z_1z_2}(\bullet;\tau)=\int_{[0,1]_\ge^{2}}e(\mathcal{A}([\bullet_1,\bullet_0]))(\mathbf{t}_2)\d \mathbf{t}_2
			=\int_{[0,1]_\ge^{2}}\left(\nabla\Up_{\bullet\graf\tau}(u_0)-\nabla\Up_{\tau}(u_0)\nabla\Up_{\bullet}(u_1)\right)\d \mathbf{t}_2\,,
			\end{align*}
			We refer the reader to  Example \ref{Example: J def} for a detailed calculation of the components of $J_{\rho,\ell}$ for the trees of interest in this case. 
			The identification $w=\tau_1=\<1>$ implies
			\begin{align*}
			 &H_{z_1z_2}(\,\<1>\,;\tau)=\int_{[0,1]_\ge^{3}}e(\mathcal{A}([\bullet_1,\bullet_0]))(\mathbf{t}_2)\d \mathbf{t}_2=\int_{[0,1]_\ge^{3}}\left(\nabla\Up_{\<1>\,\graf\tau}(u_0)-\nabla\Up_\tau(u_0)\nabla\Up_{\<1>}(u_1)\right)\d \mathbf{t}_2\,,
			 \end{align*}
			 while the identification $w=\tau_1\tau_2$ for $\tau_1=\tau_2=\bullet$ yields
			 \begin{align*}
			&H_{z_1z_2}(\bullet\bullet;\tau)=-\int_{[0,1]_\ge^{3}}e(\mathcal{A}([[\bullet_2,\bullet_1],\bullet_0]))(\mathbf{t}_2)\d \mathbf{t}_2+\int_{[0,1]_\ge^{3}}e(\mathcal{A}([\bullet_2,[\bullet_1,\bullet_0]))(\mathbf{t}_2)\d \mathbf{t}_2
			\\&=-\int_{[0,1]_\ge^{3}}\bigl(\textcolor{red}{\nabla\Up_{(\bullet\graf\bullet)\graf\tau}(u_1)+\nabla\Up_{(\bullet\graf\bullet)\graf\tau}(u_2)}
			\\&\hspace{3cm}-\nabla\Up_\tau(u_0)\nabla\Up_{\bullet\graf\bullet}(u_1)-\nabla\Up_{\tau}(u_0)\nabla\Up_{\bullet}(u_1)\nabla\Up_{\bullet}(u_2)\bigr)\d \mathbf{t}_2
			\\&+\int_{[0,1]_\ge^{3}}\left(\nabla\Up_{\bullet\graf(\bullet\graf\tau)}(u_2)-\nabla\Up_{\bullet\graf\tau}(u_1)\nabla\Up_{\bullet}(u_2)\right)\d \mathbf{t}_3\,.
		\end{align*}
		Observe that $J_{[[\bullet_2,\bullet_1],\bullet_0],1}(1)=J_{[[\bullet_2,\bullet_1],\bullet_0],1}(2)=1$, while all other components of the multi-index vanish, see Example \ref{Example: J def}. This justifies the presence of the identical red contributions at different evaluation points. A 
		direct comparison shows that \eqref{Eq: B^3 ex} coincides with \eqref{Eq: B^3}.
	\end{example}
	
	At the heart of the proof of Theorem \ref{Thm: Taylor Lipschitz} is the following Proposition, which shows that $H$ is characterised by a recursive expression tightly linked to the Guin-Oudom product.
	\begin{proposition}\label{Prop: recursive form remainder}
		For any forest $w=\tau_1\cdots\tau_k\in\cF$ and any rooted tree $\tau_0\in\cT$ we have
		\begin{equation}\label{eq:recurema}
			\sum_{i=1}^k\frac{1}{\pi(\tau_i; w)}H_{z_1z_2}(w\setminus\tau_i;\tau_0) \, \Up_{\tau_i}(z_1)-\frac{1}{\pi(w)}\Up_{w\curvearrowright\tau_0}(z_1)=H_{z_1z_2}(w;\tau_0)\cdot z^{[1]}\,,
\end{equation}
		where $\pi(\tau_i;w)$ is the number of times the tree $\tau_i$ appears in the forest $w$ and $w\setminus\tau_i$ is the forest obtained by removing an instance of $\tau_i$ in $w$, see Definition \ref{Def: symmetry factor}.
	\end{proposition}
	Before proving this result, we need a collection of technical lemmas that tackle the problem of manipulating the expression  $H_{z_1z_2}(w\setminus\tau_i;\tau_0) \, \Up_{\tau_i}(z_1)$.
	\begin{lemma}\label{lem:difffun}
		Let $f_1,\ldots,f_{\ell+1}:\R^n\to\R^n$ differentiable functions and $0\le j_1<\cdots<j_\ell<k$, $j_{\ell+1}:= k$. Define, for $i=1, \ldots,\ell+1$, the maps $F_i:\R^n\to\R^n$ as
		\begin{equation}\label{eq:fF}
		F_{\ell+1}:=f_{\ell+1}, \qquad F_{i}:=(\nabla f_{i})\cdot F_{i+1}=\sum_{j=1}^n \frac{\partial f_i}{\partial y_j}\, F^j_{i+1}, \quad i=1,\ldots, \ell\,.
		\end{equation}
		Then we have
		\begin{align}
			&\int_{[0,1]_\geq^k}\left[\left(\prod_{i=1}^\ell\nabla f_i(u_{j_i})\right) F_{\ell+1}(z_1)- F_1(u_{j_1})\right]\d \mathbf{t}_{k-1}\nonumber
			\\ &=-\sum_{m=1}^\ell\sum_{j=j_{m}+1}^{j_{m+1}}\int_{[0,1]_\geq^{k+1}}\left(\prod_{i=1}^{m}
			\nabla f_i(u_{j_i})\right)\, \nabla F_{m+1}(u_j) \cdot z^{[1]}\d \mathbf{t}_{k}\,.
			\label{eq:replace}
		\end{align}
	\end{lemma}
	\begin{proof}
		For $\ell=1$ and $t_1\in[0,T]$ we have
		\[
		\begin{split}
			\nabla f_1(u_{j_1})\, f_2(z_1)- (\nabla f_1 \cdot f_2)(u_{j_1})&= -\nabla f_1(u_{j_1}) (f_2(u_{j_1})-f_2(z_1))
			\\ &=-\int_0^{t_1} \nabla f_1(u_{j_1})\,\nabla  f_2(u') \cdot z^{[1]}\d t', \qquad u':=z_1+t' z^{[1]}\,.
		\end{split}
		\]
		Then, resorting to \eqref{Eq: int exchange 1} for $F_2:=f_2$, $F_1:=\nabla f_1 \cdot f_2$,
		\[
		\begin{split}
			& \int_{[0,1]_\geq^k}\left[\nabla f_1(u_{j_1})\, F_2(z_1) -F_1(u_{j_1})\right] \d \mathbf{t}_{k-1}=-\sum_{j=j_1+1}^{k}\int_{[0,1]_\geq^{k+1}}	\nabla f_1(u_{j_1})\,\nabla  F_2(u_{j}) \cdot z^{[1]}\d \mathbf{t}_{k}\,,	\end{split}
		\]
		which coincides with \eqref{eq:replace}. Now we argue by induction on $\ell$. We write
		\[
		F_{\ell+1}(z_1) = F_{\ell+1}(u_{j_\ell}) - \int_0^{t_{j_\ell}} \nabla F_{\ell+1}(u')\cdot z^{[1]} \d t' \, , \qquad u':=z_1+t' z^{[1]}\,,\qquad t'\in[0,1]\,.
		\]
		Substituting this expression in the left-hand side of \eqref{eq:replace} we obtain
		\begin{align*}
			\text{LHS of \eqref{eq:replace}}&=
			\int_{[0,1]_\geq^k}\left[\left(\prod_{i=1}^\ell\nabla f_i(u_{j_i})\right) F_{\ell+1}(u_{j_\ell})- F_1(u_{j_1})\right]\d \mathbf{t}_{k-1}
			\\ & \quad- \int_{[0,1]_\geq^k}\left(\prod_{i=1}^\ell\nabla f_i(u_{j_i})\right) \int_0^{t_{j_\ell}} \nabla F_{\ell+1}(u')\cdot z^{[1]} \d t' \, \d \mathbf{t}_{k-1}
			\\ &= \int_{[0,1]_\geq^k}\left[\left(\prod_{i=1}^{\ell-1}\nabla f_i(u_{j_i})\right) F_\ell(u_{j_\ell})- F_1(u_{j_1})\right]\d \mathbf{t}_{k-1}
			\\ & \quad- \sum_{j=j_\ell+1}^{k}\int_{[0,1]_\geq^{k+1}}\left(\prod_{i=1}^\ell\nabla f_i(u_{j_i})\right) \, \nabla F_{\ell+1}(u_j) \, \cdot z^{[1]}\d \mathbf{t}_{k},
		\end{align*}
		where in the last step we used \eqref{Eq: int exchange 1}. Now by recurrence on $\ell$
		\begin{align*}
			&\int_{[0,1]_\geq^k}\left[\left(\prod_{i=1}^\ell\nabla f_i(u_{j_i})\right) f_{\ell+1}(z_1)- F_1(u_{j_1})\right]\d \mathbf{t}_{k-1}
			\\ & =-\sum_{m=1}^{\ell-1}\sum_{j=j_{m}+1}^{j_{m+1}}\int_{[0,1]_\geq^{k+1}}\left(\prod_{i=1}^{m}
			\nabla f_i(u_{j_i})\right)\, \nabla F_{m+1}(u_j) \cdot z^{[1]}\d \mathbf{t}_{k}
			\\ & \quad- \sum_{j=j_\ell+1}^{k}\int_{[0,1]_\geq^{k+1}}\left(\prod_{i=1}^\ell\nabla f_i(u_{j_i})\right) \, \nabla F_{\ell+1}(u_j) \, \cdot z^{[1]}\d \mathbf{t}_{k}
			\\ & = -\sum_{m=1}^\ell\sum_{j=j_{m}+1}^{j_{m+1}}\int_{[0,1]_\geq^{k+1}}\left(\prod_{i=1}^{m}
			\nabla f_i(u_{j_i})\right)\, \nabla F_{m+1}(u_j) \cdot z^{[1]}\d \mathbf{t}_{k}\,.
		\end{align*}
		The proof is complete.
	\end{proof}
	
	The next key result links Lemma \ref{lem:difffun} to our specific setting. 	
	\begin{lemma}\label{Cor: Upsilon multiplication}
		Let $\tau_0,\ldots,\tau_k\in\cT$ and $\rho=[\ldots[\bullet_{k-1},\rho_{m_L(\rho)}],\ldots\rho_1]\in\cT_b^{I_{k-1}}$. 
		For $\ell\in\{1,\ldots, m_L(\rho)+1\}$ and $0\le j_1<\cdots<j_\ell<k$, $j_{\ell+1}:= k$, we have 
		\begin{align}	\nonumber
			& \int_{[0,1]_\geq^k}\left[\left(\prod_{i=1}^{\ell}\nabla \Up_{G\left(\rho^{(\ell)}_i\right)}(u_{j_i})\right) \Up_{\tau_k}(z_1)- \Up_{G\left(\bullet_{k}\to_\ell\rho\right)}(u_{j_1})\right]\d\mathbf{t}_{k-1}
			\\ & =-\sum_{m=1}^{\ell}\sum_{j=j_{m}+1}^{j_{m+1}}\int_{[0,1]_\geq^{k+1}}\left(\prod_{i=1}^{m+1}
			\nabla \Up_{G\left((\bullet_{k}\to_\ell\rho)^{(m+1)}_i\right)}(u_{j_i'})\right) \cdot z^{[1]}\d \mathbf{t}_{k}\, ,
			\label{eq: Cor B}
		\end{align}
		recall Definition \ref{Def: link binary to rooted} and \eqref{eq:rhoelli}, where $j_i'=j_i$ if $i\le m$ and $j_{m+1}'=j$. 
	\end{lemma}
	\begin{proof}
			Setting $f_i:=\Up_{G\left(\rho^{(\ell)}_i\right)}$, $i=1,\ldots,\ell$, and $F_{\ell+1}:=f_{\ell+1}:=\Up_{\tau_k}$,
		  by \eqref{eq:rhoelli} and by the morphism property of elementary differentials in Lemma \ref{lem: morphism upsilon} the functions
		\begin{align*}
			F_i=\Up_{G([\ldots[[\bullet_k,P_{\rho,\ell}],\rho_{\ell-1}],\ldots,\rho_i])}\,,\qquad \forall i\in\{1,\ldots, \ell\}\,,
		\end{align*}
		satisfy \eqref{eq:fF}. 		By Definitions \ref{Def: trunk} and \ref{Def: insertion},  and by \eqref{eq:rhoelli}, since
	$$\bullet_{k}\to_\ell\rho=[\ldots[[\bullet_k,P_{\rho,\ell}],\rho_{\ell-1}],\ldots, \rho_1],$$ we also have for $1\le i\le m\le \ell$
	\begin{align*}
	&\left(\bullet_{k}\to_\ell\rho\right)^{(\ell+1)}_i = \rho_i, \quad i=1,\ldots,\ell-1,\qquad  
	\\ & \left(\bullet_{k}\to_\ell\rho\right)^{(\ell+1)}_\ell= P_{\rho,\ell},
	\qquad  
	\\ & \left(\bullet_{k}\to_\ell\rho\right)^{(\ell+1)}_{\ell+1}=\bullet_k,
	\\ 
	&\left(\bullet_{k}\to_\ell\rho\right)^{(m+1)}_i = \rho_i=\rho^{(\ell)}_i, \quad i\le m< \ell,
	\\ & 
\left(\bullet_{k}\to_\ell\rho\right)^{(m+1)}_{m+1}= [\ldots[[\bullet_k,P_{\rho,\ell}],\rho_{\ell-1}],\ldots,\rho_{m+1}], 	\qquad  m<\ell.
	\end{align*}
	In particular $f_i=\Up_{G\left((\bullet_{k}\to_\ell\rho)^{(m+1)}_i\right)}$ for
	$i\le m\le \ell$ and $F_{m+1}=\Up_{G\left((\bullet_{k}\to_\ell\rho)^{(m+1)}_{m+1}\right)}$ for $m\le\ell$.
Then \eqref{eq: Cor B} follows from  \eqref{eq:replace}.
	\end{proof}
	We introduce a reduced version of the map $\cA$ from \eqref{Eq: operation A}. For $\rho\in\cT_b$:
	\begin{align*}
		\mathcal{A}'(\rho):=-\sum_{\ell=2}^{m_L(\rho)+1}\bigotimes_{i=1}^{\ell}\rho^{(\ell)}_i \,. 
	\end{align*}
	Observe that, if $\rho$ has $k$ leaves, then by the definition \eqref{eq: e} of $e$
	\begin{align}
		&e(\mathcal{A}'(\rho))(\mathbf{t}_{k-1}) =-\sum_{\ell=2}^{m_L(\rho)+1}\sum_{\bj_\ell\in\{0,\ldots,k-1\}^\ell_<}J_{\rho,\ell}(\,\bj_\ell)\left(\prod_{i=1}^{\ell}\nabla \Up_{G\left(\rho^{(\ell)}_i\right)}(u_{j_i})\right)\,.\label{Eq: A'}
			\end{align}
	\begin{lemma}\label{lem: Upsilon descent}
		For a binary tree $\rho\in\cT_{b}^{I_{k-1}}$, a tree $\tau_k\in\cT$ and $e\circ\CA$, $e\circ\cA'$ as in \eqref{eq:e(A)}, \eqref{Eq: A'} respectively, we have
		\begin{align}
			&\int_{[0,1]_\geq^k}\left[ e(\mathcal{A}(\rho))(\mathbf{t}_{k-1}) \, \Up_{\tau_{k}}(z_1)-
			\sum_{\ell=1}^{m_L(\rho)+1} (-1)^{\un{(\ell>1)}}\sum_{\bj_\ell\in\{0,\ldots,k-1\}^\ell_<}J_{\rho,\ell}(\,\bj_\ell)\, \Up_{G\left(\bullet_{k}\to_\ell\rho\right)}(u_{j_1})\right]\d\mathbf{t}_{k-1}\nonumber
			\\& = \sum_{\ell=1}^{m_L(\rho)+1} (-1)^{\un{(\ell>1)}} \int_{[0,1]_\geq^{k+1}}e(\mathcal{A}'(\bullet_{k}\to_\ell\rho))(\mathbf{t}_k)\cdot z^{[1]}\d \mathbf{t}_k\,.
			\label{eq:eAA'}
		\end{align}
	\end{lemma}

	\begin{proof}
	The left-hand side of \eqref{eq:eAA'} is the sum over $(\ell,\,\bj_\ell)$ of $(-1)^{\un{(\ell>1)}}$ times the LHS of \eqref{eq: Cor B}, namely
	\[
\begin{split}
& \text{LHS of \eqref{eq:eAA'}} = \sum_{\ell=1}^{m_L(\rho)+1} (-1)^{\un{(\ell>1)}}\sum_{\bj_\ell\in\{0,\ldots,k-1\}^\ell_<}J_{\rho,\ell}(\,\bj_\ell)\, \left\{
\text{LHS of \eqref{eq: Cor B}}\right\} .
\end{split}
\]
Then if we perform the same sum on the RHS of \eqref{eq: Cor B} we obtain
	\[
\begin{split}
 \sum_{\ell=1}^{m_L(\rho)+1} &(-1)^{\un{(\ell>1)}}\sum_{\bj_\ell\in\{0,\ldots,k-1\}^\ell_<}J_{\rho,\ell}(\,\bj_\ell)\, \left\{
\text{RHS of \eqref{eq: Cor B}}\right\}
\\ & = -\sum_{\ell=1}^{m_L(\rho)+1} (-1)^{\un{(\ell>1)}}\sum_{m=1}^{\ell}\sum_{\bj_\ell\in\{0,\ldots,k-1\}^\ell_<}J_{\rho,\ell}(\,\bj_\ell)\sum_{j=j_{m}+1}^{j_{m+1}}
\\  & \qquad \qquad \qquad \cdot \int_{[0,1]_\geq^{k+1}}\left(\prod_{i=1}^{m+1}
			\nabla \Up_{G\left((\bullet_{k}\to_\ell\rho)^{(m+1)}_i\right)}(u_{j_i'})\right) \cdot z^{[1]}\d \mathbf{t}_{k}
\end{split}
\]
where $j_{m+1}=k$ and $j_i'=j_i$ if $i\le m$ and $j_{m+1}'=j$. Now for $m\in\{1,\ldots,\ell-1\}$ by \eqref{Eq: 2}
\[
\begin{split}
& \sum_{\bj_{m+1}\in\{0,\ldots,k\}^{m+1}_<} J_{\bullet_{k}\to_\ell\rho,m+1}(\,\bj_{m+1})\, F(\,\bj_{m+1})
 = \sum_{\bj_\ell\in\{0,\ldots,k-1\}^\ell_<}\sum_{j=j_{m}+1}^{j_{m+1}}J_{\rho,\ell}(\,\bj_\ell) \,F(\,j_1,\ldots,j_m,j)
\end{split}
\]
and for $m=\ell$ by \eqref{Eq: 3}
\[
\begin{split}
\sum_{\bj_{\ell+1}\in\{0,\ldots,k\}^{\ell+1}_<} J_{\bullet_{k}\to_\ell\rho,\ell+1}(\,\bj_{\ell+1})\, F(\,\bj_{m+1})
 = \sum_{\bj_\ell\in\{0,\ldots,k-1\}^\ell_<}\sum_{j=j_{\ell}+1}^{k}J_{\rho,\ell}(\,\bj_\ell) \,F(\,j_1,\ldots,j_\ell,j)\,,
\end{split}
\]
so that
	\[
\begin{split}
\sum_{\ell=1}^{m_L(\rho)+1}& (-1)^{\un{(\ell>1)}}\sum_{\bj_\ell\in\{0,\ldots,k-1\}^\ell_<}J_{\rho,\ell}(\,\bj_\ell)\, \left\{
\text{RHS of \eqref{eq: Cor B}}\right\}
\\ & = -\sum_{\ell=1}^{m_L(\rho)+1} (-1)^{\un{(\ell>1)}}\sum_{m=1}^{\ell}\sum_{\bj_{m+1}\in\{0,\ldots,k\}^{m+1}_<} 
J_{\bullet_{k}\to_\ell\rho,m+1}(\,\bj_{m+1}) 
\\  & \qquad \qquad \qquad \cdot \int_{[0,1]_\geq^{k+1}}\left(\prod_{i=1}^{m+1}
			\nabla \Up_{G\left((\bullet_{k}\to_\ell\rho)^{(m+1)}_i\right)}(u_{j_i})\right) \cdot z^{[1]}\d \mathbf{t}_{k}\,.
\end{split}
\]
Now since 
\[	
	e(\mathcal{A}'(\bullet_{k}\to_\ell\rho))(\mathbf{t}_{k})= -
		\sum_{a=2}^{\ell+1}\sum_{\bj_a\in\{0,\ldots,k\}^a_<}J_{\bullet_{k}\to_\ell\rho,a}(\,\bj_a) \, \prod_{i=1}^{a}\nabla \Up_{G\left((\bullet_{k}\to_\ell\rho)^{(a)}_i\right)}(u_{j_i}) \, ,
		\]
with the change of variable $a=m+1$ we obtain
	\[
\begin{split}
& \sum_{\ell=1}^{m_L(\rho)+1} (-1)^{\un{(\ell>1)}}\sum_{\bj_\ell\in\{0,\ldots,k-1\}^\ell_<}J_{\rho,\ell}(\,\bj_\ell)\, \left\{
\text{RHS of \eqref{eq: Cor B}}\right\}
\\ & = \sum_{\ell=1}^{m_L(\rho)+1} (-1)^{\un{(\ell>1)}}\int_{[0,1]_\geq^{k+1}}
e(\mathcal{A}'(\bullet_{k}\to_\ell\rho))(\mathbf{t}_{k}) \cdot z^{[1]}\d \mathbf{t}_{k}
= \text{RHS of \eqref{eq:eAA'}} \,.
\end{split}
\]
The proof is complete
\end{proof}
	
	We are now ready for the:
	\begin{proof}[of Proposition \ref{Prop: recursive form remainder}]
		Fix $\tau_0,\ldots,\tau_k\in\cT$ and $w=\tau_1\cdots\tau_k$.
		Recalling the explicit expression for the grafting of forests derived in Corollary \ref{Cor: C and J} we have
		\begin{align*}
			&\frac{1}{\pi(w)}\Upsilon_{w\graf\tau_0}(z_1)=
			\\&=\frac{1}{\pi(w)}\sum_{p\in\Sigma_k} \sum_{\rho\in\cT_b^{p(I_{k-1})}}\!\!\sum_{\ell=1}^{m_L(\rho)+1} \!\!\sum_{\bj_\ell\in\{0,\ldots,k-1\}^\ell_<}\!\!(-1)^{k-m_R(\bullet_{p(k)}\to_\ell\rho)} J_{\rho,\ell}(\,\bj_\ell)\!\int_{[0,1]_\geq^{k}}\!\!\Up_{G( \bullet_{p(k)}\to_\ell\rho)}(z_1)\,\d\mathbf{t}_{k-1}.
		\end{align*}
		Notice that exchanging identical trees in $\bt=(\tau_1,\ldots,\tau_k)$ amounts to changing the decorations of $\rho$, while leaving the realisation $G(\rho)$ unchanged, recall Definition \ref{Def: link binary to rooted}. 
	
		Observing that $|\cP(I_k,\bt)|=\frac{k!}{\pi(w)}=\frac{|\Sigma_k|}{\pi(w)}$, recall \eqref{Eq: permutation factor}, we conclude, recalling Definition \ref{Def: sym H},
		\begin{align*}
			&\frac{1}{\pi(w)}\Upsilon_{w\graf\tau_0}(z_1)=
		\\&=\sum_{p\in\cP(I_k,\bt)}\sum_{\rho\in\cT_b^{p(I_{k-1})}}\!\!\sum_{\ell=1}^{m_L(\rho)+1}\!\!\sum_{\bj_\ell\in\{0,\ldots,k-1\}^\ell_<}(-1)^{k-m_R(\bullet_{p(k)}\to_\ell\rho)}J_{\rho,\ell}(\,\bj_\ell)\int_{[0,1]_\geq^{k}} \Up_{G(\bullet_{p(k)}\to_\ell\rho)}(z_1)\d\mathbf{t}_{k-1}.
		\end{align*}

		We adopt the notations $I^i_{k}:=\{1,\ldots,i-1,i+1,\ldots,k\}$ and $\bt^i:=(\tau_1,\ldots,\tau_{i-1},
		\tau_{i+1},\ldots,\tau_k)$. Resorting to  \eqref{Eq: def H} we can write
		\begin{align*}
			&H_{z_1z_2}(w\setminus\tau_i;\tau_0) \, \Up_{\tau_i}(z_1)=\sum_{q\in\cP(I^i_k,\bt^i)}\sum_{\rho\in\cT_b^{q(I^i_{k})}}(-1)^{k-1-m_R(\rho)}\int_{[0,1]_\geq^k}e(\mathcal{A}(\rho))(\mathbf{t}_{k-1}) \, \Up_{\tau_{i}}(z_1)\d\mathbf{t}_{k-1}\,.
		\end{align*}
		By Lemma \ref{lem: transposition} we obtain 
		\begin{align*}
			&\sum_{i=1}^k\frac{1}{\pi(\tau_i; w)}H_{z_1z_2}(w\setminus\tau_i;\tau_0) \, \Up_{\tau_i}(z_1)\\
			&=\sum_{p\in\cP(I_k,\bt)}	\sum_{\rho\in\cT_b^{p(I_{k-1})}}(-1)^{k-1-m_R(\rho)}\int_{[0,1]_\geq^k}e(\mathcal{A}(\rho))(\mathbf{t}_{k-1}) \, \Up_{\tau_{p(k)}}(z_1)\d\mathbf{t}_{k-1}\,.
		\end{align*}
		Observe that 
		\[
		m_R(\bullet_i\to_1\rho)=m_R(\rho)+\un{(\ell=1)}\,,
		\]
		which implies $\un{(\ell>1)}+k-1-m_R(\rho)=k-m_R(\bullet_i\to_\ell\rho)$, for all $\ell>0$. As a result
		\begin{align*}
			\text{LHS of \eqref{eq:recurema}} &= 
			\sum_{p\in\cP(I_k,\bt)}	\sum_{\rho\in\cT_b^{p(I_{k-1})}}(-1)^{k-1-m_R(\rho)}\left[\int_{[0,1]_\geq^k}\,e(\mathcal{A}(\rho))(\mathbf{t}_{k-1}) \, \Up_{\tau_{p(k)}}(z_1)\d\mathbf{t}_{k-1} \right.
			\\&-\left.\sum_{\ell=1}^{m_L(\rho)+1}\sum_{\bj_\ell\in\{0,\ldots,k-1\}^\ell_<}(-1)^{\un{(\ell>1)}}\int_{[0,1]_\geq^{k}}J_{\rho,\ell}(\,\bj_\ell)\, \Up_{G(\bullet_{p(k)}\to_\ell\rho)}(z_1)\d \mathbf{t}_{k-1}\right].
		\end{align*}
		Since $\Up_{G(\bullet_{p(k)}\to_\ell\rho)}$ in the last expression is computed at $z_1$ instead of $u_{j_1}$, we can not 
		apply Lemma \ref{lem: Upsilon descent} yet, but we need an additional recentering of the second term:
		\begin{align}\label{Eq: integration step}
			& \text{LHS of \eqref{eq:recurema}} 
			=\sum_{p\in\cP(I_k,\bt)}	\sum_{\rho\in\cT_b^{p(I_{k-1})}}(-1)^{k-1-m_R(\rho)}\left[\int_{[0,1]_\geq^k}\,e(\mathcal{A}(\rho))(\mathbf{t}_{k-1}) \, \Up_{\tau_{p(k)}}(z_1)\d\mathbf{t}_{k-1}\right.\nonumber
			\\&\quad\left.-\sum_{\ell=1}^{m_L(\rho)+1} (-1)^{\un{(\ell>1)}} \sum_{\bj_\ell\in\{0,\ldots,k-1\}^\ell_<} J_{\rho,\ell}(\,\bj_\ell) \int_{[0,1]_\geq^{k}}\Up_{G(\bullet_{p(k)}\to_\ell\rho)}(u_{j_1})\d \mathbf{t}_{k-1}\right.
			\\&\quad\left.+\sum_{\ell=1}^{m_L(\rho)+1} (-1)^{\un{(\ell>1)}} \sum_{\bj_\ell\in\{0,\ldots,k-1\}^\ell_<} J_{\rho,\ell}(\,\bj_\ell) \int_{[0,1]_\geq^{k}}\int_0^{t_{j_1}}
			\nabla\Up_{G(\bullet_{p(k)}\to_\ell\rho)}(u_{k})\cdot z^{[1]} \d t_k\d \mathbf{t}_{k-1}\right].\nonumber
		\end{align}
		By Lemma \ref{lem: Upsilon descent}, the first two lines in the right-hand side of \eqref{Eq: integration step} coincide with the RHS of \eqref{eq:eAA'}, namely
		\begin{align*}
			&\sum_{\ell=1}^{m_L(\rho)+1} (-1)^{\un{(\ell>1)}}
			\int_{[0,1]_\geq^{k+1}}e\left(\mathcal{A}'(\bullet_{p(k)}\to_\ell\rho)\right)(\mathbf{t}_k)\cdot z^{[1]} \d \mathbf{t}_k\,. 
		\end{align*}
		Invoking \eqref{Eq: int exchange 1}, the third line in the right-hand side of \eqref{Eq: integration step} is equal to
		\begin{align}
			&\sum_{\ell=1}^{m_L(\rho)+1} (-1)^{\un{(\ell>1)}} \sum_{\bj_\ell\in\{0,\ldots,k-1\}^\ell_<}J_{\rho,\ell}(\,\bj_\ell)\sum_{i=j_1+1}^k\int_{[0,1]_\geq^{k+1}}\nabla\Up_{G(\bullet_{p(k)}\to_\ell\rho)}(u_{i})\cdot z^{[1]} \d \mathbf{t}_k \nonumber
			\\&\overset{\eqref{Eq: 1}}{=}\sum_{\ell=1}^{m_L(\rho)+1}(-1)^{\un{(\ell>1)}}\sum_{j_1\leq k}J_{\bullet_{p(k)}\to_\ell\rho,1}(j_1) \int_{[0,1]_\geq^{k+1}}\nabla\Up_{G(\bullet_{p(k)}\to_\ell\rho)}(u_{j_1})\cdot z^{[1]} \d \mathbf{t}_k\,. \label{eq: temp holder taylor1}
		\end{align} 
		We conclude by \eqref{eq:e(A)}-\eqref{Eq: A'} that
		\begin{align*}
						& \text{LHS of \eqref{eq:recurema}} = 
			\sum_{p\in\cP(I_k,\bt)}	\sum_{\rho\in\cT_b^{p(I_{k-1})}}(-1)^{k-1-m_R(\rho)}\nonumber
			\\&\hspace{1cm}\left[\sum_{\ell=1}^{m_L(\rho)+1}(-1)^{\un{(\ell>1)}}\sum_{j_1\leq k} J_{\bullet_{p(k)}\to_\ell\rho,1}(j_1) 
			\int_{[0,1]_\geq^{k+1}} \nabla\Up_{G(\bullet_{p(k)}\to_\ell\rho)}(u_{j_1})\cdot z^{[1]} \d \mathbf{t}_k \right.\nonumber
			\\&\qquad\hspace{1cm}\left.+\sum_{\ell=1}^{m_L(\rho)+1} (-1)^{\un{(\ell>1)}}
			\int_{[0,1]_\geq^{k+1}}e\left(\mathcal{A}'(\bullet_{p(k)}\to_\ell\rho)\right)(\mathbf{t}_k)\cdot z^{[1]} \d \mathbf{t}_k
			\right]\nonumber
			\\&=\sum_{p\in\cP(I_k,\bt)}	\sum_{\rho\in\cT_b^{p(I_{k-1})}}\sum_{\ell=1}^{m_L(\rho)+1}(-1)^{k-m_R(\bullet_{p(k)}\to_\ell\rho)}\int_{[0,T]^{k+1}_\geq}e\left(\cA(\bullet_{p(k)}\to_\ell\rho)\right)(\mathbf{t}_k)\cdot z^{[1]}\d\mathbf{t}_k\,. 
		\end{align*}
		Now, an application of Lemma \ref{Lem: construction binary trees} with the choice $F(\rho)=(-1)^{k-m_R(\rho)}e(\cA(\rho))$
		guarantees that all possible insertions of $\bullet_{p(k)}$ over all trees in $\cT_b^{p(I_{k-1})}$ span $\cT_b^{p(I_k)}$, so that in particular
		\begin{equation}\label{eq: temp holder taylor3}
			\sum_{\rho\in\cT_b^{p(I_{k-1})}}\sum_{\ell=1}^{m_L(\rho)+1}(-1)^{k-m_R(\bullet_{p(k)}\to_\ell\rho)}e\left(\cA(\bullet_{p(k)}\to_\ell\rho)\right)(\mathbf{t}_k)
			=	\sum_{\rho\in\cT_b^{p(I_{k})}}(-1)^{k-m_R(\rho)} e(\cA(\rho))(\mathbf{t}_k)\,.
		\end{equation}
		As a result the LHS of \eqref{eq:recurema} equals
		\begin{align*}
			\sum_{p\in\cP(I_k,\bt)}	\sum_{\rho\in\cT_b^{p(I_{k})}}(-1)^{k-m_R(\rho)}\int_{[0,T]^{k+1}_\geq}e(\cA(\rho))(\mathbf{t}_k)\cdot z^{[1]}\d\mathbf{t}_k
		=H_{z_1z_2}(w,\tau_0)\cdot z^{[1]}\,,
		\end{align*}
		by the definition \eqref{Eq: def H} of $H_{z_1z_2}(w,\tau_0)$.
	\end{proof}

	We can finally prove Theorem \ref{Thm: Taylor Lipschitz}.
	\begin{proof}[ of Theorem \ref{Thm: Taylor Lipschitz}]
		We proceed by induction on $m\in\N$. For $m=0$ the left-hand side of \eqref{Eq: identity remainder H} reads $\Up_{\tau_0}(z_2)-\Up_{\tau_0}(z_1)$ while, on account of the definition of $H_{z_1z_2}(\one; \tau_0)$, the right-hand side is $\int_s^t\nabla G(\bullet_0)(u_1) \cdot z^{[1]}\d t_1$. These evidently coincide, hence proving the first step of the induction. Now suppose that \eqref{Eq: identity remainder H} holds true for an $m\in\N$.
		By \eqref{Eq: inductive B} we have
			\begin{align*}
			B^{[m+2]}_{z_1z_2}(\tau) 
			=\underbrace{\Up_{\tau_0}(z_2)-\sum_{w\in\cF^{\leq m}}\frac{\Up_{w\curvearrowright\tau_0}(z_1)}{\pi(w)} \, \X(w)}_{B^{[m+1]}_{z_1z_2}(\tau) }-
			\underbrace{\sum_{w\in\cF^{= m+1}}\frac{\Up_{w\curvearrowright\tau_0}(z_1)}{\pi(w)} \, \X(w)}_{C}\,.
		\end{align*}
		By induction we can write
		\begin{align*}
			B^{[m+2]}_{z_1z_2}(\tau)=\underbrace{\sum_{w\in\cF^{\leq m}}H_{z_1z_2}(w;\tau_0)\cdot z^{[m+1-|w|]} \, \X(w)}_{D}-C\,.
		\end{align*}
		Recalling that by \eqref{Eq: inductive z}
		\begin{equation*}
			z^{[i]}=z^{[i+1]}+\sum_{\gamma\in\cT^{=i}}\Up_{\gamma}(z_1) \, \X(\gamma)\,,
		\end{equation*}
		we can raise the level of the remainders $z^{[i]}$ from $i$ to $i+1$, for $i=1,\ldots, m+1$, so that 
		\begin{align*}
			D&=\underbrace{\sum_{w\in\cF^{\leq m}}H_{z_1z_2}(w;\tau_0)\cdot z^{[m+2-|w|]} \, \X(w)}_{E}
		+\underbrace{\sum_{w\in\cF^{\leq m}}H_{z_1z_2}(w;\tau_0) \, \X(w)\!\!\sum_{\gamma\in\cT^{= m+1-|w|}}\Up_{\gamma}(z_1) \, \X(\gamma)}_{F}.
		\end{align*}
		As a result, $B^{[m+2]}_{z_1z_2}(\tau) =E+F-C$. What is missing from the sum in $E$ is the case $|w|=m+1$.
		Therefore it remains to prove that 
		\[
		F-C = \sum_{w\in\cF^{= m+1}}H_{z_1z_2}(w;\tau_0)\cdot z^{[m+2-|w|]} \, \X(w)\,.
		\]
		Now we have
		\begin{align*}
			F&=\sum_{w\in\cF^{\leq m}}\sum_{\gamma\in\cT^{= m+1-|w|}}H_{z_1z_2}(w;\tau_0) \, \Up_{\gamma}(z_1) \, \X(w\gamma)
\\
			&=\sum_{k=1}^{m+1}\sum_{w=\tau_1\cdots\tau_k\in\cF^{= m+1}}\sum_{i=1}^k\frac{1}{\pi(\tau_i; w)}H_{z_1z_2}(w\setminus \tau_i;\tau_0) \, \Up_{\tau_i}(z_1) \, \X(w)\,.
		\end{align*}
		Hence, by the recursive property of $H$ stated in Proposition \ref{Prop: recursive form remainder}, $F-C$ takes the form
		\begin{align*}
			&\sum_{k=1}^{m+1}\sum_{w=\tau_1\cdots\tau_k\in\cF^{= m+1}}\left[\sum_{i=1}^k\frac{1}{\pi(\tau_i; w)}H_{z_1z_2}(w\setminus \tau_i;\tau_0) \, \Up_{\tau_i}(z_1)-\frac{1}{\pi(w)} \Up_{w\curvearrowright\tau_0}(z_1) \right] \X(w)
			\\ & = \sum_{w\in\cF^{= m+1}} H_{z_1z_2}(w;\tau_0)\cdot z^{[1]} \, \X(w)\, ,
		\end{align*}
		which is exactly the term corresponding to $|w|=m+1$ missing from the sum in $E$ above.
		We conclude that
		\begin{align*}
			B^{[m+2]}_{z_1z_2}(\tau) &=E+F-C
			=\sum_{w\in\cF^{\leq m+1}}H_{z_1z_2}(w;\tau_0)\cdot z^{[m+2-|w|]} \, \X(w)\,,
		\end{align*}
		which proves our claim.
	\end{proof}
	The compact formula in \eqref{Eq: identity remainder H} allows to obtain a priori estimates for solution remainders of Davie solutions when all $\Upsilon_\tau$ are assumed to be Lipschitz continuous for $\tau\in\cT^{\le N}$. However, we will show that one can weaken such assumptions by imposing suitable Hölder-continuity of $\Upsilon_\tau$ when $\tau\in\cT^{= N}$, see Assumption \ref{As: regularity assumption} below. In order to cover this more general case, we must rewrite \eqref{Eq: identity remainder H} keeping the highest order term non-integrated. 
	\begin{corollary}\label{Cor: Taylor Lipschitz}
		Under the hypothesis of Theorem \ref{Thm: Taylor Lipschitz} we have
		\begin{align*}
			&\Up_{\tau_0}(z_2)-\sum_{w\in\cF^{\leq m}}\frac{\Up_{w\curvearrowright\tau_0}(z_1)}{\pi(w)} \, \X(w)\nonumber\\
			&=\sum_{w\in\cF^{\leq m-1}}H_{z_1z_2}(w;\tau_0)\cdot z^{[m+1-|w|]} \, \X(w)+\sum_{k=1}^{m}\sum_{w=\tau_1\cdots\tau_k\in\cF^{= m}}\sum_{p\in\cP(I_k,\bt)}\left[\sum_{\rho\in\cT_b^{p(I_{k-1})}}\right.
		\\&\left.\sum_{\ell=1}^{m_L(\rho)+1}\!\!\!\!\sum_{\bj_\ell\in\{0,\ldots,k-1\}^\ell_<}\!\!\!\!J_{\rho,\ell}(\,\bj_\ell)\,(-1)^{k-m_R(\bullet_{p(k)}\to_\ell\rho)}\int_{[0,1]_\geq^k} \!\!\left(\Up_{G(\bullet_{p(k)}\to_\ell\rho)}(u_{j_1})-\Up_{G(\bullet_{p(k)}\to_\ell\rho)}(z_1)\right)\d\mathbf{t}_{k-1}\right.
			\\&\qquad \left.+\sum_{\rho\in\cT_b^{p(I_{k})}}\int_{[0,1]_\geq^{k+1}}(-1)^{k-m_R(\rho)}\,e(\mathcal{A}'(\rho))(\mathbf{t}_k)\cdot z^{[1]}\d\mathbf{t}_k\right]\X(w)\,.	
		\end{align*}
	\end{corollary}
	\begin{proof}
		The result follows by a procedure analogous to the proof of Theorem \ref{Thm: Taylor Lipschitz}, the only difference being that the term involving trees of the highest order should be left not integrated. Alternatively, we have by \eqref{Eq: identity remainder H}
			\begin{align*}
			&\Up_{\tau_0}(z_2)-\sum_{w\in\cF^{\leq m}}\frac{\Up_{w\curvearrowright\tau_0}(z_1)}{\pi(w)} \, \X(w)
			\\ & =\sum_{w\in\cF^{\leq m-1}}H_{z_1z_2}(w;\tau_0) \cdot z^{[m+1-|w|]}\, \X(w)
			+\sum_{w\in\cF^{=m}}H_{z_1z_2}(w;\tau_0) \cdot z^{[1]}\, \X(w)\,.
		\end{align*}
		We exploit \eqref{Eq: def H} and \eqref{Eq: A'} to write
		\begin{align*}
		&\sum_{w\in\cF^{=m}}H_{z_1z_2}(w;\tau_0) \cdot z^{[1]}\, \X(w)\\
		&=\sum_{k=1}^m\sum_{w=\tau_1\cdots\tau_k\in\cF^{=m}}\sum_{p\in\cP(I_k,\bt)}\sum_{\rho\in\cT_b^{p(I_{k})}}\int_{[0,1]_\ge^{k+1}}(-1)^{k-m_R(\rho)}\,e(\mathcal{A}(\rho))(\mathbf{t}_k)\d\mathbf{t}_k\cdot z^{[1]}\, \X(w)
		\\&=\sum_{k=1}^m\sum_{w=\tau_1\cdots\tau_k\in\cF^{=m}}\sum_{p\in\cP(I_k,\bt)}\sum_{\rho\in\cT_b^{p(I_{k})}}\left[\sum_{j_1\leq k}J_{\rho,1}(j_1)\int_{[0,1]_\ge^{k+1}}(-1)^{k-m_R(\rho)}\,\nabla\Up_{G(\rho)}(u_{j_1})\d\mathbf{t}_k\cdot z^{[1]}\,\right.
		\\&\hspace{2cm}\left.+\int_{[0,1]_\ge^{k+1}}(-1)^{k-m_R(\rho)}\,e(\mathcal{A}'(\rho))(\mathbf{t}_k)\d\mathbf{t}_k\cdot z^{[1]}\,\right] \X(w)\,.
		\end{align*}
		We can then repeat steps from \eqref{eq: temp holder taylor1} to \eqref{eq: temp holder taylor3} in reverse to obtain
		\begin{align*}
			&\sum_{\rho\in\cT_b^{p(I_{k})}}\sum_{j_1\leq k}J_{\rho,1}(j_1)\int_{[0,1]_\ge^{k+1}}(-1)^{k-m_R(\rho)}\,\nabla\Up_{G(\rho)}(u_{j_1})\d\mathbf{t}_k\cdot z^{[1]}\,
			\\&=\sum_{\rho\in\cT_b^{p(I_{k-1})}}\!\!\!\!\sum_{\ell=1}^{m_L(\rho)+1}\!\!\!\!\sum_{\bj_\ell\in\{0,\ldots,k-1\}^\ell_< }\!\!\!\!J_{\rho,\ell}(\,\bj_\ell)(-1)^{k-m_R(\bullet_{p(k)}\to_\ell\rho)}\int_{[0,1]_\ge^k}\int_0^{t_{j_1}}\,\nabla\Up_{G(\bullet_{p(k)}\to_\ell\rho)}(u_k)\d\mathbf{t}_k\cdot z^{[1]}
			\end{align*}
			\vspace{-0.5cm}
			\begin{align*}
			&=\sum_{\rho\in\cT_b^{p(I_{k-1})}}\sum_{\ell=1}^{m_L(\rho)+1}\!\!\!\!\sum_{\bj_\ell\in\{0,\ldots,k-1\}^\ell_<}\!\!\!\!J_{\rho,\ell}(\,\bj_\ell)\,(-1)^{k-m_R(\bullet_{p(k)}\to_\ell\rho)}\hspace{1cm}
			\\&\hspace{5cm}\int_{[0,1]_\geq^k} \!\!\left(\Up_{G(\bullet_{p(k)}\to_\ell\rho)}(u_{j_1})-\Up_{G(\bullet_{p(k)}\to_\ell\rho)}(z_1)\right)\d\mathbf{t}_{k-1},
		\end{align*}
		which leads to the sought result.
	\end{proof}

	\section{A priori estimates}\label{Sec: well-posedness}
	The first part of the paper was entirely dedicated to the derivation of explicit expressions for the elementary differential remainders $B^{[m+1]}(\tau_0)$ for $\tau_0\in\cT^{\leq N}$ and $m\leq N-|\tau_0|$. In this final section we apply these representation formulae to the derivation of a priori estimates on the Davie solution under the following minimal assumptions on the coefficients $\Upsilon_\tau$.
\begin{assumption}\label{As: regularity assumption}
	For $\alpha\in(0,1]$, $N:=\lfloor\frac{1}{\alpha}\rfloor$  and $\beta\in(\frac{1}{\alpha}-N,1]$,
	assume that
	\vspace{0.2cm}
	\begin{itemize}
		\item $\Up_\tau$ is globally Lipschitz continuous for all $\tau\in\cT^{\le N-1}$,
		\item $\Up_\tau\in\mathcal{C}^\beta$ for all $\tau\in\cT^{=N}$.
	\end{itemize}
\end{assumption}

	First, we introduce suitable analytic tools and preliminary results which will be needed throughout this section. 
	We consider a modified version of the H\"older-like norm $\| \cdot \|_{\beta}$ by introducing a one-parameter weight that helps to tame the growth of the norm over large volumes. This tool will be useful when trying to turn local results into global ones. As we will detail below, the advantage resides in the possibility of considering arbitrarily large intervals $[0, T]$, at the expense of fixing a sufficiently small weight parameter $\mu > 0$.
	
	\begin{definition}[Weighted Norms]\label{Def: weighted norms}
		Given $F \in C ([0, T]^2_{\leq}, \mathbb{R}^n)$, $G \in C ([0,
		T]^3_{\leq}, \mathbb{R}^n)$ and $\beta,\mu \in (0, \infty)$ we define
		\begin{align*}
			&\| F \|_{\beta, \mu}:=\sup_{0 \leq s \leq t \leq
				T} \mathbbm{1}_{\{ 0 < t - s \leq \mu \}} e^{- \frac{t}{\mu}}
			\frac{| F_{s t} |}{(t - s)^{\beta}}\,,  \\
			&\| G \|_{\beta, \mu} := \sup_{0 \leq s \leq u \leq
				t \leq T} \mathbbm{1}_{\{ 0 < t - s \leq \mu \}} e^{-
				\frac{t}{\mu}} \frac{| G_{s u t} |}{(t - s)^{\beta}}\,.  
		\end{align*}
	\end{definition}
	Notice that the standard norm is restored by sending $\mu$ to infinity:
	\[ \lim_{\mu \rightarrow \infty} \| \cdot \|_{\beta, \mu} = \| \cdot
	\|_{\beta} . \]
	
	\begin{remark}\label{Rem: equivalence weighted norm}
		We stress that $\| \cdot \|_{\beta, \mu}$ is, in general, just a seminorm, since the argument may have a support disjoint from the region identified by the indicator function. In addition, when $F = \delta f$ for $f\in\mathcal{C}^\beta$, the seminorm $\| \cdot \|_{\beta,\mu}$ is equivalent to $\| \cdot \|_{\beta}$. Indeed, for all $\beta\in(0,1)$ and $\mu>0$, a direct calculation entails 
		\begin{equation*}
			\|\delta f\|_{\beta,\mu}\leq \|\delta f\|_{\beta}\leq \left(1+\frac{T}{\mu}\right)e^{\frac{T}{\mu}}\|\delta f\|_{\beta,\mu}\,.
		\end{equation*}
	\end{remark}
	
		Now we are in the position to state the desired a priori estimates of the solution remainders. 
	\begin{theorem}[A priori estimate]\label{th: a priori branched}
		Under Assumption \ref{As: regularity assumption}, for any solution $Z$ of \eqref{eq: sol Davie branched} we have
		\begin{equation*}
			\|Z^{[N+1]}\|_{(N+\beta)\alpha,\mu} \lesssim_{\X,\sigma,T} 
			(1\vee\| Z^{[1]}\|_{\alpha,\mu})+\sum_{\ell=2}^{m+1}\|Z^{[\ell]}\|_{\ell\alpha,\mu}\, ,
		\end{equation*}
		and, if either $T$ or $\mu$ are sufficiently small, 
		\begin{align}\nonumber
			&(1\vee \| Z^{[1]} \|_{\alpha, \mu})+\sum_{i=2}^N \| Z^{[i]} \|_{i \alpha, \mu}  \le
			\\ & \le 2\left[ (1\vee (|\Up_{\bullet}(Z_0)| \, \|\mathbb{X}(\bullet)\|_{\alpha}))+ \sum_{\tau\in\cT^{\le N}\setminus\{\bullet\}} |\tau| \,|\Up_{\tau}(Z_0)| \, \|\mathbb{X}(\tau)\|_{|\tau|\alpha}\right].\label{eq: a priori branched}
		\end{align}
		If moreover $\beta=1$, namely if $\Up_\tau$ is globally Lipschitz continuous for all $\tau\in\cT^{\le N}$, 
		then for any solution $Z$ of \eqref{eq: sol Davie branched},
		if either $T$ or $\mu$ are small enough, we have
		\begin{equation}\label{eq: a priori branched2}
			\sum_{i=1}^N \| Z^{[i]} \|_{i \alpha, \mu} \le 2\sum_{\tau\in\cT^{\le N}} |\tau| \,|\Up_{\tau}(Z_0)| \, \|\mathbb{X}(\tau)\|_{|\tau|\alpha}\,.
		\end{equation}
	\end{theorem}

	We will also need a weighted version of the supremum norm.
	\begin{definition}[Weighted supremum norm]
		For $f \in C ([0, T], \mathbb{R}^n)$ we define
		\[ \| f \|_{\infty, \mu} :=\sup_{0 \leq t \leq T} e^{-
			\frac{t}{\mu}} | f_t | . \]
	\end{definition}
	For a number of useful results on weighted norms, see Appendix \ref{App: analytic tools}

	\subsection{The Sewing bound}\label{Sec: sewing bound}
	
	The main analytical tool in this work is the celebrated Sewing Lemma \cite{gubinelli2004controlling, FeDe06}. More precisely, to establish a priori estimates we do not need the 
	full statement, but only the so-called Sewing bound, see Theorem \ref{Thm: sewing bound}, a terminology borrowed from \cite{CGZ25}. First we give a technical lemma on convergence of generalised Riemann sums.  
	\begin{lemma}
		Consider $R \in C ([0, T]^2_{\leq}, \mathbb{R}^n)$ such that
		$R_{{st}} = o (t - s)$. For any $0 \leq a < b
		\leq T$ let us take any sequence of partitions $(\mathsf{P}_n)_{n
			\geqslant 0}$ of $[a, b]$ with vanishing mesh, i.e. $\lim_{n \rightarrow
			\infty} | \mathsf{P}_n | = 0$. If we define $I_{\mathsf{P}_n} (R) :=
		\sum_{i = 1}^{\#\mathsf{P}} R_{t_{i - 1} t_i}$ we have
		\[ \lim_{n \rightarrow \infty} I_{\mathsf{P}_n} (R) = 0\,. \]
		Moreover, if $\mathsf{P}_0 = \{ a, b \}$, than
		\begin{equation}
			R_{a b} = \sum_{n = 0}^{\infty} (I_{\mathsf{P}_n} (R) - I_{\mathsf{P}_{n
					+ 1}} (R))\,, \qquad 0 \leq a < b \leq T\,. \label{Rab1}
		\end{equation}
	\end{lemma}
	\begin{proof}
		For $\mathsf{P}_n = \{ a = t_0^n < \ldots < t^n_{\#\mathsf{P}_n} = b
		\}$, by definition we have
		\[ | I_{\mathsf{P}_n} (R) | \leq \sum_{i = 1}^{\#\mathsf{P}_n} |
		R_{t_{i - 1}^n t_i^n} | \leq \left\{ \max_{j = 1, \ldots,
			\#\mathsf{P}_n} \frac{| R_{t_{i - 1}^n t_i^n} |}{(t_j^n - t_{j -
				1}^n)} \right\} \sum_{j = 1}^{\#\mathsf{P}_n} (t_j^n - t_{j - 1}^n)\,,
		\]
		which implies $| I_{\mathsf{P}_n} (R) | \rightarrow 0$ as $n \rightarrow
		\infty$, since the sum is finite, $R_{s t} = o (t - s)$ and $\max_{j = 1,
			\ldots, \#\mathsf{P}_n} (t_j^n - t_{j - 1}^n) = | \mathsf{P}_n |$ tends
		to $0$.
		Finally, since $\lim_{m \rightarrow \infty} I_{\mathsf{P}_m} (R) = 0$ and
		$I_{\mathsf{P}_0} (R) = R_{a b}$ for $\mathsf{P}_0 = \{ a, b \}$, we
		have
		\begin{align*}
			R_{a b} & =  \lim_{m \rightarrow \infty} (I_{\mathsf{P}_0} (R) -
			I_{\mathsf{P}_m} (R)) = \lim_{m \rightarrow \infty} \sum_{i = 0}^{m - 1}
			(I_{\mathsf{P}_i} (R) - I_{\mathsf{P}_{i+1}} (R)) = \sum_{i = 0}^{\infty} (I_{\mathsf{P}_i} (R) - I_{\mathsf{P}_{i+1}} (R))\,,
		\end{align*}
		which proves \eqref{Rab1}.
	\end{proof}
	
	\begin{theorem}[Sewing bound]\label{Thm: sewing bound}
		Consider $R \in C ([0, T]^2_{\leq}, \mathbb{R}^n)$ such that
		$R_{s t} = o (t - s)$. Then, for all $\eta \in (1, \infty)$,
		\begin{equation}
			\| R \|_{\eta} \leq K_{\eta} \| \delta R \|_{\eta}\,, \hspace{1.5cm} K_{\eta} :=(1 - 2^{1 - \eta})^{- 1} .
			\label{SewingBound}
		\end{equation}
	\end{theorem}
	\begin{proof}
		Assume $\| \delta R \|_{\eta} < \infty$ for a given $\eta > 1$, otherwise the statement trivially holds true.	Fix $0 \leq a < b \leq T$. For $n \geqslant 0$ we call
		$\mathsf{P}_n :=\left\{ t_i^n :=a + \frac{i}{2^n} (b - a) : 0
		\leq i \leq 2^n \right\}$ the dyadic partitions of $[a, b]$.
		Note that $\mathsf{P}_0 $ is the trivial partition. By \eqref{Rab1} we
		can write
		\begin{equation}
			| R_{a b} | \leq \sum_{n = 0}^{\infty} | I_{\mathsf{P}_n} (R) -
			I_{\mathsf{P}_{n+ 1}} (R) |\, . \label{Rab}
		\end{equation}
		Now note that if we remove a single point $t_i$ from a partition
		$\mathsf{P}= \{ t_0 < t_1 < \ldots < t_q \}$ we obtain a new partition
		$\mathsf{P}^{'}$ and we can write
		\begin{equation}
			I_{\mathsf{P}^{'}} (R) - I_{\mathsf{P}} (R) = \delta R_{t_{i - 1} t_i
				t_{i + 1}} . \label{increments}
		\end{equation}
		Observe that removing all the points $t_{2 j + 1}^{n+1}$, with $0
		\leq j \leq 2^n - 1$, from $\mathsf{P}_{n+1}$ we obtain
		$\mathsf{P}_n$, so that by \eqref{increments} we have
		\begin{align*}
			| I_{\mathsf{P}_{n+1} } (R) - I_{\mathsf{P}_n} (R) | &\leq 
			\sum_{j = 0}^{2^n - 1} | \delta R_{t^{n+1}_{2 j} t_{2 j + 1}^{n+1}
				t_{2 j + 2}^{n+1}} |
				\\ & \leq 2^n \| \delta R \|_{\eta} \left( \frac{2 (b - a)}{2^{n+
					1}} \right)^{\eta} =  2^{- (\eta - 1) n} \| \delta R \|_{\eta} (b - a)^{\eta} \,. 
		\end{align*}
		Substituting in \eqref{Rab} we get, for $0 \leq a < b \leq T$,
		\begin{align*}
			| R_{a b} | & \leq \sum_{n = 0}^{\infty} 2^{- (\eta - 1) n} \|
			\delta R \|_{\eta} (b - a)^{\eta} = (1 - 2^{1 - \eta})^{- 1} \| \delta R \|_{\eta} (b - a)^{\eta} = K_{\eta} \| \delta R \|_{\eta} (b - a)^{\eta},
		\end{align*}
		which proves \eqref{SewingBound}.
	\end{proof}
	
	\begin{remark}
		An extension of the sewing bound to the weighted case follows suit from recalling that the weighted seminorm is always controlled by the non-weighted one. Indeed, for  $0 \leq s \leq t \leq T$, the proof of \eqref{SewingBound} entails that $| R_{a b} |
		\leq K_{\eta} \| \delta R \|_{\eta, [s, t]} (t - s)^{\eta}$. Hence
		\[ e^{- \frac{t}{\mu}} \frac{| R_{s t} |}{(t - s)^{\eta}} \leq e^{-
			\frac{t}{\mu}} K_{\eta} \| \delta R \|_{\eta, [s, t]} \leq
		K_{\eta} \| \delta R \|_{\eta, \mu}\,, \]
		which, by taking the supremum over $0
		\leq s \leq t \leq T$, $t - s \leq \mu$ and under the hypothesis of Theorem \ref{Thm: sewing bound}, leads to
		\begin{align}
			\| R \|_{\eta, \mu} \leq K_{\eta} \| \delta R \|_{\eta, \mu}\,, \quad \quad  K_{\eta} :=(1 - 2^{1 - \eta})^{- 1} \,.
			\label{weighted Sewing bound}
		\end{align}
	\end{remark}
	
	\subsection{Derivation of the a priori bound}
	
	In this section we establish a priori estimates for the Davie solution of the rough equation \eqref{eq: sol Davie branched} driven by an $\alpha$-branched rough path for generic $\alpha\in(0,1]$ as stated in Theorem \ref{th: a priori branched0} under Assumption \ref{As: regularity assumption}, exploiting the generalised Taylor expansions derived in Section \ref{Sec: Algebraic formula for the remainder}. Below we will highlight how its expression is crucial for our proof. 
	
	First we derive the control around the diagonal for the errors $Z^{[i]}$ defined in \eqref{Eq: Z remainders}.
	\begin{proposition}\label{Prop: regularity Z}
		Let $Z$ be a solution of \eqref{eq: sol Davie branched} and $\sigma\in C^{N-1}$. Then, for all $i\in\{1,\ldots,N\}$, $\|Z^{[i]}\|_{i \alpha}<+\infty$.
	\end{proposition}
	\begin{proof}
		By the definition of a Davie solution, see \eqref{eq: sol Davie branched},
		\[ | \delta Z_{s t} | \leq \sup_{0 \leq s' \leq T} \sup_{\tau'\in\cT^{\leq N}}| \Up_{\tau'}(Z_{s'})|  \sum_{\tau\in\cT^{\leq N}} | \mathbb{X}_{s t}(\tau)
		| + |Z^{[N+1]}_{st}|, \qquad |Z^{[N+1]}_{st}|=o (t - s). \]
		Then it suffices to show that the constant $C := \sup_{0 \leq s' \leq T} \sup_{\tau'\in\cT^{\leq N}}| \Up_{\tau'}(Z_{s'})|$ is finite. Since  $\Up_\tau$ displays at most the $(N-1)$-st derivative of $\sigma$, it is continuous for every $\tau\in\cT^{\leq N}$. Hence it remains to prove that $Z$ is bounded.
		Let us fix $\bar{\delta} > 0$ such that $| Z^{[N+1]}_{st}| \leq 1$ for
		all $0 \leq s \leq t \leq T$ with $| t - s | \leq
		\bar{\delta}$. Since we can write $[0, T]$ as a finite union of intervals
		$[\bar{s}, \bar{t}]$ such that $\bar{t} - \bar{s} \leq \bar{\delta}$,
		we will focus only on one interval of this kind. By the definition of
		solution we can write
		\[ \sup_{t \in [\bar{s}, \bar{t}]} | Z_t | \leq | Z_{\bar{s}} | +
		\sum_{\tau\in\cT^{\leq N}} | \Upsilon_{\tau}(Z_{\bar s})| \sup_{t \in [\bar{s},
			\bar{t}]} | \mathbb{X}_{s t}(\tau) | + 1 < \infty \,. \]
		We proved that $Z$ is bounded, therefore $C < \infty$. Observe that we
		can write for $i=2,\ldots,N$
		\[ 
		| Z^{[i]}_{s t} | \leq C \sum_{\substack{\tau\in\cT\\ i\leq|\tau|\leq N}} | \mathbb{X}_{s t}(\tau) | + 
		|Z^{[N+1]}_{st}|\, . 
		\]
		Now, since $\|\mathbb{X}(\tau)\|_{|\tau| \alpha}<+\infty$, we conclude that $\|Z^{[i]}\|_{i \alpha}<+\infty$.
	\end{proof}
	
	Our strategy for the derivation of an a priori estimate on the remainder $Z^{[N+1]}$ relies on the Sewing bound, which involves the increments $\delta Z^{[N+1]}$. Given the convenient form of $\delta Z^{[N+1]}$ derived in \eqref{dZN+1}, we first deduce a control on the norm of the elementary differential remainders $B^{[m+1]}$ in terms of lower order solution remainders. 
	
	\begin{lemma}\label{lem: H estimate}
		Under Assumption \ref{As: regularity assumption}, for $B^{[m+1]}$ defined in \eqref{eq:Bn} and $m=0,\ldots,N-1$, we have
		\begin{equation}\label{eq:5.8.1}
			\|B^{[m+1]}(\tau)\|_{(m+\beta)\alpha,\mu}\lesssim_{\X,\sigma} \sum_{\ell=2}^{m}\|Z^{[\ell]}\|_{\ell\alpha,\mu}+(\mu\wedge T)^{1-\beta}\left(1\vee\| Z^{[1]}\|_{\alpha,\mu}\right).
		\end{equation}
		If $\beta=1$, namely if $\Up_\tau$ is globally Lipschitz continuous for all $\tau\in\cT^{\leq N}$, then
		\begin{equation}\label{eq:5.8.2}
			\|B^{[m+1]}(\tau)\|_{(m+\beta)\alpha,\mu}\lesssim_{\X,\sigma} \sum_{\ell=1}^{m}\|Z^{[\ell]}\|_{\ell\alpha,\mu}.
		\end{equation}
	\end{lemma}
	\begin{proof}
		We rely on the expression derived in Corollary \ref{Cor: Taylor Lipschitz}, by choosing $\X=\X_{st}$, as well as $z_1=Z_s$, $z_2=Z_t$ for $(s,t)\in[0,T]_\leq^2$ and defining $H_{st}:=H_{Z_sZ_t}$. Then we obtain for $m=0,\ldots,N-1$
		\begin{align*}
			&B^{[m+1]}_{st}(\tau)
			=\sum_{w\in\cF^{\leq m-1}}H_{z_1z_2}(w;\tau_0)\cdot z^{[m+1-|w|]} \, \X(w)+\sum_{k=1}^{m}\sum_{w=\tau_1\cdots\tau_k\in\cF^{= m}}\sum_{p\in\cP(I_k,\bt)}\left[\sum_{\rho\in\cT_b^{p(I_{k-1})}}\right.
			\\&\left.\sum_{\ell=1}^{m_L(\rho)+1}\!\!\!\!\sum_{\bj_\ell\in\{0,\ldots,k-1\}^\ell_<}\!\!\!\!J_{\rho,\ell}(\,\bj_\ell)\,(-1)^{k-m_R(\bullet_{p(k)}\to_\ell\rho)}\int_{[0,1]_\geq^k} \!\!\left(\Up_{G(\bullet_{p(k)}\to_\ell\rho)}(u_{j_1})-\Up_{G(\bullet_{p(k)}\to_\ell\rho)}(z_1)\right)\d\mathbf{t}_{k-1}\right.
			\\&\left.+\sum_{\rho\in\cT_b^{p(I_{k})}}\int_{[0,1]_\geq^{k+1}}(-1)^{k-m_R(\rho)}\,e(\mathcal{A}'(\rho))(\mathbf{t}_k)\cdot z^{[1]}\d\mathbf{t}_k\right]\X(w)\,.	
		\end{align*}
	For a fixed $m=0,\ldots,N-1$ and $\rho\in\cT_b$ we call
	\begin{align*}
		&H_1^m:= \sum_{w\in\cF^{\leq m-1}} H_{st}(w;\tau)  \cdot Z_{st}^{[m+1-|w|]}\, \X_{st}(w)\,,
		\\&H_2:=\sum_{\ell=1}^{m_L(\rho)+1}\sum_{\bj_\ell\in\{0,\ldots,k-1\}^\ell_<}J_{\rho,\ell}(\,\bj_\ell)  \\
		&\hspace{2cm}\int_{[0,1]_\geq^k}(-1)^{k-m_R(\bullet_{p(k)}\to_\ell\rho)}\,\, \left(\Up_{G(\bullet_{p(k)}\to_\ell\rho)}(u_{j_1})-\Up_{G(\bullet_{p(k)}\to_\ell\rho)}(Z_s)\right)\d\mathbf{t}_{k-1}\,,
	\end{align*}
		\begin{align*}
		&H_3^\rho:=\sum_{\ell=1}^{m_L(\rho)+1}\sum_{\bj_\ell\in\{0,\ldots,k-1\}^\ell_<}J_{\rho,\ell}(\,\bj_\ell)\,
			\\&\hspace{2cm}\int_{[0,1]_\geq^{k+1}}(-1)^{k+1-m_R(\rho)}\prod_{r=1}^{\ell-1}\nabla \Up_{G(\rho_r)}(u_{j_r})\, \nabla\Up_{G(P_{\rho,\ell})}(u_{j_\ell})\d\mathbf{t}_k\,Z^{[1]}_{st}\,,
		\end{align*}
		where $H_3^\rho$ takes this form on account of the definition of $e(\mathcal{A}'(\rho))$, see \eqref{eq:e(A)} and \eqref{Eq: A'}.
		
		We estimate each term separately. By \eqref{Eq: def H} and since $\Up_\tau$ is globally Lipschitz for all $\tau\in\cT^{\leq N-1}$,
		\begin{align*}
			\|H_1^m\|_{(m+\beta)\alpha,\mu}&\leq \sum_{w\in\cF^{\le m-1}}\|H(w,\tau)\|_\infty\|\X(w)\|_{|w|\alpha}\|Z^{[m+1-|w|]}\|_{(m+\beta-|w|)\alpha,\mu}
			\\&\lesssim_{\X,\sigma} \sum_{\ell=2}^{m+1}\|Z^{[\ell]}\|_{\ell\alpha,\mu}
		\end{align*} 
		and all the norms involved are finite.
		For what concerns $H_3^\rho$, since $G(\rho_1),\ldots,G(\rho_r),G(P_{\rho,\ell})\in\cT^{\le m-1}$, we proceed similarly and obtain
		\begin{align*}
			\|H_3^\rho\|_{\beta\alpha,\mu}&\lesssim\sum_{\ell=1}^{m_L(\rho)+1} \prod_{r=1}^{\ell-1}\|\nabla \Up_{G(\rho_r)}\|_\infty \|\nabla\Up_{G(P_{\rho,\ell})}\|_\infty (\mu\wedge T)^{(1-\beta)\alpha}\|Z^{[1]}\|_{\alpha,\mu}\,.
		\end{align*} 
		The application of \eqref{eta1} to adjust the H\"older exponent is left understood. 
		To bound $H_2^\rho$ we recall that $\Up_\tau\in\mathcal{C}^\beta$ for all $\tau\in\cT^{=m}$, which entails
		\begin{align*}
			|\Up_{G(\bullet_{p(k)}\to_\ell\rho)}(u_{j_1})-\Up_{G(\bullet_{p(k)}\to_\ell\rho)}(Z_s)|\leq[\Up_{G(\bullet_{p(k)}\to_\ell\rho)}]_\beta |t_{j_1} Z^{[1]}_{st}|^\beta.
		\end{align*}
		Then, recalling \eqref{Eq: weighted norm power}, we obtain 
		\begin{align*}
			\|H_2^\rho\|_{\beta\alpha,\mu}\lesssim\sum_{\ell=1}^{m_L(\rho)+1}\sum_{\,\bj_\ell\in\{0,\ldots,k-1\}^\ell_<}J_{\rho,\ell}(\,\bj_\ell)
			\,[\Up_{G(\bullet_{p(k)}\to_\ell\rho)}]_\beta\| Z^{[1]}\|_{\alpha,\mu}^\beta\,.
		\end{align*} 
		Returning to the full expression of $B^{[m+1]}_{st}(\tau)$ we derive the bound
		\begin{align*}
			\|B^{[m+1]}(\tau)&\|_{(m+\beta)\alpha,\mu}\leq \|H_1^m\|_{(m+\beta)\alpha,\mu}
			\\&\quad+\sum_{k=1}^{m}\sum_{w=\tau_1\cdots\tau_k\in\cF^{= m}}\sum_{p\in\cP(I_k,\bt)}\sum_{\rho\in\cT_b^{p(I_{k-1})}}\|H_2^\rho\|_{\beta\alpha,\mu}\|\X(w)\|_{|w|\alpha}
			\\&\quad+\sum_{k=1}^{m}\sum_{w=\tau_1\cdots\tau_k\in\cF^{= m}}\sum_{p\in\cP(I_k,\bt)}\sum_{\rho\in\cT_b^{p(I_{k})}}\|H_3^\rho\|_{\beta\alpha,\mu}\|\X(w)\|_{|w|\alpha}
			\\&\qquad\lesssim_{\X,\sigma} \sum_{\ell=2}^{m+1}\|Z^{[\ell]}\|_{\ell\alpha,\mu}+\| Z^{[1]}\|_{\alpha,\mu}^\beta+(\mu\wedge T)^{1-\beta}\| Z^{[1]}\|_{\alpha,\mu}.
		\end{align*}
		Observe that when $\| Z^{[1]}\|_{\alpha,\mu}\leq 1$ then $\| Z^{[1]}\|_{\alpha,\mu}^\beta\leq 1$, while if $\| Z^{[1]}\|_{\alpha,\mu}>1$, then $\| Z^{[1]}\|^\beta_{\alpha,\mu}\leq\| Z^{[1]}\|_{\alpha,\mu}$ (and $<$ if $\beta<1$). Thus we conclude that
		\begin{align*}
			&\|B^{[m+1]}(\tau)\|_{(m+\beta)\alpha,\mu}\lesssim_{\X,\sigma} \sum_{\ell=2}^{m+1}\|Z^{[\ell]}\|_{\ell\alpha,\mu}+(\mu\wedge T)^{1-\beta}(1\vee\| Z^{[1]}\|_{\alpha,\mu}\,).
		\end{align*}
		The proof is complete.
	\end{proof}
	Now we can prove Theorem \ref{th: a priori branched}, namely our desired a priori estimates of the solution remainders. 
	\begin{proof}[ of Theorem \ref{th: a priori branched}]
		Since $(N+\beta) \alpha > 1$ and $Z_{s t}^{[N+1]} = o(t-s)$ we can apply the weighted Sewing bound \eqref{weighted Sewing bound} to control the $(N+1)$-st remainder with its increment as $\| Z^{[N+1]} \|_{(N+\beta)\alpha, \mu} \leq K_{(N+\beta) \alpha} \| \delta Z^{[N+1]} \|_{(N+\beta) \alpha,\mu}$. By \eqref{dZN+1} it follows that
		\begin{equation}\label{Eq: bound deltaZ_N+1 branched}
			\|\delta Z^{[N+1]}\|_{(N+\beta)\alpha,\mu}\leq\sum_{\tau\in\cT^{\leq N}} \|B^{[N-|\tau|+1]}(\tau)\|_{(N-|\tau|+\beta)\alpha,\mu} \|\mathbb{X}(\tau)\|_{|\tau|\alpha}.
		\end{equation}
		For the sake of a lighter notation we define $\varepsilon :=(\mu \wedge T)^{\alpha}$. By \eqref{eq:5.8.1} we have for $m=0,\ldots,N-|\tau|$
		\begin{align*}
			\|B^{[m+1]}(\tau)\|_{(m+\beta)\alpha,\mu}&\lesssim_{\X,\sigma} \sum_{\ell=2}^{m+1}\|Z^{[\ell]}\|_{\ell\alpha,\mu}+\varepsilon^{1-\beta}(1\vee\| Z^{[1]}\|_{\alpha,\mu})\\
			&\lesssim_{\X,\sigma,T} (1\vee\| Z^{[1]}\|_{\alpha,\mu})+\sum_{\ell=2}^{m+1}\|Z^{[\ell]}\|_{\ell\alpha,\mu}\,.
		\end{align*}
		Plugging the last expression in \eqref{Eq: bound deltaZ_N+1 branched} we get
		\begin{align}\label{Eq: first bound deltaZ branched}
			\|\delta Z^{[N+1]}\|_{(N+\beta)\alpha,\mu}
			\lesssim_{\X,\sigma,T} (1\vee\| Z^{[1]}\|_{\alpha,\mu})+\sum_{\ell=2}^{m+1}\|Z^{[\ell]}\|_{\ell\alpha,\mu}\,.
		\end{align}
		As a second step we shall bound the right-hand side of \eqref{Eq: first bound deltaZ branched} in terms of the rough path components.  
		Observe that by definition we can write $Z_{s t}^{[i]} = \sum_{k=i}^N\sum_{\tau\in\cT^{=k}}\Up_{\tau}(Z_s)\,
		\mathbb{X}_{s t}(\tau) + Z^{[N+1]}_{s t}$ for all $i\in\{1,\ldots, N\}$. Thus
		\begin{equation}\label{eq:ZZZ}
			\| Z^{[i]} \|_{i \alpha, \mu}\leq \sum_{k=i}^N\sum_{\tau\in\cT^{=k}}\varepsilon^{k-i}\|\Up_\tau\|_{\infty,\mu}
			\|\mathbb{X}(\tau)\|_{|\tau|\alpha} + \varepsilon^{N+\beta-i} \| Z^{[N+1]} \|_{(N+\beta)
				\alpha, \mu} \, ,
		\end{equation}
		where we used \eqref{2factors} and Lemma \ref{Lem: bound weighted holder}.
		Resorting to the bound \eqref{supholdweighted}, the Lipschitz continuity of $\Up_\tau$ for $\tau\in\cT^{\le N-1}$, the H\"older continuity of $\Up_\tau$ for $\tau\in\CT^{=N}$ and \eqref{Eq: first bound deltaZ branched} we have
		\begin{align*}
			\| Z^{[i]} \|_{i \alpha, \mu}&\leq \sum_{k=i}^{N-1}\sum_{\tau\in\cT^{=k}}\varepsilon^{k-i}(|\Up_{\tau}(Z_0)|+3\varepsilon\|\nabla\Up_\tau\|_{\infty}\|\delta Z\|_{\alpha,\mu})
			\|\mathbb{X}(\tau)\|_{k\alpha}
			\\&\quad+\sum_{\tau\in\cT^{=N}}\varepsilon^{N-i}(|\Up_{\tau}(Z_0)|+3\varepsilon^{\beta}[\Up_\tau]_{\beta\alpha}\|\delta Z\|^\beta_{\alpha,\mu})
			\|\mathbb{X}(\tau)\|_{N\alpha}\\
			&\quad + \varepsilon^{(N+\beta-i)\alpha} 
			K_{(N+1) \alpha} c_{\X,\sigma,T} \sum_{\ell=1}^N\left[(1\vee\| Z^{[1]}\|_{\alpha,\mu})+\sum_{\ell=2}^{m+1}\|Z^{[\ell]}\|_{\ell\alpha,\mu}\right]\,,
		\end{align*}
		for $ c_{\X,\sigma,T}$ the proportionality constant in \eqref{Eq: first bound deltaZ branched}.
		Observe that similarly to Lemma \ref{lem: H estimate} we can bound $\|Z^{[1]}\|_{\alpha,\mu}^\beta\leq(1\vee\|Z^{[1]}\|_{\alpha,\mu})$. This allows to rearrange the terms as 
		\begin{align*}
			\| Z^{[i]} \|_{i \alpha, \mu}&\le  \sum_{k=i}^N\sum_{\tau\in\cT^{=k}}|\Up_{\tau}(Z_0)| \, \|\mathbb{X}(\tau)\|_{k\alpha}
			+c_{\X,\sigma,T} \varepsilon^{\beta\alpha}\left[(1\vee\| Z^{[1]}\|_{\alpha,\mu})+\sum_{\ell=2}^{m+1}\|Z^{[\ell]}\|_{\ell\alpha,\mu}\right]\,.
		\end{align*}
		Note now that for $A,B,C\geq 0$, $A\le B+C$ implies $(1\vee A)\le (1\vee B)+C$. Then for $i=1$
		\begin{align*}
			1\vee \| Z^{[1]} \|_{\alpha, \mu}&\le (1\vee (|\Up_{\bullet}(Z_0)| \, \|\mathbb{X}(\bullet)\|_{\alpha}))+ \sum_{k=2}^N\sum_{\tau\in\cT^{=k}}
			|\Up_{\tau}(Z_0)| \, \|\mathbb{X}(\tau)\|_{k\alpha}
			\\ & \quad +c_{\X,\sigma,T} \varepsilon^{\beta\alpha}\left[(1\vee\| Z^{[1]}\|_{\alpha,\mu})+\sum_{\ell=2}^{m+1}\|Z^{[\ell]}\|_{\ell\alpha,\mu}\right]\,.
		\end{align*}
		Summing over all $i\in\{1,\ldots,N\}$ we obtain
		\begin{align*}
			A & := (1\vee \| Z^{[1]} \|_{\alpha, \mu})+\sum_{i=2}^N \| Z^{[i]} \|_{i \alpha, \mu} 
			\\ & \le (1\vee (|\Up_{\bullet}(Z_0)| \, \|\mathbb{X}(\bullet)\|_{\alpha}))+ \sum_{k=2}^N\sum_{\tau\in\cT^{=k}} k \,|\Up_{\tau}(Z_0)| \, \|\mathbb{X}(\tau)\|_{k\alpha}
			+c_{\X,\sigma,T} \varepsilon^{\beta\alpha}A\,.
		\end{align*}
		By choosing $\varepsilon=\varepsilon(\sigma,\X)$ small enough so that
		$c_{\X,\sigma,T} \varepsilon^{\beta\alpha} \le \frac{1}{2}$,
		we can absorb the last term in the left hand side and obtain
		\begin{align*}
			(1\vee \| Z^{[1]} \|_{\alpha, \mu})+\!\!\sum_{i=2}^N \| Z^{[i]} \|_{i \alpha, \mu} &\le
			2\left[ (1\vee (|\Up_{\bullet}(Z_0)| \, \|\mathbb{X}(\bullet)\|_{\alpha}))+ \!\!\sum_{\tau\in\cT^{\le N}} |\tau| \,|\Up_{\tau}(Z_0)| \, \|\mathbb{X}(\tau)\|_{|\tau|\alpha}\right]
		\end{align*}
		namely \eqref{eq: a priori branched}. The proof of \eqref{eq: a priori branched2} is similar and based on \eqref{eq:5.8.2} rather than
		on \eqref{eq:5.8.1}.
	\end{proof}
	
	We state and prove now analogous a priori estimates under the assumption of boundedness of higher order derivatives of $\Upsilon_\tau$, using our
	more elementary formula for the generalised Taylor remainder \eqref{Eq: Taylor}.
	
	\begin{proposition}\label{lem:bounded}
		Assume that $\nabla^{k}\Up_\tau$ is bounded for all $\tau\in\cT^{\le N}$ and $k+|\tau|\le N+1$. Then, for any solution $Z$ of \eqref{eq: sol Davie branched},
		\begin{equation}\label{Eq: a priori bounded}
			\|Z^{[N+1]}\|_{(N+1)\alpha,\mu} \lesssim_{\X,\sigma,\varepsilon} \sum_{i=1}^N \| Z^{[i]} \|_{i \alpha, \mu} \,.
		\end{equation}
		In addition, for either $T$ or $\mu$ sufficiently small,
		\begin{equation}\label{Eq: a priori bounded 2}
			 \sum_{i=1}^N \| Z^{[i]} \|_{i \alpha, \mu} \lesssim_{\X,\sigma}1\,.
		\end{equation}
	\end{proposition}
	\begin{proof}
		We assume $Z^{[N+1]}_{st}=o(t-s)$. Then for $\ell=1,\ldots,N$
		\[
		\|Z^{[\ell]}\|_{\ell\alpha}\le \|Z^{[N+1]}\|_{(\ell-1+\beta)\alpha}+\sum_{i=\ell+1}^N\sum_{\tau\in\cT^{=i}} \|\Upsilon_\tau(Z)\|_\infty \, \|\X(\tau)\|_{\ell\alpha}<+\infty\,.
		\]
		Resorting to the decomposition of $\delta Z^{[N+1]}$ derived in Proposition \ref{Lem: delta Z}, it holds
		\[
		\|\delta Z^{[N+1]}\|_{(N+1)\alpha,\mu}\leq\sum_{\tau\in\cT^{\leq N}} \|B^{[N-|\tau|+1]}(\tau)\|_{(N-|\tau|+1)\alpha,\mu} \|\mathbb{X}(\tau)\|_{|\tau|\alpha}\,.
		\]
		We recall the explicit expression in \eqref{eq:boubou}:
		\[
		B^{[m+1]}_{st}(\tau)
		=\sum_{k=0}^{m}\sum_{w=\tau_1\cdots\tau_k\in\cF^{\le m}}\int_{[0,1]}\nabla^{k+1}\Up_\tau(Z_s+rZ^{[1]}_{st})\cdot  Z^{[m+1-|w|]}_{st}\,\frac{[\Upsilon\X]_{st}^{w}}{\pi(w)}\, (1-r)^{k}\d r\,,
		\]	
		where $[\Upsilon\X]_{st}^{w}:=\bigotimes_{i=1}^{k}\Up_{\tau_i}(Z_s) \, \X_{st}(\tau_i)$. We fix
		\begin{equation*}
			M:=\max_{\substack{\tau\in\cT^{\leq N}\\k+|\tau|\le N+1}}\|\nabla^k\Up_{\tau}\|_{\infty}\,,
		\end{equation*}
		which is finite on account of the boundedness assumptions on derivatives of the elementary differential components. Then 
		\[
		\|B^{[m+1]}(\tau)\|_{(m+1)\alpha,\mu}\le  \sum_{k=0}^{m}\sum_{w=\tau_1\cdots\tau_k\in\cF^{\le m}}M^{k+1} \,
		\|Z^{[m+1-|w|]}\|_{(m+1-|w|)\alpha,\mu} \, \|\X(w)\|_{|w|\alpha}\,.
		\]
		As a result 
		\[
		\begin{split}
			&\|\delta Z^{[N+1]}\|_{(N+1)\alpha,\mu}
			\\ & \le \sum_{\tau\in\cT^{\leq N}} 
			\sum_{k=0}^{N-|\tau|}\sum_{w=\tau_1\cdots\tau_k\in\cF^{\le N-|\tau|}}M^{k+1} \,
			\|Z^{[N-|\tau|+1-|w|]}\|_{(N-|\tau|+1-|w|)\alpha,\mu} \, \|\X(w)\|_{|w|\alpha} \, \|\mathbb{X}(\tau)\|_{|\tau|\alpha}
			\\ & 	\lesssim_{\X,M} \sum_{\ell=1}^{N} 	\|Z^{[\ell]}\|_{\ell\alpha,\mu} 
		\end{split}
		\]
		and by the Sewing bound
		\[
		\|Z^{[N+1]}\|_{(N+1)\alpha,\mu}\lesssim_{\X,M} \sum_{\ell=1}^{N} \|Z^{[\ell]}\|_{\ell\alpha,\mu} \,.
		\]
		Now we go back to \eqref{eq:ZZZ}. Since $\|\Upsilon_\tau\|_\infty\le M$, then 
		\begin{equation}\label{Eq: control Z^i}
		\| Z^{[i]} \|_{i \alpha, \mu}\leq \sum_{k=i}^N\sum_{\tau\in\cT^{=k}}\varepsilon^{k-i}M
		\|\mathbb{X}(\tau)\|_{|\tau|\alpha} + \varepsilon^{N+1-i} \| Z^{[N+1]} \|_{(N+1)
			\alpha, \mu} \, ,
		\end{equation}
		and therefore 
		\[
		\|Z^{[N+1]}\|_{(N+1)\alpha,\mu}\lesssim_{M,\X}  \sum_{\ell=1}^{N} \|Z^{[\ell]}\|_{\ell\alpha,\mu}
		\le C_{M,\X}+ \sum_{\ell=1}^N\varepsilon^{N+1-\ell} \|Z^{[N+1]}\|_{(N+1)\alpha}\,,
		\]
		which corresponds to \eqref{Eq: a priori bounded}.
		As a result, for sufficiently small $\varepsilon=\varepsilon(\sigma,\X)$, the latter inequality implies
		\begin{align*}
			\|Z^{[N+1]}\|_{(N+1)\alpha,\mu}\lesssim_{\X,M} 1\,.
		\end{align*}
		Plugging this estimate in \eqref{Eq: control Z^i} and summing over $i=1,\ldots, N$ yields \eqref{Eq: a priori bounded 2}.
	\end{proof}

	\appendix
	
	\section{On grafting and smooth rough paths}\label{App:A}
	
	\subsection{Properties of grafting}\label{App: grafting}
	In this appendix we discuss two possible natural notions of grafting of trees, see \cite{manchon2011lois}. First we give a
	graphical interpretation of $\graf$ from Definition \ref{Def: grafting}.
	\begin{lemma}
			For $\tau\in\cT$ and $w\in\cF\setminus\{\one\}$ we have
		\begin{equation}\label{eq: grafting n}
			\tau\curvearrowright w:=\sum_{v\in\cF}n(\tau,w;v)\,v\,,
		\end{equation}
		where $n(\tau,w;v)$ denotes the number of ways of linking the root of $\tau$ to any node of $w$ via a new edge to obtain $v$. 
	\end{lemma}

	\begin{proof}
	We proceed by induction on the number of vertices of $w$ in \eqref{eq: grafting n}. First it is easy to see that
	\[
	n(\tau,\bullet_i;v) = \un{(v=[\bullet]_i)}\,,
	\]
	which proves the base case $w\in\cF^{=1}$.
		Fix now $n\ge 2$ and suppose that the lemma holds true for all $v\in\cF^{\le n-1}$. Then for $[\tau_1\cdots\tau_k]_i\in\cT^{=n}$,
		 by \eqref{Eq: grafting iterative} and by the recurrence assumption,
		\begin{align*}
			\tau\graf[\tau_1&\cdots\tau_k]_i=[\tau \tau_1\cdots\tau_k+\tau\graf (\tau_1\cdots\tau_k)]_i\\
			&=[\tau \tau_1\cdots\tau_k]_i+\sum_{i=1}^k[\tau_1\cdots(\tau\graf\tau_i)\cdots\tau_k]_i
			\\&=[\tau \tau_1\cdots\tau_k]_i+\sum_{i=1}^k\left[\tau_1\cdots\left(\sum_{\tau''\in\cT}n(\tau,\tau_i;\tau'')\,\tau''\right)\cdots\tau_k\right]_i
			\\&=[\tau \tau_1\cdots\tau_k]_i+\sum_{v\in\cF}n(\tau,\tau_1\cdots\tau_k;v)\,[v]_i\,.
		\end{align*}
		Observe that, for $w=\tau_1\cdots\tau_k$, the graphical interpretation of this choice of grafting implies 
		\[
		n(\tau,[\tau_1\cdots\tau_k]_i;[v]_i)=\un{(v=\tau \tau_1\cdots\tau_k)}+n(\tau,\tau_1\cdots\tau_k;v)\,.
		\]
		This yields \eqref{eq: grafting n} for $w\in\cT^{=n}$. The case $w\in\cF^{=n}$ follows by the Leibniz rule in $w$ for both sides of \eqref{eq: grafting n}.
	\end{proof}
	Another natural choice of grafting widely adopted in the literature is the one whose $\star$-product, see \eqref{Eq: star product}, is the dual of the Connes-Kreimer co-product \cite{hoffman2003combinatorics}.
	\begin{definition}
		For $\tau,\tau'\in\cT$ we set
		\begin{equation*}
			\tau\curvearrowright_{\rm CK}\tau':=\sum_{\tau''\in\cT}m(\tau,\tau';\tau'')\,\tau''\,,
		\end{equation*}
		where $m(\tau,\tau';\tau'')$ denotes the number of edges of $\tau''$ that, once removed, produce $\tau'$ as the subtree that contains the root of $\tau''$ and $\tau$ as the remaining part.
	\end{definition}

	\begin{remark}\label{Rem: some remarks on grafting}
		As pointed out in \cite{manchon2011lois}, the pre-Lie algebras $(\langle\cT\rangle,\curvearrowright_{\rm CK})$ and $(\langle\cT\rangle,\curvearrowright)$ are isomorphic. 
		They differ only by the combinatorial coefficients, for example
		\[
		\<0>\graf\<2>=\<3>+2\,\<201>\,,
		\]
		\[
		\<0>\graf_{\rm CK}\<2>=3\,\<3>+\<201>\,.
		\]
		More specifically, the linear map $\varphi:\langle \cT\rangle\to\langle\cT\rangle$ such that $\varphi(\tau):=\frac{\tau}{s(\tau)}$, see \eqref{Eq: symmetry factor} for the definition of the symmetry factor $s$ of a rooted tree, satisfies
		\[
		\varphi(\tau\curvearrowright_{\rm CK}\rho)=\varphi(\tau)\curvearrowright\varphi(\sigma)\,.
		\]
	\end{remark}

	We conclude this section with the proof of Lemma \ref{lem: morphism upsilon}, which is a standard result in the literature on branched rough paths, see \emph{e.g.} \cite[Lemma 3.7]{bonnefoi2022priori}.
	
	\begin{proof}[of Lemma \ref{lem: morphism upsilon}]
		We perform a twofold induction. First we prove \eqref{eq: morphism upsilon} for $k=1$, by induction on the degree of $\tau$, i.e.
		\[
		\Up_{\tau_1\graf\tau}=\nabla\Up_\tau \cdot \Up_{\tau_1}\,.
		\] 
		The base case $\tau=\bullet$ follows directly from Definition \ref{Def: elementary differential trees}. Suppose now that the relation is true for all trees of degree lower than $\tau=[\tau'_1,\ldots,\tau'_m]$. On account of the recursive definition of grafting and by the inductive hypothesis, we have
		\begin{align*}
			\Up_{\tau_1\graf\tau}&=\Up_{\tau_1\graf[\tau'_1\cdots\tau'_m]}=\Up_{[\tau_1\tau'_1\cdots\tau'_m]}+\sum_{j=1}^m\Up_{[\tau'_1\cdots(\tau'_1\graf\tau'_j)\cdots\tau'_m]}\\
			&=\nabla^{m+1}\sigma\cdot\left[\Up_{\tau_1}\prod_{j=1}^m \Up_{\tau'_j}\right]+\sum_{j=1}^m \nabla^{m}\sigma\cdot[\Up_{\tau'_1}\cdots\Up_{\tau_1\graf\tau'_j}\cdots\Up_{\tau'_m}]\,.
		\end{align*}
		On the other hand, by the Leibniz rule
		\begin{align*}
			&\nabla\Up_{[\tau'_1\cdots\tau'_m]}\cdot \Up_{\tau_1}=\nabla\left(\nabla^{m}\sigma\cdot\left[\prod_{j=1}^m \Up_{\tau'_j}\right]\right)\cdot\Up_{\tau_1}
			\\&=\nabla^{m+1}\sigma\cdot	\left[\prod_{j=1}^m \Up_{\tau'_j}\Up_{\tau_1}\right]+\sum_{j=1}^m\nabla^{m}\sigma\cdot[\Up_{\tau'}\cdots\nabla\Up_{\tau'_j}\cdot\Up_{\tau_1}\cdots\Up_{\tau'_m}]
			\\&=\nabla^{m+1}\sigma\cdot	\left[\prod_{j=1}^m \Up_{\tau'_j}\Up_{\tau_1}\right]+\sum_{j=1}^m\nabla^{m}\sigma\cdot[\Up_{\tau'_1}\cdots\Up_{\tau_1\graf\tau'_j}\cdots\Up_{\tau'_m}]\,.
		\end{align*}
		In the last line we resorted to the inductive hypothesis. A comparison between these two expressions closes this first induction argument. As a second step we perform an induction on the number of trees in the forest $\tau_1\cdots\tau_k$. The base case has already been proved. We assume that the claim holds true for $k-1$ trees. By the Guin-Oudom extension of grafting and the linearity of $\Upsilon$:
		\begin{align*}
			\Up_{(\tau_1\cdots\tau_k)\graf\tau'}&=\Up_{\tau_1\graf((\tau_2\cdots\tau_k)\graf\tau')}-\sum_{j=2}^k\Up_{(\tau_2\cdots(\tau_1\graf\tau_j)\cdots\tau_k)\graf\tau'}\,.
		\end{align*} 
		By the inductive hypothesis and the Leibniz rule the last expression reads
		\begin{align*}
			&\nabla\Up_{(\tau_2\cdots\tau_k)\graf\tau'}\cdot\Up_{\tau_1}-\sum_{j=2}^k\nabla^k\Up_{\tau'}\cdot[\Up_{\tau_2}\cdots\Up_{\tau_1\graf\tau_j}\cdots\Up_{\tau_k}]
			\\&=\nabla\left(\nabla^{k-1}\Up_{\tau'}\cdot\left[\prod_{j=2}^{k-1}\Up_{\tau_j}\right]\right)\cdot\Up_{\tau_1}-\sum_{j=2}^k\nabla^k\Up_{\tau'}\cdot[\Up_{\tau_2}\cdots\Up_{\tau_1\graf\tau_j}\cdots\Up_{\tau_k}]\\
			&=\nabla^{k}\Up_{\tau'}\cdot\left[\prod_{j=1}^k\Up_{\tau_j}\right]+\sum_{j=2}^k\nabla^{k-1}\Up_{\tau'}\cdot[\Up_{\tau_2}\cdots\nabla\Up_{\tau_j}\cdot\Up_{\tau_1}\cdots\Up_{\tau_k}]
			\\&\quad-\sum_{j=2}^k\nabla^k\Up_{\tau'}\cdot[\Up_{\tau_2}\cdots\nabla\Up_{\tau_j}\cdot\Up_{\tau_1}\cdots\Up_{\tau_k}]
			=\nabla^{k}\Up_{\tau'}\cdot\left[\prod_{j=1}^k\Up_{\tau_j}\right].
		\end{align*}
		This concludes the proof. 
	\end{proof}
	
	\subsection{Smooth rough paths}\label{App: smooth rough paths}
	
	Let $X:[0,T]\to\R^d$ be a path of class $C^1$. In this more regular setting, integration against $\dot X$ poses no analytical hurdles. Thus there is a natural way of lifting $X$ to the collection of iterated integrals labelled by rooted trees, the so called \emph{smooth branched rough path} $\X:[0,T]_\leq^2\times \cF\to\R$. We define it iteratively as follows: 
	\begin{enumerate}
		\item $\X_{st}(\one)=1$.
		\item For all $w\in\cF$:
		\begin{equation}\label{Eq: definition smooth rough}
			\X_{st}(w\curvearrowright\bullet_i):=\int_s^t\frac{\X_{sr}(w)}{\pi(w)}\,\dot{X}_r^i\d r\,.
		\end{equation}
		\item For all $\tau_1,\ldots,\tau_k\in\cT$:
		\[
		\X_{st}(\tau_1\cdots\tau_k) = \X_{st}(\tau_1)\cdots\X_{st}(\tau_k)\,.
		\]
	\end{enumerate}
	We can rewrite \eqref{Eq: definition smooth rough} as
	\begin{align}\label{Eq: def Linares}
		\X_{st}(\tau)=\varepsilon(\tau)+\sum_{u\in\cF}\sum_{i=1}^d \int_s^t\frac{\X_{sr}(u)}{\pi(u)}\langle u\curvearrowright\bullet_i,\tau\rangle\dot{X}_r^i\d r\,.
	\end{align}
	where $\langle\cdot,\cdot\rangle:\cF\times\cF\to\R$ is the duality pairing introduced in \eqref{Eq: pairing}, while $\varepsilon$ is the dual af the empty forest in $\mathcal{H}^\ast$. 
	\begin{lemma}
		Let $\X$ be defined as per \eqref{Eq: def Linares}. Then the Chen relation holds:
		\begin{equation}\label{Eq: Chen}
			\X_{st}=\X_{su}\star\X_{ut}\,,
		\end{equation}
		where $\star$ is defined in \eqref{Eq: star product}.
	\end{lemma}
	\begin{proof}
		As a first step we observe that, for $\tau\in\cT$ we have by \eqref{Eq: dual identification}
		\begin{align*}
			(\X_{su}\star\X_{ut})(\tau)&=\sum_{w_1,w_2\in\cF}\frac{\X_{su}(w_1)}{\pi(w_1)}\frac{\X_{ut}(w_2)}{\pi(w_2)}\langle w_1\star w_2,\tau\rangle\\
			&=\X_{su}(\tau)+\sum_{w_1\in\cF, \rho\in\cT}\frac{\X_{su}(w_1)}{\pi(w_1)} \, \X_{ut}(\rho) \, \langle w_1\curvearrowright \rho,\tau\rangle,
		\end{align*}
		so that proving \eqref{Eq: Chen} is equivalent to proving
		\[
		\X_{st}(\tau)=\X_{su}(\tau)+\sum_{w_1\in\cF, \rho\in\cT}\frac{\X_{su}(w_1)}{\pi(w_1)} \, \X_{ut}(\rho) \, \langle w_1\curvearrowright \rho,\tau\rangle.
		\]
	We proceed by induction on the number of vertices of $\tau$. For $\tau=\bullet_i$, $i=1,\ldots, d$, the result immediately follows. Let us suppose that \eqref{Eq: Chen} holds true for all trees in $\cT^{\le n-1}$.
	Then we can extend it to forests in $\cF^{\le n}$ since 
		\begin{align*}
			\X_{st}(\tau_1\tau_2)&=\X_{st}(\tau_1) \, \X_{st}(\tau_2)=(\X_{su}\star \X_{ut})(\tau_1) \, (\X_{su}\star \X_{ut})(\tau_2)
						=(\X_{su}\star \X_{ut})(\tau_1\tau_2)\,,
		\end{align*}	
		for any $\tau_1,\tau_2\in\cT^{\le n-1}$. Fixing $\tau\in\cT^{=n}$, by \eqref{Eq: def Linares} we have
		\begin{align*}
			\X_{st}(\tau)&=\varepsilon(\tau)+\sum_{w\in\cF}\int_s^t\frac{\X_{su}(w)}{\pi(w)}\langle w\curvearrowright\bullet,\tau\rangle\dot{X}_r\d r
			\\ & =\X_{su}(\tau)+\sum_{w\in\cF}\int_u^t\frac{\X_{su}(w)}{\pi(w)}\langle w\curvearrowright\bullet,\tau\rangle\dot{X}_r\d r\\
			&=\X_{su}(\tau)+\sum_{w_1,w_2\in\cF}\int_u^t\frac{\X_{su}(w_1)}{\pi(w_1)}\frac{\X_{ur}(w_2)}{\pi(w_2)}\langle (w_1\star w_2)\curvearrowright\bullet,w\rangle\dot{X}_r\d r\,,
		\end{align*}
		where in the last line we used the induction hypothesis.
		The Guin-Oudom extension of $\curvearrowright$ satisfies the crucial property (see \cite[Proposition 3.9(v)]{oudom2008lie})
		\[
		(w_1\star w_2)\curvearrowright w_3=w_1\curvearrowright(w_2\curvearrowright w_3)\,,\qquad w_1,w_2,w_3\in\cF\,.
		\]
		Then we obtain
		\begin{align*}
			\X_{st}(\tau)&=\X_{su}(\tau)+\sum_{w_1,w_2\in\cF}\int_u^t\frac{\X_{su}(w_1)}{\pi(w_1)}\frac{\X_{ur}(w_2)}{\pi(w_2)}\langle (w_1\star w_2)\curvearrowright\bullet,\tau\rangle\dot{X}_r\d r\\
			&=\X_{su}(\tau)+\sum_{w_1,w_2\in\cF}\int_u^t\frac{\X_{su}(w_1)}{\pi(w_1)}\frac{\X_{ur}(w_2)}{\pi(w_2)}\langle w_1\curvearrowright(w_2\curvearrowright\bullet),\tau\rangle\dot{X}_r\d r\\
			&=\X_{su}(\tau)+\sum_{w_1,w_2\in\cF}\frac{\X_{su}(w_1)}{\pi(w_1)}\langle w_1\curvearrowright\int_u^t\frac{\X_{ur}(w_2)}{\pi(w_2)}(w_2\curvearrowright\bullet)\,\dot{X}_r\d r,\tau\rangle\\
			&=\X_{su}(\tau)+\sum_{w_1\in\cF,\rho\in\cT}\frac{\X_{su}(w_1)}{\pi(w_1)} \, \X_{ut}(\rho) \, \langle w_1\curvearrowright \rho,\tau\rangle,
		\end{align*}
		where in the last line we used Definition \ref{Eq: def Linares} and the fact that $\one\notin\cT$. 
	\end{proof}
	
	\begin{remark}
		Note that the standard definition of a branched rough path \cite{gubinelli2010ramification,hairer2015geometric,manchonrough} is based on 
		$\curvearrowright_{\rm CK}$
		rather than on $\curvearrowright$. For Chen's relation to hold, the recursive definition of a smooth rough path \eqref{Eq: definition smooth rough} must be modified into
		\begin{equation*}
			\X_{st}^{\rm CK}(w\curvearrowright_{\rm CK}\bullet_i):=\int_s^t \X_{sr}^{\rm CK}(w) \,\dot{X}_r^i\d r\,.
		\end{equation*}
		The two definitions differ only by a combinatorial coefficient:
		\begin{equation*}
			\X_{st}(w):=\frac{\pi(w)}{s(w)} \, \X_{st}^{\rm CK}(w)\,.
		\end{equation*}
		Our preference for rough paths satisfying Chen's relation with respect to $\graf$ is motivated by the fact that the elementary differential of Definition \ref{Def: elementary differential trees} is a morphism from the pre-Lie algebra $(\langle\cT\rangle,\graf)$ to the pre-Lie algebra of vector fields generated by rooted trees, where the pre-Lie product is given by differentiation: see Lemma \ref{lem: morphism upsilon}. Choosing $\graf$ instead of $\graf_{\rm CK}$ ultimately leads
		to simpler formulae with fewer combinatorial coefficients.
	\end{remark}
	
	We now go back to the solution of the controlled ODE \eqref{Eq: controlled ODE}-\eqref{integral form} and show why it motivates 
	the Ansatz for the notion of solution \emph{à la} Davie presented in Definition \ref{def:aladavie}.
	
	\begin{lemma}
		If $X$ is of class $C^1$ and $\sigma$ is smooth, any solution $Z$ to \eqref{Eq: controlled ODE}-\eqref{integral form} satisfies for all $m\geq 0$
		\begin{equation}\label{eq:zt-zs}
			Z_{st}^{[m+1]} = \sum_{i=1}^d\int_s^t \left(
			\Up_{\bullet_i}(Z_u) - \sum_{w\in\CF^{\le m-1}} \frac{\Up_{w\curvearrowright\bullet_i}(Z_s)}{\pi(w)}  \, \mathbb{X}_{s u}(w)
			\right) \, \dot{X}_u^i \d u\,,
		\end{equation}
		where $\X$ is the smooth rough path \eqref{Eq: definition smooth rough}.
	\end{lemma}
	\begin{proof}
		For $m=0$ the desired formula amounts to $Z_t-Z_s = \int_s^t \sigma(Z_u)\, \dot{X}_u \d u$, which follows from \eqref{integral form}
		if we recall that $\Up_\bullet=\sigma$. Now, suppose
		we have proved the claim for $m\ge 0$. Then we have
		\[
		\begin{split} 
			Z_{st}^{[m+1]} -Z_{st}^{[m+2]} &=
			\sum_{\tau\in\cT^{=m+1}} \Upsilon_\tau(Z_s)\, \X_{st}(\tau) 
			= \sum_{i=1}^d\sum_{w\in\CF^{=m}} \Upsilon_{w\curvearrowright\bullet_i}(Z_s) \, \mathbb{X}_{s t}(w\curvearrowright\bullet_i)
			\\ &  = \sum_{i=1}^d\int_s^t \sum_{w\in\CF^{=m}} \frac{\Upsilon_{w\curvearrowright\bullet_i}(Z_s)}{\pi(w)} \, \mathbb{X}_{s u}(w)
			\, \dot{X}_u^i \d u\,.
		\end{split}
		\]
		By swapping the terms and exploiting the induction hypothesis we conclude. 
	\end{proof}

	\begin{proposition}\label{pr:daviesmooth}
		Suppose that $X\in C^1$ and that $\sigma$ is smooth. Then any solution $Z$ to \eqref{integral form} satisfies
		\begin{equation}\label{eq: Davie smooth2}
			Z_{st}^{[m+1]} = \delta Z_{st}-\sum_{\tau\in\cT^{\leq m}}\Up_\tau(Z_s) \, \X_{st}(\tau)=O\left((t-s)^{m+1}\right)
		\end{equation}
		uniformly for $0\le s\le t\le T$, for any $m\geq 0$, where $\X$ is the smooth rough path \eqref{Eq: definition smooth rough}.
	\end{proposition}
	\begin{proof}
		If $m=0$, then \eqref{eq: Davie smooth2} follows from \eqref{integral form} since
		\[
		Z_t - Z_s = \int_s^t \sigma (Z_u) \, \dot{X_u} \d u = O(t-s)\,.
		\]
		Combining \eqref{eq:zt-zs} and \eqref{Eq: Taylor} we have, for any $m\geq 1$,
		\begin{align*}
			&	Z_{st}^{[m+1]} =\sum_{i=1}^d \int_s^t \left(
			\Up_{\bullet_i}(Z_u) - \sum_{w\in\CF^{\le m-1}} \frac{\Up_{w\curvearrowright\bullet_i}(Z_s)}{\pi(w)}  \, \mathbb{X}_{s u}(w)
			\right) \, \dot{X}_u^i \d u
			\\ & =
			\sum_{k=0}^{m-1}\sum_{\substack{w=\tau_1\cdots \tau_k\in\cF,\,\ell\ge1\\|w|+\ell=m}}\int_s^t\int_{[0,1]}\nabla^{k+1}\Up_\bullet(Z_s+rZ^{[1]}_{su})\cdot \left[Z^{[\ell]}_{su}\,\frac{[\Upsilon\X_{su}]^w}{\pi(w)}\right](1-r)^{k}\d r \cdot \dot{X}_u\d u.
		\end{align*}
		For $m\geq 1$ we obtain
		\[
		|Z_{st}^{[m+1]}| \lesssim \sum_{\ell=1}^{m}\int_s^t |Z^{[\ell]}_{su}| \, |u-s|^{m-\ell} \d u\,.
		\]
		Since $|Z_{st}^{[1]}|\lesssim |t-s|$, by recurrence on $m\ge 0$ we find that $|Z_{st}^{[m+1]}| \lesssim |t-s|^{m+1}$.
	\end{proof}
	
	\begin{remark}\label{rem:davie}
		The Ansatz for the notion of solution \emph{à la} Davie of the rough equation
		\eqref{Eq: controlled ODE}-\eqref{integral form}, presented in Definition \ref{def:aladavie}, is based on the expansion \eqref{eq: Davie smooth2}
		in the smooth case. If $X$ is $\alpha$-Hölder rather than of classe $C^1$, then the remainders $Z^{[m+1]}$ are expected to be of order $O(|t-s|^{(m+1)\alpha})$,
		for $m=0,\ldots,N-1$.
	\end{remark}

	\section{Analytic and combinatorial tools}\label{App: analytic tools}
	In this appendix we collect useful technical lemma that are used throughout the paper.
	\subsection{Integrals and combinatorics}
	We need a technical Lemma that generalises the exchange of integrals we encounter in the proof of Lemma \ref{lem: Upsilon descent}.
	\begin{lemma}\label{Lem: exchange integrals}
		Consider $k\in\N$, $0\le i < j \le k$ and an integrable function $f:[0,1]\to\R$. Setting $t_k:=0$ it holds
		\begin{align}
			&\int_{[0,1]_\geq^{k}}\int_{t_j}^{t_i} f(r)\d r\d \mathbf{t}_{k-1}=\sum_{\ell=i+1}^{j}\int_{[0,1]_\geq^{k+1}}f(t_{\ell})\d \mathbf{t}_{k}\, .
			\label{Eq: int exchange 1}
		\end{align}
	\end{lemma}
	\begin{proof}
		We decompose  the indicator function of the integration domain as
		\[
		\un{[t_j,t_i[}(r) = \sum_{\ell=i+1}^{j} \un{[t_\ell,t_{\ell-1}[}(r)\,.
		\]
		When $r\in[t_\ell,t_{\ell-1}[$, we write $(t_0,\ldots,t_{\ell-1},r,t_\ell,\ldots,t_{k-1})=\mathbf{t}_k'$. Then a simple change of variable yields \eqref{Eq: int exchange 1}.
	\end{proof}
	
	We prove a combinatorial identity on sets of permutations.
	\begin{lemma}\label{lem: transposition}
		Consider $I_k:=\{1,\ldots,k\}$, $I^i_{k}:=\{1,\ldots,i-1,i+1,\ldots,k\}$ and $\tau_1,\ldots,\tau_k\in\cT$. Setting $w=\tau_1\cdots\tau_k$, the following identity holds true:
		\begin{align} 
			\sum_{i=1}^k\frac{1}{\pi(\tau_i; w)}\sum_{q\in\cP(I^i_k,\bt^i)}\tau_i\otimes\tau_{q(1)}\otimes\ldots\otimes\tau_{q(i-1)}\otimes&\tau_{q(i+1)}\otimes\ldots\otimes\tau_{q(k)}\nonumber
			\\&=\sum_{p\in\cP(I_k,\bt)}\tau_{p(1)}\otimes\ldots\otimes\tau_{p(k)}\label{eq: transposition}\,,
		\end{align}
		where $\cP(\cdot)$ is as in Definition \ref{Def: sym H} and $\pi(\tau_i;w)$ is the number of trees equal to $\tau_i$ in the forest $w$ i.e. for $w=\hat\tau^{\ell_1}_1\cdots\hat\tau_r^{\ell_r}$ such that $\hat\tau_i\ne\hat\tau_j$ for $i\ne j$, we have $\pi(\tau_i;w)=\ell_i$.
	\end{lemma}
	\begin{proof}
		We fix a permutation $p\in \cP(I_k,\bt)$ and we represent by $w=\tau_1^{\ell_1}\cdots\tau^{\ell_r}_r$, for some $r\leq k$, the forest induced by $\tau_1,\ldots, \tau_k$. Then
		\begin{align*}
			\sum_{p\in\cP(I_k,\bt)}\tau_{p(1)}\otimes\ldots&\otimes\tau_{p(k)}=\sum_{p\in\cP(I_k,\bt)}\sum_{j=1}^r\un{(\tau_{p(1)}=\hat\tau_j)}\hat\tau_j\otimes(\tau_{p(2)}\otimes\ldots\otimes\tau_{p(k)})
			\\&=\sum_{p\in\cP(I_k,\bt)}\sum_{j=1}^r\frac{1}{\pi(\hat\tau_j;w)}\un{(\tau_{p(1)}=\hat\tau_j)}\sum_{a=1}^k\un{(\tau_i=\hat\tau_j)}\tau_a\otimes(\tau_{p(2)}\otimes\ldots\otimes\tau_{p(k)})\,,
		\end{align*}
		where in the last line the factor $\frac{1}{\pi(\hat\tau_j;w)}$ was introduced to compensate the repetitions of $\hat\tau_j$ in $w$. Then, exchanging the sums over $j$ and $a$ leads to
		\begin{align*}
			\sum_{p\in\cP(I_k,\bt)}\sum_{a=1}^k\frac{1}{\pi(\tau_a;w)}&\un{(\tau_{p(1)}=\tau_a)}\sum_{j=1}^r\un{(\tau_a=\hat\tau_j)}\tau_a\otimes(\tau_{p(2)}\otimes\ldots\otimes\tau_{p(k)})
			\\&=\sum_{p\in\cP(I_k,\bt)}\sum_{a=1}^k\frac{1}{\pi(\tau_a;w)}\un{(\tau_{p(1)}=\tau_a)}\tau_{p(1)}\otimes(\tau_{p(2)}\otimes\ldots\otimes\tau_{p(k)})\,.
		\end{align*}
		To conclude we observe that \eqref{eq: transposition} follows from using, for $p\in\cP(I_k,\bt)$ and $q\in\cP(I^i_k,\bt^i)$, the bijection defined by: $i:=p(1)$, \, $p(\,j+1)=q(\,j)$ for $0<j<i$,\,
		$p(\,j)=q(\,j)$ for $i<j\leq k$.
	\end{proof}
	
	\subsection{Useful bounds}
	We recall useful bounds, especially on weighted norms, that form the toolset for proving the a priori bounds on the solution of the rough equation, see Section \ref{Sec: well-posedness}.

	\begin{lemma}\label{Lem: bound weighted holder}
		Fix $f\in C([0,T],\R^n)$ and $F \in C ([0, T]^2_{\leq}, \mathbb{R}^n)$. For all $\alpha_1,
		\alpha_2 > 0$ and $\beta\in(0,1]$ we have
		\begin{align}
			&\| F \|_{\alpha_1, \mu} \leq (\mu \wedge T)^{\alpha_2} \| F \|_{\alpha_1 +
				\alpha_2, \mu}\,, \label{eta1}
			\\&\|F^\beta\|_{\alpha_1,\mu}\leq \|F\|_{\frac{\alpha_1}{\beta},\mu}^{\beta}\,,\label{Eq: weighted norm power}
			\\&\| f \|_{\infty}  \leq  | f_0 | + T^{\beta} \| \delta f \|_{\beta}\,, 
			\label{suphold}
			\\&\| f \|_{\infty, \mu} \leq | f_0 | + 3 (\mu \wedge T)^{\beta} \|
			\delta f \|_{\beta, \mu}\, .  \label{supholdweighted}
		\end{align}
	\end{lemma}
	\begin{proof}
		The continuous embedding of weighted norms represented by \eqref{eta1} follows from the chain of inequalities:
		\begin{align*}
			\| F \|_{\alpha_1, \mu} & = \sup_{0 \leq s \leq t \leq T}
			\mathbbm{1}_{\{ 0 < t - s \leq \mu \}} e^{- \frac{t}{\mu}}
			\frac{| F_{s t} |}{(t - s)^{\alpha_1}}\leq \sup_{0 \leq s \leq t \leq T}
			\mathbbm{1}_{\{ 0 < t - s \leq \mu \}} e^{- \frac{t}{\mu}} (\mu
			\wedge T)^{\alpha_2} \frac{| F_{s t} |}{(t - s)^{\alpha_1 + \alpha_2}}\\
			& = (\mu \wedge T)^{\alpha_2} \| F \|_{\alpha_1 + \alpha_2, \mu}\,.
		\end{align*}
		In a similar way, \eqref{Eq: weighted norm power} follows from a manipulation of the definition of weighted seminorm:
		\begin{align*}
			\|F^\beta\|_{\alpha_!,\mu}&=\sup_{0 \leq s \leq t \leq T}
			\mathbbm{1}_{\{ 0 < t - s \leq \mu \}} e^{- \frac{t}{\mu}}
			\frac{| F_{s t}^\beta |}{(t - s)^{\alpha_1}}\leq \sup_{0 \leq s \leq t \leq T}
			\mathbbm{1}_{\{ 0 < t - s \leq \mu \}} e^{-\frac{\beta t}{\mu}}
			\left(\frac{| F_{s t} |}{(t - s)^{\alpha_1/\beta}}\right)^\beta \\
			& \leq \sup_{0 \leq s \leq t \leq T}\left(\mathbbm{1}_{\{ 0 < t - s \leq \mu \}}e^{- \frac{t}{\mu}}\frac{| F_{s t} |}{(t - s)^{\alpha_1/\beta}}\right)^\beta\leq \|F\|_{\frac{\alpha_1}{\beta},\mu}^{\beta}\,.
		\end{align*}
		For any $f \in C ([0, T] _{\leq}, \mathbb{R}^n)$ and for fixed $t \in (0,
		T]$ we can write
		\[ | f_t | \leq | f_0 | + | f_t - f_0 | = | f_0 | + t^{\beta} \frac{|
			f_t - f_0 |}{t^{\beta}} \leq | f_0 | + T^{\beta} \| \delta f
		\|_{\beta} . \]
		When taking the supremum over $[0,T]$, this coincides with \eqref{suphold}. For what concerns \eqref{supholdweighted} we consider two
		cases: if $t \in (0, \mu \wedge T]$
		we have
		\[ e^{- \frac{t}{\mu}} | f_t | \leq | f_0 | + t^{\beta} e^{-
			\frac{t}{\mu}} \frac{| f_t - f_0 |}{t^{\beta}} \leq | f_0 | +
		(\mu \wedge T)^{\beta} \| \delta f \|_{\beta, \mu}, \]
		which implies (\ref{supholdweighted}). Instead, if $\mu < t \leq T$ we define $N
		:=\min \{ n \in \mathbb{N}: n \mu \geqslant t \} \geqslant 2$, so
		that $\frac{t}{N} \leq \mu$. For $i \geqslant 0$ set $t_i = i
		\frac{t}{N}$, in such a way that $t_N = t$. We get
		\begin{align*}
			e^{- \frac{t}{\mu}} | f_t | &\leq  | f_0 | + \sum_{i = 1}^N (t_i
			- t_{i - 1})^{\beta} e^{- \frac{t - t_i}{\mu}} \left[ e^{-
				\frac{t_i}{\mu}} \frac{| f_{t_i} - f_{t_{i - 1}} |}{(t_i - t_{i -
					1})^{\beta}} \right]\\ & \leq | f_0 | + (\mu \wedge T)^{\beta} \| \delta f \|_{\beta,
				\mu} \sum_{i = 1}^N e^{- \frac{t - t_i}{\mu}}
			 = | f_0 | + (\mu \wedge T)^{\beta} \| \delta f \|_{\beta, \mu}
			\sum_{i = 1}^N e^{- (N - i) \frac{t}{N \mu}}\,.
		\end{align*}
		Since $(N - 1) \mu < t$ and $\mu < t$, then $N \mu < 2 t$ or, equivalently $\frac{t}{N \mu} \geqslant \frac{1}{2}$. As a result
		\[ \sum_{i = 1}^N e^{- (N - i) \frac{t}{N \mu}} = \sum_{(N - i) = 0}^{N -
			1} e^{- (N - i) \frac{t}{N \mu}} = \frac{1 - e^{- \frac{t}{\mu}}}{1 -
			e^{- \frac{t}{N \mu}}} \leq \frac{1}{1 - e^{- \frac{1}{2}}}
		\leq 3\,. \]
		This concludes the proof.
	\end{proof}
	
	When handling the expansions over trees at the heart of the solution theory \emph{à la} Davie, one has to deal multiple times with norms of products of functions and increments. The following bounds show how to deal with the weight when splitting the norm over single components of the product.
	\begin{lemma}\label{Lem: Weighted bounds}
		Let $\beta, \beta_1,\beta_2, \mu \in (0, \infty)$ .
		\begin{align}
			&\text{If} \quad F_{s t} = g_s H_{s t} \quad \text{or} \quad F_{s t} = g_t
			H_{s t} \quad \text{then} \quad \| F \|_{\beta, \mu} \leq \| g
			\|_{\infty, \mu} \| H \|_{\beta}\,.  \label{2factors}\\
			&\text{If} \quad G_{s u t} = F_{s u} H_{u t} \quad  \text{then} \quad  \| G
			\|_{\beta_1+\beta_2, \mu} \leq \| F \|_{\beta_1, \mu} \| H \|_{\beta_2}\,. \nonumber 
		\end{align}
	\end{lemma}
	\begin{proof}
		If $F_{s t} = g_t H_{s t}$ we write
		\begin{align*}
			\| F \|_{\beta, \mu} & =  \sup_{0 \leq s \leq t \leq T}
			\mathbbm{1}_{\{ 0 < t - s \leq \mu \}} e^{- \frac{t}{\mu}}
			\frac{| F_{s t} |}{(t - s)^{\beta}} = \sup_{0 \leq s \leq t \leq T} \mathbbm{1}_{\{ 0 <
				t - s \leq \mu \}} e^{- \frac{t}{\mu}} \frac{| g_t H_{s t} |}{(t
				- s)^{\beta}}\\
			& \leq  \| g \|_{\infty, \mu} \sup_{0 \leq s \leq t
				\leq T} \frac{| H_{s t} |}{(t - s)^{\beta}} = \| g \|_{\infty, \mu} \| H \|_{\beta}\,.
		\end{align*}
		For $F_{s t} = g_s H_{s t}$ the desired bound follows from the previous calculation by observing that $e^{- t / \mu} \leq e^{- s / \mu}$ for all $s \leq t$. If $G_{s u t} = F_{s u} H_{u t}$, its seminorm reads
		\begin{align*}
			\| G \|_{\beta_1+\beta_2, \mu} & = \sup_{0 \leq s \leq u \leq t
				\leq T} \mathbbm{1}_{\{ 0 < t - s \leq \mu \}} e^{-
				\frac{t}{\mu}}  \frac{| F_{s u} H_{u t} |}{(t - s)^{\beta_1+\beta_2}}\\
			&\leq \sup_{0 \leq s \leq u \leq t \leq T}
			\mathbbm{1}_{\{ 0 < t - s \leq \mu \}} e^{- \frac{s}{\mu}} 
			\frac{| F_{s u} | | H_{u t} |}{(u - s)^{\beta_1} (t - u)^{\beta_2}} \leq \| F \|_{\beta_1, \mu} \| H \|_{\beta_2} .
		\end{align*}
	\end{proof}

	\endappendix
	
	\bibliographystyle{alpha}
	
	\bibliography{bibliography}

\end{document}